\documentclass[12pt,reqno]{amsart}
\usepackage{amsmath,amssymb,amsthm,mathtools,wasysym,calc,verbatim,enumitem,tikz,pgfplots,url,hyperref,mathrsfs,cite,fullpage,bbm}

\usepackage{comment}

\newtheorem{theorem}{Theorem}[section]
\newtheorem{lemma}[theorem]{Lemma}
\newtheorem{corollary}[theorem]{Corollary}
\newtheorem{prop}[theorem]{Proposition}

\newtheorem{definition}[theorem]{Definition}

\newtheorem{claim}[theorem]{Claim}
\newtheorem{fact}[theorem]{Fact}

\theoremstyle{definition}

\theoremstyle{remark}
\newtheorem{remark}[theorem]{Remark}

\newcommand\N{\mathbb{N}}
\newcommand\R{\mathbb{R}}
\newcommand\Z{\mathbb{Z}}
\newcommand\cA{\mathcal{A}}
\newcommand\cB{\mathcal{B}}
\newcommand\cN{\mathcal{N}}
\newcommand\cP{\mathcal{P}}

\newcommand\cW{\mathcal{W}}
\newcommand\cQ{\mathcal{Q}}
\newcommand\cU{\mathcal{U}}
\newcommand\cE{\mathcal{E}}
\newcommand\cF{\mathcal{F}}
\newcommand\cM{\mathcal{M}}
\newcommand\cK{\mathcal{K}}

\newcommand\cT{\mathcal{T}}
\newcommand{\zb}{\bar{\zeta}}
\newcommand{\dist}{\mathrm{dist}}
\newcommand{\ls}{\lesssim}
\newcommand{\Span}{\mathrm{Span}}

\def\Pr{\mathbb{P}}
\def\S{\mathcal{S}}
\newcommand\Ex{\mathbb{E}}

\newcommand{\one}{\mathbf{1}}
\newcommand\eps{\varepsilon}
\renewcommand{\P}{\mathbb{P}}
\renewcommand{\leq}{\leqslant}
\renewcommand{\geq}{\geqslant}
\renewcommand{\le}{\leqslant}

\renewcommand{\to}{\rightarrow}

\usepackage{appendix}

\def\eps{\varepsilon}

\def\E{\mathbb{E}}

\def\R{\mathbb{R}}
	
\def\Z{\mathbb{Z}}
\def\N{\mathbb{N}}
\def\PP{\mathbb{P}}
\def\S{\mathbb{S}}
\def\1{\mathbf{1}}
\def\l{\lambda}
\def\k{\kappa}

\def\s{\sigma}
\def\t{\theta}

\def\g{\gamma}
\def\z{\zeta}
\def\G{\Gamma}	

\def\la{\langle}
\def\ra{\rangle}	

\def\Xt{\widetilde{X}}
\def\Yt{\widetilde{Y}}

\def\Span{\mathrm{Span\,}}

\def\Inc{\mathrm{Incomp\,}}
\def\Comp{\mathrm{Comp\,}}
\def\Sym{\mathrm{Sym\,}}
\def\col{\mathrm{Col\,}}

\def\EE{\mathbb{E}}

\def\T{\mathbb{T}}
\def\cA{\mathcal{A}}
\def\cQ{\mathcal{Q}}	
\def\cC{\mathcal{C}}
\def\cL{\mathcal{L}}
\def\cI{\mathcal{I}}

\def\L{\Lambda}

\def\vp{\varphi}
\newcommand{\tz}{\tilde{\zeta}}

\newcommand{\HS}{\mathrm{HS}}
\newcommand{\Dhat}{\hat{D}}

\begin{document}
	\author{Marcelo Campos, Matthew Jenssen, Marcus Michelen, Julian Sahasrabudhe}
	
	\title{The least singular value of a random symmetric matrix}

\begin{abstract}
	 Let $A$ be a $n \times n$ symmetric matrix with $(A_{i,j})_{i\leq j}$ independent and identically distributed according to a subgaussian distribution.
	  We show that 
	 $$\P(\sigma_{\min}(A) \leq \eps n^{-1/2} ) \leq C \eps +  e^{-cn},$$
	where $\sigma_{\min}(A)$ denotes the least singular value of $A$ and the constants $C,c>0 $ depend only on the distribution of the entries of $A$. This result
	confirms the folklore conjecture on the lower tail of the least singular value of such matrices and is best possible up to the dependence of the constants on the distribution of $A_{i,j}$. Along the way, we prove that the probability that $A$ has a repeated eigenvalue is $e^{-\Omega(n)}$, thus confirming a conjecture of Nguyen, Tao and Vu.
	\end{abstract}
	
	\address{Instituto de Matem\'atica Pura e Aplicada (IMPA). }
	\email{marcelo.campos@impa.br}
	\address{King's College London. Department of Mathematics.}
	\email{matthew.jenssen@kcl.ac.uk}
	\address{University of Illinois Chicago. Department of Mathematics, Statistics and Computer Science.}
	\email{michelen.math@gmail.com}
	\address{University of Cambridge. Department of Pure Mathematics and Mathematics Statistics.} 
	\email{jdrs2@cam.ac.uk}\maketitle

	\maketitle
	\vspace{-4mm}	
\section{Introduction}
	
Let $A$ be a $n\times n$ random symmetric matrix whose entries on and above the diagonal $(A_{i,j})_{i\leq j}$ are i.i.d.\ with mean $0$ and variance $1$.
This matrix model, sometimes called the Wigner matrix ensemble, was introduced in the 1950s in the seminal work of Wigner \cite{Wigner}, who established the 
famous ``semi-circular law'' for the eigenvalues of such matrices.

In this paper we study the extreme behavior of the \emph{least singular value} of $A$, which we denote by $\s_{\min}(A)$.
Heuristically, we expect that $\s_{\min}(A) = \Theta(n^{-1/2})$  and thus it is natural to consider 
\begin{equation}\label{eq:main-quantity} \PP( \s_{\min}(A)  \leq \eps n^{-1/2} ), \end{equation}
for all $\eps \geq 0$ (see Section \ref{subsec:history}). 
In this paper we prove a bound on this quantity which is optimal up to constants, for all random symmetric matrices with  i.i.d.\ \emph{subgaussian} entries. 
This confirms the folklore conjecture, explicitly stated by Vershynin in \cite{vershynin-invertibility}. 

		\begin{theorem}\label{thm:main}
		Let $\zeta$ be a subgaussian random variable with mean $0$ and variance $1$ and let $A$ be a $n \times n$ random symmetric matrix whose entries above the diagonal $(A_{i,j})_{i\leq j}$ are independent and distributed according to $\zeta$.  Then for every $\eps\geq 0$,
		\begin{equation} \label{eq:thmMain}\P_A(\sigma_{\min}(A) \leq \eps n^{-1/2}) \leq C \eps + e^{-cn}, \end{equation}
		where $C,c>0$ depend only on $\zeta$.
	\end{theorem}
		
This conjecture is sharp up to the value of the constants $C,c>0$ and resolves the ``up-to-constants'' analogue of the Spielman--Teng conjecture for 
random symmetric matrices (see Section~\ref{subsec:history}). Also note that the special case $\eps = 0 $ tells us that the singularity probability of any random symmetric $A$
with subgaussian entry distribution is exponentially small, generalizing our previous work \cite{RSM2} on the $\{-1,1\}$ case.

\subsection{Repeated eigenvalues} Before we discuss the history of the least singular value problem, we highlight one further contribution of this paper:
a proof that a random symmetric matrix has no repeated eigenvalues with probability $1-e^{-\Omega(n)}$.

In the 1980s Babai conjectured that the adjacency matrix of the binomial random graph $G(n,1/2)$ has no repeated eigenvalues 
with probability $1-o(1)$ (see \cite{tao2017random}). Tao and Vu \cite{tao2017random} proved this conjecture in 2014 and, in subsequent work on the topic with Nguyen \cite{nguyen-tao-vu-repulsion}, went on to conjecture the probability that a random symmetric matrix with i.i.d.\ subgaussian entries has no repeated eigenvalues is $1-e^{-\Omega(n)}$. 
In this paper we prove this conjecture en route to proving Theorem~\ref{thm:main}, our main theorem.

\begin{theorem}\label{thm:simple-spectrum}
Let $\zeta$ be a subgaussian random variable with mean $0$ and variance $1$ and let $A$ be a $n \times n$ random symmetric matrix where $(A_{i,j})_{i\leq j}$ are independent and distributed according to $\zeta$. Then $A$ has no repeated eigenvalues with probability at least $1-e^{-cn}$, where $c>0$ is a constant depending only on $\zeta$.
\end{theorem}

Theorem~\ref{thm:simple-spectrum} is easily seen to be sharp whenever $A_{i,j}$ is discrete: consider the event that three rows 
of $A$ are identical; this event has probability $e^{-\Theta(n)}$ and results in two $0$ eigenvalues. Also note that the constant in Theorem~\ref{thm:simple-spectrum}
can be made arbitrary small; consider the entry distribution $\z$ which takes value $0$ with probability $1-p$ and each of $\{-p^{-1/2},p^{-1/2}\}$ with probability $p/2$. Here the 
probability of $0$ being a repeated root is $\geq e^{-(3+o(1))pn}$.

We in fact prove a more refined version Theorem~\ref{thm:simple-spectrum} which gives an upper bound on the probability that two eigenvalues of $A$ fall into an interval of length $\eps$.
This is the main result of Section~\ref{sec:evalcrowding}. For this, we let $\lambda_1(A)\geq\ldots\geq \lambda_n(A)$ denote the eigenvalues of the $n\times n$ real symmetric matrix $A$.

	\begin{theorem}\label{th:repulsion}Let $\zeta$ be a subgaussian random variable with mean $0$ and variance $1$ and let $A$ be a $n \times n$ random symmetric matrix where $(A_{i,j})_{i\leq j}$ are independent and distributed according to $\zeta$. Then for each $\ell < cn$ and all $\eps \geq 0$ we have 
		$$\max_{k \leq n-\ell} \, \P\big( |\lambda_{k+\ell}(A) - \lambda_{k}(A)| \leq \eps n^{-1/2}   \big) \leq \left(C\eps \right)^{\ell} + 2e^{-cn} \, ,$$
		where $C,c>0$ are constants, depending only on $\zeta$.
	\end{theorem}

In the following subsection we describe the history of the least singular value problem.  In Section~\ref{ss:approx-neg-cor}, we discuss a technical theme which is developed in this paper and then, in Section~\ref{sec:sketch}, we go on to give a sketch of Theorem~\ref{thm:main}.

\subsection{History of the least singular value problem} \label{subsec:history}

The behavior of the least singular value was first studied for random matrices $B_n$ with \emph{i.i.d.}\ coefficients,
rather than for \emph{symmetric} random matrices. For this model, the history goes back to von Neumann~\cite{vonNeumann} who suggested that one typically has
\[ \s_{\min}(B_n) \approx n^{-1/2},\] while studying approximate solutions to linear systems. 
This was then more rigorously conjectured by Smale \cite{smale} and proved by Szarek \cite{szarek} and Edelman \cite{edelman} in the case that $B_n = G_n$ is a random matrix with i.i.d.\ \emph{standard gaussian} entries. 
Edelman found an exact expression for the density of the least singular value in this case. By analysing this expression, one can deduce  that 
\begin{equation}\label{eq:edelman}
 \PP( \s_{\min}(G_n)  \leq \eps n^{-1/2} ) \leq \eps, 
\end{equation}
for all $\eps \geq 0$ (see e.g.\ \cite{spielman-teng-ICM}).
	While this gives a very satisfying understanding of the gaussian case, one encounters serious difficulties when trying to extend this result to other distributions. Indeed Edelman's proof relies crucially on an exact description of the joint distribution of eigenvalues that is available in the gaussian setting. In the last 20 or so years, intense study of the least singular value of i.i.d.\ random matrices has been undertaken with the overall goal of proving an appropriate version of \eqref{eq:edelman} for different entry distributions and models of random matrices.

An important and challenging feature of the more general problem arises in the case of \emph{discrete} distributions, where the matrix $B_n$ can become singular with non-zero probability. This singularity event will affect the quantity \eqref{eq:main-quantity} for very small $\eps$ and thus estimating the probability that $\s_{\min}(B_n) = 0$ is a crucial aspect of generalizing \eqref{eq:edelman}. This is reflected in the famous and influential Spielman--Teng conjecture \cite{spielman-teng-simplex} which proposes the bound 
\begin{equation}\label{eq:spielman-teng}
 \PP( \s_{\min}(B_n)  \leq \eps n^{-1/2} ) \leq \eps + 2e^{-cn},
\end{equation}
where $B_n$ is a Bernoulli random matrix. Here this added exponential term ``comes from'' the singularity probability of $B_n$.

In this direction, a key breakthrough was made by Rudelson \cite{rudelson-annals} who proved that if $B_n$ has
 i.i.d.\ \emph{subgaussian} entries then $$\P(\sigma_{\min}(B_n) \leq \eps n^{-1/2} ) \leq C \eps n  + n^{-1/2}\,.$$
This result was extended in a series of works \cite{tao-vu-ILO, RV-rectangle,taovu-general-least-sing-value,tao-vu-condition-number} culminating in the influential work of Rudelson and Vershynin \cite{RV} who showed 
the ``up-to-constants'' version of Spielman-Teng:
\begin{equation}\label{eq:RVapproxST}
 \PP( \s_{\min}(B_n)  \leq \eps n^{-1/2} ) \leq C\eps + e^{-cn},
\end{equation}
where $B_n$ is a matrix with i.i.d.\ entries that follow any subgaussian distribution and $C,c>0$ depend only on $\zeta$. 
A key ingredient in the proof of~\eqref{eq:RVapproxST} is a novel approach to the ``inverse Littlewood-Offord problem,'' a perspective pioneered by Tao and Vu~\cite{tao-vu-ILO} (see Section~\ref{ss:approx-neg-cor} for more discussion). 

Another very different approach was taken by Tao and Vu \cite{taoVu-universality-smallest} who showed that the distribution of the least singular value of $B_n$ is identical to the least singular value of the Gaussian matrix $G_n$, up to scales of size $n^{-c}$. In particular they prove that 
\begin{equation} \label{eq:tao-vu-universal} \big| \PP( \s_{\min}(B_n)  \leq \eps n^{-1/2} ) - \PP( \s_{\min}(G_n) \leq \eps n^{-1/2}) \big| = O(n^{-c_0}), \end{equation}
thus resolving the Speilman-Teng conjecture for $\eps \geq n^{-c_0}$, in a rather strong form. 

While falling just short of the Spielman-Teng conjecture, the work of Tao and Vu~\cite{taoVu-universality-smallest}, Rudelson and Vershynin \cite{RV} and subsequent refinements by Tikhomirov \cite{tikhomirov-lsv} and Livshyts, Tikhomirov and Vershynin \cite{tikhomirov-lsv} (see also \cite{rebrova-lsv,livshyts-lsv}) leave us with a very strong understanding of the least singular value for \emph{i.i.d.}\ matrix models.
However, progress on the analogous problem for random symmetric matrices, or \emph{Wigner random matrices}, has come somewhat more slowly and more recently: in the symmetric case, even proving that $A_n$ is non-singular with probability $1-o(1)$ was not resolved until the important 2006 paper of Costello, Tao and Vu \cite{costello-tao-vu}.

Progress on the symmetric version of Spielman--Teng continued with Nguyen~\cite{nguyen-singularity,nguyen-least-singular-value} and, independently, Vershynin~\cite{vershynin-invertibility}. Nguyen proved that for any $B>0$ there exists an $A>0$ for which\footnote{Nguyen in \cite{nguyen-least-singular-value} actually proves the same result for random matrices of the form $A_n + F$, where 
$F$ is a fixed symmetric $n\times n$ matrix satisfying $\|F\|_{op} \leq n^{O(1)}$.}
\[
	\P(\sigma_{\min}(A_n) \leq n^{-A}) \leq n^{-B}.
\] Vershynin \cite{vershynin-invertibility} proved that if $A_n$ is a matrix with subgaussian entries then, for all $\eps>0$, we have 
\begin{equation}\label{eq:vershynin}
	\P(\sigma_{\min}(A_n) \leq \eps n^{-1/2}) \leq C_\eta\eps^{1/8 -\eta} + 2e^{-n^c},
\end{equation}
for all $\eta>0$, where the constants $C_\eta,c > 0$ may depend on the underlying subgaussian random variable. 
He went on to conjecture that $\eps$ should replace $\eps^{1/8 - \eps}$ as the correct order of magnitude and that $e^{-cn}$ should replace $e^{-n^{c}}$.

After Vershynin, a series of works \cite{ferber2019counting,ferber-jain,CMMM,JSS-symmetric,RSM1} made progress on singularity probability (i.e.\  the $\eps = 0$ case
of Vershynin's conjecture), and we, in \cite{RSM2}, ultimately showed that the singularity probability is exponentially small, when $A_{i,j}$ is uniform in $\{-1,1\}$:
\[ \PP( \det(A_n) = 0 ) \leq e^{-cn}, \]
which is sharp up to the value of $c>0$.

However, for general $\eps$ the state of the art is due to Jain, Sah and Sawhney \cite{JSS-symmetric}, who improved on Vershynin's bound~\eqref{eq:vershynin} by showing
  \[
	\P(\sigma_{\min}(A_n) \leq \eps  n^{-1/2} ) \leq C\eps^{1/8} + e^{-\Omega(n^{1/2})}\,,
\]	under the subgaussian hypothesis on $A_n$.

For large $\eps$, for example $\eps \geq n^{-c}$, another very different and powerful set of techniques have been developed, which in fact apply more generally to the distribution of other ``bulk'' eigenvalues and additionally give distributional information on the eigenvalues. The works of Tao and Vu \cite{taoVu-universality-Wigner, taovu-general-least-sing-value}, Erd\H{o}s, Schlein and Yau~\cite{ESY-AOP,ESY-CMP,ESY-IMRN}, Erd\H{o}s, Ram\'{i}rez, Schlein, Tao, Vu, Yau \cite{ERSTVY}, and specifically Bourgade, Erd\H{o}s, Yau and Yin \cite{BEYY} tell us that 
\begin{equation} \label{eq:BEYY}\PP( \sigma_{\min}(A_n) \leq \eps n^{-1/2} ) \leq \eps + o(1), \end{equation}
thus obtaining the correct dependence\footnote{Tao and Vu, with Corollary 24 in \cite{taoVu-universality-Wigner}, prove that the distribution of $\s_{\min}$ remains asymptotically invariant if the distribution of the entries $A_{i,j}$ is replaced by a distribution that matches four moments with the original distribution. A follow-up work \cite{ERSTVY} joint with Erd\H{o}s, Ram\'{i}rez, Schlein, and Yau describes an approach to combine ideas from the works \cite{ESY-AOP,ESY-CMP,ESY-IMRN} to remove the moment matching assumptions of \cite{taoVu-universality-Wigner}, but does not explicitly address the problem of the least singular value.  The work \cite{BEYY} builds on these works to prove the sharp, non-quantitative statement at \eqref{eq:BEYY}. See the discussion below Theorem 2.2 of \cite{BEYY} for more detail.} on $\eps$ when $n$ is sufficiently large compared to $\eps$. These results are similar in flavor to \eqref{eq:tao-vu-universal} in that they show 
the distribution of various eigenvalue statistics are closely approximated by the corresponding statistics in the gaussian case. We note however that it appears these techniques are limited to these large $\eps$ and different ideas are required for $\eps < n^{-C}$, and certainly for $\eps$ as small as $e^{-\Theta(n)}$.

Our main theorem, Theorem~\ref{thm:main}, proves Vershynin's conjecture and thus proves the optimal dependence on $\eps$ for all $\eps > e^{-cn}$, up to constants.

\subsection{Approximate negative correlation}\label{ss:approx-neg-cor}

Before we sketch the proof of Theorem~\ref{thm:main}, we highlight a technical theme of this paper:  the approximate negative correlation of certain ``linear events''. While this is only one of several new ingredients in this paper, we 
isolate these ideas here, as they seem to be particularly amenable to wider application. We refer the reader to Section~\ref{sec:sketch} for a more complete overview of the new ideas in this paper.

We say that two events $A,B$ in a probability space are \emph{negatively
correlated} if 


\[ \PP(A\cap B) \leq \P(A) \P(B).\]
Here we state and discuss two \emph{approximate} negative correlation results: one of which is from our paper \cite{RSM2}, but is used in an entirely different context, and one of which is new. 

We start by describing the latter result, which says that a ``small ball'' event is approximately negatively correlated with a 
large deviation event. This complements our result from \cite{RSM2} which says that two ``small ball events'', of different types, are negatively correlated. In particular, we prove something in the spirit of the following inequality, though in a slightly more technical form.
\begin{equation}\label{eq:truenegcore} \PP_X\big( |\la X, v \ra| \leq \eps  \text{ and }  \la  X, u \ra >t \big) \lesssim \PP_X(|\la X, v \ra| \leq \eps )\PP_X( \la X, u \ra >t ),  \end{equation}
where $u, v$ are unit vectors and $t,\eps >0$ and $X = (X_1,\ldots,X_n)$ with i.i.d. subgausian random variables with mean $0$ and variance $1$.

To state and understand our result, it makes sense to first
consider, in isolation, the two events present in \eqref{eq:truenegcore}. The easier of the two events is $\la X, u \ra >t$, which is a large deviation event for which we may apply the essentially sharp and classical inequality (see Chapter 3.4 in \cite{vershynin2018high})
\[ \PP_X(  \la  X, u \ra >t ) \leq e^{-ct^2},\]
where $c>0$ is a constant depending only on the distribution of $X$. 

We now turn to understand the more complicated small-ball event $|\la X , v \ra| \leq \eps$ appearing in \eqref{eq:truenegcore}.
Here, we have a more subtle interaction between $v$ and the distribution of $X$, and thus we first consider the simplest possible case: when $X$ has i.i.d.\ standard \emph{gaussian} entries. Here, one may calculate
\begin{equation} \label{eq:negcorGaussian} \P_X(|\langle X, v \rangle| \leq \eps) \leq  C\eps , \end{equation}
for all $\eps>0$, where $C>0$ is an absolute constant. However, as we depart from the case when $X$ is gaussian, a much richer behavior emerges when the
vector $v$ admits some ``arithmetic structure''. For example,
if $v = n^{-1/2}(1,\ldots,1)$ and the $X_i$ are uniform in $\{-1,1\}$ then 
\[ \PP_X( |\la X, v \ra| \leq \eps ) = \Theta(n^{-1/2}),\] for any $0< \eps < n^{-1/2}$. This, of course, stands in contrast to \eqref{eq:negcorGaussian} for all $\eps \ll n^{-1/2}$ and suggests that we employ an appropriate measure of the arithmetic structure of $v$.

For this, we use the notion of the ``least common denominator'' of a vector, introduced by Rudelson and Vershynin \cite{RV}. For parameters $\alpha,\gamma \in (0,1)$ define the \emph{Least Common Denominator} (LCD) of $v \in \R^n$ to be
\begin{equation} \label{eq:lcd-def} 
	D_{\alpha,\gamma}(v):=\inf\bigg\{\phi>0:~\|\phi v\|_{\T}\leq \min\left\{\gamma\phi\|v\|_2, \sqrt{\alpha n}\right\}\bigg\},
\end{equation}
where $ \| v \|_{\T} : = \dist(v,\Z^n)$, for all $v \in \R^n$. What makes this definition useful is the important ``inverse Littlewood-Offord theorem'' of Rudelson and Vershynin \cite{RV}, which tells us (roughly speaking) that one has \eqref{eq:negcorGaussian} whenever $D_{\alpha,\g}(v) = \Omega(\eps^{-1})$.  

This notion of Least Common Denomonator is inspired by Tao and Vu's introduction and development of ``inverse Littlewood-Offord theory'', which is a 
collection of results guided by the meta-hypothesis: ``If $\PP_X( \la X,v\ra = 0 )$ is large \emph{then} $v$ must have structure''. We refer the reader to the paper of Tao and Vu~\cite{tao-vu-ILO} and the survery of Nguyen and Vu~\cite{nguyen-vu-2} for more background and history on inverse Littlewood-Offord theory and its role in random matrix theory. 

We may now state our version of \eqref{eq:truenegcore}, which uses $D_{\alpha,\g}(v)^{-1}$ as a proxy for $\PP(|\la X, v \ra| \leq \eps)$.

\begin{theorem}\label{thm:neg-dependence-2}
For $n \in \N$, $\eps,t >0$ and $\alpha,\gamma \in (0,1)$, let $v \in \S^{n-1}$ satisfy $D_{\alpha,\gamma}(v)> C/\eps$ and let $u \in \S^{n-1}$. 
Let $\z$ be a subgaussian random variable and let $X \in \R^n$ be a random vector whose coordinates are i.i.d.\ copies of $\z$. Then
\[ \PP_X\left( |\la X,v \ra| \leq \eps \text{ and } \la X, u \ra > t \right) \leq  C \eps e^{-ct^2}   +  e^{-c(\alpha n + t^2)}, \]
where $C,c>0$ depend only on $\g$ and the distribution of $\z$.
\end{theorem}

In fact, we need a significantly more complicated version of this result (Lemma~\ref{lem:small-ball-LD}) where the small-ball event $|\la X,v\ra| \leq \eps$ is replaced with a
small-ball event of the form 
\[ |f(X_1,\ldots,X_n)| \leq \eps, \]
where $f$ is a quadratic polynomial in variables $X_1,\ldots ,X_n$. The proof of this result is carried out in Section~\ref{sec:decouplingQForms} and is an important aspect of this paper. Theorem~\ref{thm:neg-dependence-2} is stated here to illustrate the general flavor of this result and is not actually used in this paper. We do provide a proof in Appendix~\ref{app:neg-dep} for completeness and to suggest further inquiry into inequalities of the form \eqref{eq:truenegcore}.

We now turn to discuss our second approximate negative dependence result, which deals with the intersection of two different small ball events. This was originally proved in our paper \cite{RSM2}, but is put to a different use here. This result tells us that  the events
\begin{equation} \label{eq:intro-two-events} |\la X, v\ra| \leq \eps \qquad \text{    and    } \qquad |\la X, w_1 \ra| \ll 1 , \ldots , |\la X, w_k \ra| \ll 1,
\end{equation}
are approximately negatively correlated, where $X = (X_1,\ldots,X_n)$ is a vector with i.i.d.\ subgaussian entries and $w_1,\ldots,w_k$ are orthonormal. That is, we prove something in the spirit of 
\[
\PP_X\bigg(\{ |\la X, v\ra| \leq \eps \} \cap \bigcap_{i=1}^k \{ |\la X,  w_i \ra| \ll 1 \}\bigg) 
\lesssim \PP_X\big( |\la X, v \ra| \leq \eps  \big)\PP_X\bigg( \bigcap_{i=1}^k \{ |\la X, w_i \ra| \ll 1 \}\bigg),\]
though in a more technical form.   

To understand our result, again it makes sense to consider the two events  in \eqref{eq:intro-two-events} in isolation. Since we have already discussed 
the subtle event $|\la X, v \ra| \leq \eps$, we consider the event on the right of \eqref{eq:intro-two-events}. Returning to the gaussian case, we note that if $X$ has independent standard gaussian entries, then one may compute directly that 
\begin{equation}\label{eq:negCorHW}  \P_X\left(|\la X, w_1 \ra| \ll 1 , \ldots , |\la X, w_k \ra| \ll 1\right) = \PP( |X_1| \ll 1,\ldots |X_k| \ll 1 ) \leq e^{-\Omega(k)}\, ,\end{equation}
by rotational invariance of the gaussian.
Here the generalization to other random variables is not as subtle, and the well-known Hanson-Wright inequality tells us that~\eqref{eq:negCorHW} holds more generally when $X$ has general i.i.d.\ subgaussian entries.

Our innovation in this line is our second ``approximate negative correlation theorem,'' which allows us to control these two events \emph{simultaneously}.
Again we use $D_{\alpha,\g}(v)^{-1}$ as a proxy for $\PP(|\la X,v \ra| \leq \eps)$. 

Here, for ease of exposition, we state
a less general version for $X = (X_1,\ldots,X_n) \in \{-1,0,1\}$ with i.i.d. $c$-lazy coordinates, meaning that $\PP(X_i = 0) \geq 1-c$.
Our theorem is stated in full generality in Section~\ref{sec:small-ball-large-dev}, see Theorem~\ref{thm:ILwO2}.

\begin{theorem}\label{thm:invLwO} Let $\g \in (0,1)$,  $d \in \N$, $\alpha \in (0,1)$, $0\leq k \leq c_1 \alpha d$ and $\eps \geq \exp(-c_1\alpha d)$. Let $v \in \S^{d-1}$, let $w_1,\ldots,w_k \in \S^{d-1}$ be orthogonal and let $W$ be the matrix with rows $w_1,\ldots,w_k$.

If $X \in \{-1,0,1 \}^d$ is a $1/4$-lazy random vector and $D_{\alpha,\g}(v) > 16/\eps$ then
\begin{equation*} \PP_X\left( |\la X, v \ra| \leq \eps\, \text{ and }\, \|W X \|_2 \leq c_2\sqrt{k} \right) \leq C \eps e^{- c_1 k},\end{equation*} where $C,c_1,c_2 >0$ are constants, depending only on $\g$.
\end{theorem}

In this paper we will put Theorem~\ref{thm:invLwO} to a very different use to that in \cite{RSM2}, where we used it to prove a version of the following statement.

\begin{center}
Let $v \in \S^{d-1}$ be a vector on the sphere and let $H$ be a $n  \times d$ random $\{-1,0,1\}$-matrix \emph{conditioned on the event} $\|Hv\|_2 \leq  \eps n^{1/2}$, for some $\eps > e^{-cn}$. Here $d = cn$ and $c>0$ is a sufficiently small constant.
Then the probability that the rank of $H$ is $n-k$ is $\leq e^{-ckn}$.
\end{center}

In this paper we use (the generalization of) Theorem~\ref{thm:invLwO} to obtain good bounds on quantities of the form
\[ \PP_X( \|BX\|_2 \leq \eps n^{1/2} ), \]
where $B$ is a fixed matrix with an exceptionally large eigenvalue (possibly as large as $e^{cn}$), but is otherwise pseudo-random, meaning (among other things) that the rest of the spectrum does not deviate too much 
from that of a random matrix. We use Theorem~\ref{thm:invLwO} to decouple the interaction of $X$ with the largest eigenvector of $B$, from the interaction of $X$ with the rest of $B$. We refer the reader to \eqref{eq:neg-dependence} in the sketch
in Section~\ref{sec:sketch} and to Section~\ref{sec:small-ball-large-dev} for more details. 

The proof of Theorem~\ref{thm:ILwO2} follows closely along the lines of the proof of Theorem~\ref{thm:invLwO} from~\cite{RSM2}, requiring only technical modifications and adjustments. So as not to distract from the new ideas of this paper, we have sidelined this proof to the Appendix.

Finally we note that it may be interesting to investigate these approximate negative correlation results in their own right, and investigate to what extent they can be sharpened.

\section{Proof sketch} \label{sec:sketch}
Here we sketch the proof of Theorem~\ref{thm:main}. We begin by giving the rough ``shape'' of the proof, while making a few simplifying assumptions,
\eqref{eq:sketch-ass1} and \eqref{eq:sketch-ass2}. We shall then come to discuss the substantial new ideas of this paper in Section~\ref{subsec:removing-ass} where we describe the considerable lengths we must go to in order to remove our simplifying assumptions. Indeed, if one were to only tackle these assumptions using standard tools, one cannot hope for a bound much better than $\eps^{1/3}$ in Theorem~\ref{thm:main} (see Section \ref{sss:base}).  

\subsection{The shape of the proof}\label{subsec:sketch-sketch}
Recall that $A_{n+1}$ is a $(n+1)\times (n+1)$ random symmetric matrix with subgaussian entries.  Let $X := X_1,\ldots,X_{n+1}$ be the columns of $A_{n+1}$, let 
\[ V = \Span\{ X_2,\ldots,X_{n+1}\}\] and let $A_n$ be the matrix $A_{n+1}$ with the first row and column removed.
We now use an important observation from Rudelson and Vershynin \cite{RV} that allows for a geometric perspective on the least singular value problem\footnote{Here and throughout we understand $A \ls B$ to mean that there exists an absolute constant $C>0$ for which $A \leq CB$.} 
\[ \PP( \sigma_{\min}(A_{n+1}) \leq \eps n^{-1/2} ) \ls \PP( \dist(X,V) \leq \eps ) . \]

Here our first significant challenge presents itself: $X$ and $V$ are not independent and thus the event $\dist(X,V) \leq \eps $ is hard to understand directly. However, one can establish a formula for $\dist(X,V)$ that is a rational function in the vector $X$ with coefficients 
that depend only on $V$. This brings us to the useful inequality\footnote{In this sketch we will be ignoring a few exponentially rare events, and so the inequalities listed here should be understood as ``up to an additive  error of $e^{-cn}$.''} due to Vershynin \cite{vershynin-invertibility},
\begin{equation}\label{eq:sketch-1} \PP( \sigma_{\min}(A_{n+1}) \leq \eps n^{-1/2} ) 
\ls \sup_{r \in \R} \PP_{A_n,X}\big( |\la A_n^{-1}X, X \ra - r| \leq  \eps \|A_n^{-1}X\|_2 \big) ,\end{equation}
where we are ignoring the possibility of $A_n$ being singular for now. We thus arrive at our main technical focus of this paper, bounding the quantity on the right-hand-side of \eqref{eq:sketch-1}.

We now make our two simplifying assumptions that shall allow us to give the overall shape of our proof
without any added complexity. We shall then layer-on further complexities
as we discuss how to remove these assumptions. 

As a first simplifying assumption, let us assume that the collection of $X$ that dominates the probability at \eqref{eq:sketch-1} satisfies
\begin{equation} \label{eq:sketch-ass1} \|A_n^{-1}X\|_2 \approx \|A_n^{-1}\|_{\HS}. \end{equation}
This is not, at first blush, an unreasonable assumption to make as $\EE_X\, \|A_n^{-1}X\|_2^2 = \|A_n^{-1}\|_{\HS}^2 $.  Indeed, the Hanson-Wright inequality tells us that $ \|A_n^{-1}X\|_2 $ is concentrated about its mean, for all reasonable $A_n^{-1}$. However,
as we will see, this concentration is not strong enough for us here.

As a second assumption, we assume that the relevant matrices $A_n$ in the right-hand-side of \eqref{eq:sketch-1} satisfy
\begin{equation}\label{eq:sketch-ass2} \|A_n^{-1}\|_{\HS} \approx cn^{1/2}. \end{equation}
This turns out to be a \emph{very} delicate assumption, as we will soon see, but is not entirely unreasonable to make for the moment: for example 
we have $\|A_n^{-1}\|_{\HS} =  \Theta_{\delta}(n^{1/2})$ with probability $1-\delta$. This, for example, follows from Vershynin's theorem \cite{vershynin-invertibility}
along with Corollary~\ref{cl:ESY-consequence}, which is based on the work of \cite{ESY-IMRN}.

With these assumptions, we return to \eqref{eq:sketch-1} and obverse our task has reduced to proving 
\begin{equation} \label{eq:sketch-goal2} \min_r \PP_{X}\big( |\la A^{-1}X, X \ra - r| \leq  \eps n^{1/2} \big) \ls \eps , \end{equation}
for all $\eps > e^{-cn}$, where we have written $A^{-1} = A_{n}^{-1}$ and think of $A^{-1}$ as a fixed (pseudo-random) matrix. 

We observe, for a general fixed matrix $A^{-1}$ there is no hope in proving such an inequality: indeed if $A^{-1} = n^{-1/2}J$, where $J$ is the all-ones matrix, then the left-hand-side of \eqref{eq:sketch-goal2} is $\geq cn^{-1/2}$ for \emph{all} $\eps >0$, falling vastly short of our desired \eqref{eq:sketch-goal2}.

Thus, we need to introduce a collection of fairly strong ``quasi-randomness properties'' of $A$ that hold with probably $1-e^{-cn}$. These will ensure that $A^{-1}$
is sufficiently ``non-structured'' to make our goal \eqref{eq:sketch-goal2} possible. The most important and difficult of these quasi-randomness conditions 
is to show that the eigenvectors $v$ of $A$ satisfy
\[D_{\alpha,\gamma}(v) > e^{cn}, \]
for some appropriate $\alpha,\g$, where $D_{\alpha,\g}(v)$ is the \emph{least common denominator} of $v$ defined at \eqref{eq:lcd-def}.
Roughly this means that none of the eigenvectors of $A$ ``correlate'' with a re-scaled copy of the integer lattice $t\Z^n$, for any $e^{-cn} \leq t \leq 1$. 

To prove that these quasi-randomness properties hold with probability $1-e^{-cn}$ is a difficult task and depends fundamentally on the ideas in our previous paper \cite{RSM2}. Since we don't want these ideas to distract from the new ideas in this paper we have opted to carry out the details in the Appendix.

With these quasi-randomness conditions in tow, we can return to \eqref{eq:sketch-goal2} and apply Esseen's inequality 
to bound the left-hand-side of \eqref{eq:sketch-goal2} in terms of the characteristic function $\vp(\t)$ of the random variable $\la A^{-1}X, X \ra $,
\[ \min_r \PP_{X}\big( |\la A^{-1}X, X \ra - r| \leq  \eps n^{1/2} \big) \ls \eps \int_{-1/\eps}^{1/\eps} |\vp(\t)| \, d\theta . \]
While this maneuver has been quite successful in work on characteristic functions for (linear) sums of independent random variables, the characteristic function of such quadratic functions has proved to be a more elusive object. For example, even the analogue of the Littlewood-Offord theorem is not fully understood in the quadratic case \cite{costello,meka-nguyen-vu}. Here, we appeal to our quasi-random conditions to avoid some of the traditional difficulties:
we use an application of Jensen's inequality to \emph{decouple} the quadratic form and bound $\vp(\t)$  point-wise  in terms of an average over a related collection of characteristic functions of \emph{linear} sums of independent random variables
\[ |\vp(\t)|^2 \leq  \EE_{Y} |\vp( A^{-1}Y ; \t)| , \]
where $Y$ is a random vector with i.i.d.\ entries and $\vp(v; \t)$ denotes the characteristic function of the sum 
$\sum_{i} v_iX_i$, where $X_i$ are i.i.d.\ distributed according to the original distribution $\zeta$. We can then use our pseudo-random conditions on $A$ to bound 
\[ |\vp(A^{-1}Y; \t)| \ls \exp\left( -c\t^{2} \right), \]
for all but exponentially few $Y$, allowing us to show 
\[ \int_{-1/\eps}^{1/\eps} |\vp(\t)| \, d\theta 
\leq \int_{-1/\eps}^{1/\eps} \left[ \EE_{Y} |\vp(A^{-1}Y ; \t)| \right]^{1/2} 
\leq \int_{-1/\eps}^{1/\eps} \left(\exp\left( -c\t^{2} \right) + e^{-cn}\right)\, d\t = O(1) \]
and thus completing the proof, up to our simplifying assumptions. 

\subsection{Removing the simplifying assumptions}\label{subsec:removing-ass}

While this is a good story to work with, the challenge starts when we turn to remove our simplifying assumptions \eqref{eq:sketch-ass1}, \eqref{eq:sketch-ass2}.  We also note that if one only applies standard methods to remove these assumptions, then one would get stuck at the ``base case'' outlined below.  We start by discussing how to remove the simplifying assumption \eqref{eq:sketch-ass2}, whose resolution governs the overall structure of the paper.

\subsubsection{Removing the assumption \eqref{eq:sketch-ass2}} What is most concerning about making the assumption $\|A_n^{-1}\|_{\HS} \approx n^{-1/2}$ 
is that it is, in a sense, \emph{circular}: If we assume the modest-looking hypothesis $\EE\, \|A^{-1}\|_{\HS} \ls n^{1/2}$, we would be able to deduce
\[\PP( \s_{\min}(A_n) \leq \eps n^{-1/2} ) = \PP( \s_{\max}(A^{-1}_n) \geq n^{1/2}/\eps) \leq \PP( \|A^{-1}_n\|_{\HS} \geq n^{1/2}/\eps) \ls \eps, \] 
by Markov. In other words, showing that $\|A^{-1}\|_{\HS}$ is concentrated about $n^{-1/2}$ (in the above sense) actually \emph{implies} Theorem~\ref{thm:main}.
However this is not as worrisome as it appears at first. Indeed, if we are trying to prove Theorem~\ref{thm:main} for $(n+1) \times (n+1)$ matrices using the above 
outline, we only need to control the Hilbert-Schmidt norm of the inverse of the \emph{minor} $A_n^{-1}$. This suggests an inductive or (as we use) an \emph{iterative} ``bootstrapping argument'' to successively improve the bound. Thus, in effect,  we look to prove 
\[ \EE\, \|A_n^{-1}\|^{\alpha}_{\HS}\1( \s_{\min}(A_n) \geq e^{-cn} ) \ls n^{\alpha/2}, \]
for successively larger $\alpha \in (0,1]$. Note we have to cut out the event of $A$ being singular from our expectation, as this event has non-zero probability.

\subsubsection{ Base case } \label{sss:base} In the first step of our iteration, we prove a ``base case'' of 
\begin{equation}\label{eq:sketch-base} \P(\sigma_{\min}(A_n) \leq \eps n^{-1/2} ) \ls \eps^{1/4} +  e^{-cn}\,\end{equation}
without the assumption \eqref{eq:sketch-ass2} which is equivalent to 
\[ \EE \, \|A_n^{-1}\|^{1/4}_{\HS}\1( \s_{\min}(A_n) \geq e^{-cn} )  \ls n^{1/8}.\] 
To prove this ``base case'' we upgrade \eqref{eq:sketch-1} to 
\begin{equation} \label{eqbaseprep}
\P\left(\sigma_{\min}(A_{n+1}) \leq \eps n^{-1/2} \right) 
\ls  \eps  + \sup_{r \in \R}\, \P\left(\frac{|\langle A_n^{-1}X, X\rangle  - r|}{ \|A_n^{-1} X \|_2} \leq C \eps , \|A_{n}^{-1}\|_{\HS} \leq \frac{n^{1/2}}{\eps} \right)  \,.\end{equation} In other words, we can intersect with the event  
\begin{equation}\label{eq:basis-base} \| A_n^{-1} \|_{\HS} \leq n^{1/2}/\eps \end{equation} at a loss of only $C\eps$ in probability. 

We then push through the proof outlined in Section~\ref{subsec:sketch-sketch} to obtain our initial weak bound of \eqref{eq:sketch-base}. For this, we first use the Hanson-Wright inequality to give a weak version of \eqref{eq:sketch-ass1}, and then use~\eqref{eq:basis-base} as a weak version of our assumption \eqref{eq:sketch-ass2}. We note that this base step \eqref{eq:sketch-base} already improves the best known bounds on the least singular value problem for random symmetric matrices.

\subsubsection{Bootstrapping} To improve on this bound we use a ``bootstrapping'' lemma which, after applying it three times, allows us to improve \eqref{eq:sketch-base} to the near-optimal result 
\begin{equation}\label{eq:nearop}  \P(\sigma_{\min}(A_{n}) \leq \eps n^{-1/2}) \ls  \eps\sqrt{\log 1/\eps} +  e^{-cn}\,.  \end{equation} 

Proving this bootstrapping lemma essentially reduces to the problem of getting good estimates on 
\begin{equation}\label{eq:highd-small-ball} \PP_X\left( \|A^{-1}X\|_2 \leq s  \right) \qquad \text{ for } \qquad  s \in (\eps,n^{-1/2}), \end{equation}
where $A$ is a matrix with $\|A^{-1}\|_{op} = \delta^{-1}$ and $ \delta \in (\eps, c n^{-1/2})$ but is ``otherwise pseudo-random''. Here we require two additional ingredients. 

To start unpacking \eqref{eq:highd-small-ball}, we use that $\|A^{-1}\|_{op} = \delta^{-1}$ to see that if $v$ is a unit eigenvector corresponding to the largest eigenvalue of $A^{-1}$ then
\[ \|A^{-1}X\|_2 \leq s \qquad \text{ implies that }  \qquad |\la X, v\ra| < \delta s.\] 
While this leads to a decent first bound of $O(\delta s)$ on the probability \eqref{eq:highd-small-ball} (after using the quasi-randomness properties of $A$), however this is \emph{not} enough for our purposes and in fact we have to use the additional information that $X$ must \emph{also} have small inner product with many other eigenvectors of $A$ (assuming $s$ is sufficiently small). Working along these lines, we show that \eqref{eq:highd-small-ball} is bounded above by

\begin{equation} \label{eq:neg-dependence} \P_X\bigg(  |\la X, v_1 \ra| \leq s \delta  \text{ and }
  |\la X, v_i\ra| \leq \s_i s  \text{ for all } i =2,\dots, n-1 \bigg), \end{equation}
  where $w_i$ is a unit eigenvector of $A$ corresponding to the singular value $\s_i = \s_i(A)$.
Now, appealing to the quasi-random properties of the eigenvectors of $A^{-1}$, we may apply our approximate negative correlation theorem (Theorem~\ref{thm:invLwO}) 
to see that \eqref{eq:neg-dependence} is at most 
\begin{equation} \label{eq:rest-of-spec} 
	O(\delta s) \exp( - c N_{A}(-c/s,c/s)) 
\end{equation}
 where $c>0$ is a constant and  $N_{A}(a,b)$ denotes the number of eigenvalues of the matrix $A$ in the interval $(a,b)$. 
 The first $O(\delta s)$ factor comes from  the event $|\langle X, v_1 \rangle | \leq s\delta$ and the second factor comes from approximating
\begin{equation}\label{eq:prob}
 \PP_X\Big( |\la X,w_i\ra| < c \text{ for all }  i \text{ s.t. } s\s_i < c \Big)  = \exp\big(-\Theta(N_{A}(-c/s,c/s))\, \big)\,.
\end{equation}
This bound is now sufficiently strong for our purposes,
\emph{provided} the spectrum of $A$ adheres sufficiently closely to the typical spectrum of $A_n$. This now leads us to understand the rest of the spectrum of $A_n$ and, in particular, the next smallest singular values $\sigma_{n-1},\sigma_{n-2},\ldots$.

Now, this might seem like a step in the wrong direction, as we are now led to understand the behaviour of \emph{many} singular values and not just the smallest.
However, this ``loss'' is outweighed by the fact that we need only to understand these eigenvalues on scales of size $\Omega( n^{-1/2} )$, which is now well understood 
due to the important work of Erd\H{o}s, Schlein and Yau \cite{ESY-IMRN}. 

These results ultimately allow us to derive sufficiently strong results on quantities of the form \eqref{eq:highd-small-ball}, which in-turn allow us to prove our ``bootstrapping lemma''. We then use this lemma to prove the near-optimal result

\begin{equation} \label{eq:sketch4/5step2} \P(\sigma_{\min}(A_{n}) \leq \eps n^{-1/2}) \ls  \eps\sqrt{\log 1/\eps} +  e^{-cn}\,. 
\end{equation} 

\subsubsection{Removing the assumption \eqref{eq:sketch-ass1} and the last jump to Theorem~\ref{thm:main}} 
We now turn to discuss how to remove our simplifying assumption \eqref{eq:sketch-ass1}, made above, which will allow us to close the gap between
\eqref{eq:sketch4/5step2} and Theorem~\ref{thm:main}.

To achieve this, we need to consider how $\|A^{-1}X\|_2$ varies about $\|A^{-1}\|_{\HS}$, where we are, again, thinking of $A^{-1} = A_{n}^{-1}$ as sufficiently quasi-random matrix.
Now, the Hanson-Wright inequality tells us that indeed $\|A^{-1}X \|_2$ is concentrated about $\|A^{-1} \|_{\HS}$, on a scale $ \ls \|A^{-1}\|_{op}$. While this is certainly useful for us, it is far from enough to prove Theorem~\ref{thm:main}. For this, we need to rule out any ``macroscopic'' correlation between the events
\begin{equation}\label{eq:sketch-quadvquad} \{|\la A^{-1}X,X\ra -r| < K \eps \|A^{-1}\|_{\HS}  \} \text{  and } \{ \|A^{-1}X\|_2 > K\|A^{-1}\|_{\HS} \} \end{equation} for all $K >  0$. 
Our first step towards understanding \eqref{eq:sketch-quadvquad} is to replace the \emph{quadratic} large deviation event $\|A^{-1}X\|_2 > K\|A^{-1}\|_{\HS} $ with a collection of \emph{linear} large deviation events: 
\[ \la X, w_i \ra > K\log(i+1) ,\] where $w_n,w_{n-1},\ldots,w_1$ are the eigenvectors of $A$ corresponding to singular values $\s_n \leq \s_{n-1} \leq \ldots \leq \s_1$ respectively and the $\log(i+1)$ factor should be seen as a weight function that assigns more weight to the smaller singular values. 

Interestingly, we run into a similar obstacle as before: if the ``bulk'' of the spectrum of $A^{-1}$ is sufficiently 
erratic, this replacement step will be too lossy for our purposes. Thus we are lead to prove another result showing that we may assume that 
the spectrum of $A^{-1}$ adheres sufficiently to the typical spectrum of $A_n$. This reduces to proving 

\[ \EE_{A_n}\, \left[\frac{ \sum_{i=1}^n \s_{n-i-1}^{-2} (\log i )^2}{  \sum_{i=1}^n \s_{n-i-1}^{-2} } \right] = O(1) ,\]
where the left-hand-side is a statistic which measures the degree of distortion of the smallest singular values of $A_n$. To prove this, we again
lean on the work of Erd\H{o}s, Schlein and Yau \cite{ESY-IMRN}.

Thus we have reduced the task of proving the approximate independence of the events at \eqref{eq:sketch-quadvquad} to proving the approximate independence of 
the collection of events
\[ \{|\la A^{-1}X,X\ra -r| < K \eps \|A^{-1}\|_{\HS}  \} \text{  and } \{ \la v_i, X \ra > K\log(i+1)  \}. \]

This is something, it turns out, that we can handle on the Fourier side by using a quadratic analogue of our negative correlation inequality, Theorem~\ref{thm:neg-dependence-2}. The idea here is to prove an Esseen-type bound of the form
\begin{equation}\label{eqFourdecouple} \P( |\langle A^{-1} X, X \rangle - t| < \delta, \langle X,u \rangle \geq s  ) \ls   \delta e^{-s}\int_{-1/\delta}^{1/\delta} \left|\E e^{2\pi i \theta \langle A^{-1} X, X \rangle + \langle X,u \rangle  }\right|\,d\theta\,.\end{equation} Which introduces this extra ``exponential tilt'' to the characteristic function. From here one can carry out the plan sketched in Section~\ref{subsec:sketch-sketch} with this more complicated version of 
Esseen, then integrate over $s$ to upgrade \eqref{eq:sketch4/5step2} to Theorem~\ref{thm:main}.

\subsection{Outline of the rest of the paper}
In the next short section we introduce some key definitions, notation, and preliminaries that we use throughout the paper. In Section~\ref{sec:quasi-randomness} we establish a collection of crucial quasi-randomness properties that hold for the random symmetric matrix $A_n$ with probability $1-e^{-\Omega(n)}$. We shall condition on these events for most of the paper. In Section~\ref{sec:decouplingQForms} we detail our Fourier decoupling argument and establish an inequality of the form~\eqref{eqFourdecouple}. This allows us to prove our new approximate negative correlation result Lemma~\ref{lem:small-ball-LD}. In Section~\ref{sec:baseprep} we prepare the ground for our iterative argument by establishing~\eqref{eqbaseprep}, thereby switching our focus to the study of the quadratic form $\langle A_n^{-1}X, X\rangle$. In Section~\ref{sec:evalcrowding} we prove Theorem~\ref{thm:simple-spectrum} and Theorem~\ref{th:repulsion}, which tell us that the eigenvalues of $A$ cannot `crowd' small intervals. In Section~\ref{ss:spectrum} we establish regularity properties for the bulk of the spectrum of $A^{-1}$. In Section~\ref{sec:small-ball-large-dev} we deploy the approximate negative correlation result (Theorem~\ref{thm:invLwO}) in order to carry out the portion of the proof sketched between~\eqref{eq:highd-small-ball} and~\eqref{eq:prob}. In Section~\ref{sec:intermediate-bounds} we establish our base step~\eqref{eq:sketch-base} and bootstrap this to prove the near optimal bound~\eqref{eq:sketch4/5step2}. In the final section, Section~\ref{sec:proof-main}, we complete the proof of our main Theorem~\ref{thm:main}.

\section{Key Definitions and Preliminaries} \label{sec:defsAndPrelims}

We first need a few notions out of the way which are related to our paper \cite{RSM2} on the singularity of random symmetric matrices. 

\subsection{Subgaussian and matrix definitions}
	
Throughout, $\zeta$ will be a mean-zero, variance $1$ random variable. We define the \emph{subgaussian moment} of $\zeta$ to be 
	$$\| \zeta \|_{\psi_2} := \sup_{p \geq 1} p^{-1/2} (\E\, |\zeta|^p)^{1/p}\, .$$  A mean $0$, variance $1$ random variable is said to be \emph{subgaussian} if $ \| \zeta \|_{\psi_2}$ is finite.
We define $\Gamma$ be the set of subgaussian random variables and, for $B >0$, we define $\Gamma_B \subseteq \Gamma$ to be subset of $\z$ with $\| \zeta \|_{\psi_2} \leq  B$.  

For $\zeta \in \G$, define $\Sym_{n}(\z)$ to be the probability space on $n \times n$
symmetric matrices $A$ for which $(A_{i,j})_{i \geq j}$ are independent and distributed according to $\z$. 
Similarly, we write $X \sim \col_n(\zeta)$ if $X \in \R^n$ is a random vector whose coordinates are i.i.d.\ copies of $\zeta$.

We shall think of the spaces $\{\Sym_n(\zeta)\}_{n}$ as coupled
in the natural way: the matrix $A_{n+1} \sim \Sym_{n+1}(\zeta)$ can be sampled by first sampling $A_n \sim \Sym_n(\zeta)$, which we think of as the principle minor 
$(A_{n+1})_{[2,n+1] \times [2,n+1]}$, and then generating the first row and column of $A_{n+1}$ by generating a random column
$X \sim \col_n(\zeta)$.  
In fact it will make sense to work with a random $(n+1)\times (n+1)$ matrix, which we call $A_{n+1}$ throughout. This is justified as much of the work is done with the principle minor $A_n$ of $A_{n+1}$, due to the bound \eqref{eq:sketch-1} as well as Lemma \ref{lem:distance-conditioned}.

\subsection{Compressible vectors} \label{ss:comp}

We shall require the now-standard notions of \emph{compressible vectors} as defined by Rudelson and Vershynin \cite{RV}.  

For parameters $\rho,\delta \in (0,1)$, we define the set of compressible vectors $\Comp(\delta,\rho)$ to be the set of vectors in $\S^{n-1}$ that are distance at most $\rho$ from a vector supported on at most $\delta n$ coordinates.  We then define the set of \emph{incompressible} vectors to be all other unit vectors, i.e. $\Inc(\delta,\rho) := \S^{n-1} \setminus \Comp(\delta,\rho).$  The following basic fact about incompressible vectors from \cite{RV} will be useful throughout:

\begin{fact}\label{fact:crhodelta}
	For each $\delta,\rho \in (0,1)$ there is a constant $c_{\rho,\delta} \in (0,1)$ so that for all $v \in \Inc(\delta,\rho)$ we have that $|v_j|n^{1/2} \in [ c_{\rho,\delta}, c_{\rho,\delta}^{-1}]$ for at least $c_{\rho,\delta} n$ values of $j$.
\end{fact}

Fact \ref{fact:crhodelta} assures us that for each incompressible vector we can find a large subvector that is ``flat.''  Using the work of Vershynin \cite{vershynin-invertibility}, we will safely be able to ignore compressible vectors.  In particular, \cite[Proposition 4.2]{vershynin-invertibility} implies the following 
Lemma. We refer the reader to Appendix~\ref{app:compressible} for details.

\begin{lemma}\label{lem:V-compressible} For $B >0$ and $\zeta \in \Gamma_B$, let $A \sim \Sym_n(\zeta)$.  Then there exist constants $\rho,\delta,c \in (0,1) $, depending only on $B$, so that  
    $$\sup_{u \in \R^n} \P\big(\exists x \in \Comp(\delta,\rho) , \exists t \in \R : Ax = tu\big) \leq 2e^{-cn}$$
	and $$\P\big(\exists u \in \Comp(\delta,\rho), \exists t \in \R: Au = tu\big) \leq 2e^{-cn}\,.$$
\end{lemma}

The first statement says, roughly, that $A^{-1} u$ is incompressible for each fixed $u$; the second states that all unit eigenvectors are incompressible. 

\begin{remark}[Choice of constants, $\rho,\delta,c_{\rho,\delta}$]
Throughout, we let $\rho,\delta$ denote the constants guaranteed by Lemma \ref{lem:V-compressible} and $c_{\rho,\delta}$ the corresponding constant from Fact \ref{fact:crhodelta}. These constants shall appear throughout the paper and shall always be considered as fixed.
\end{remark}

Lemma \ref{lem:V-compressible} follows easily from \cite[Proposition 4.2]{vershynin-invertibility} with a simple net argument.

\subsection{Notation}

We quickly define some notation.  For a random variable $X$, we use the notation $\E_X$ for the expectation with respect to $X$ and we use the notation $\P_X$ analogously.  For an event $\cE$, we write $\one_{\cE}$ or $\one \{ \cE\}$ for the indicator function of the event $\cE$.  We write  $\E^{\cE}$ to be the expectation defined by $\E^{\cE}[\, \cdot\, ] = \E[\, \cdot \, \one_\cE]$.  For a vector $v \in \R^{n}$ and $J \subset [n]$, we write $v_J$ for the vector whose $i$th coordinate is $v_i$ if $i \in J$ and $0$ otherwise. 

We shall use the notation $X \ls Y$ to indicate that there exists a constant 
$C>0$ for which $X \leq CY$. In a slight departure from convention, \emph{we will always allow this constant to depend on the subgaussian constant} $B$, if present.
We shall also let our constants implicit in big-O notation to depend on $B$, if this constant is relevant in the context. We hope that we have been clear as to where the subgaussian constant is relevant, and so this convention is to just reduce added clutter.

\section{Quasirandomness properties}\label{sec:quasi-randomness}
In this technical section, we define a list of ``quasi-random'' properties of $A_n$ that hold with probability $1-e^{-\Omega(n)}$.  This probability is large enough that we can assume that these properties hold for all the principle minors of $A_{n+1}$.  Showing that several of these quasi-random properties hold with probability $1-e^{-\Omega(n)}$ will prove to be a challenging task and our proof will depend deeply on ideas from our previous paper~\cite{RSM2}, on the singularity probability of a random symmetric matrix. So as not to distract from the new ideas in this paper, we do most of this work in the Appendix.

\subsection{Defining the properties}

It will be convenient to assume throughout that every minor of $A_{n+1}$ is invertible and so we will perturb the matrix slightly so that we may assume this.  If we add to $A_{n+1}$ an independent random symmetric matrix whose upper triangular entries are independent gaussian random variables with mean $0$ and variance $n^{-n}$, then with probability $1 - e^{-\Omega(n)}$ the singular values of $A_{n+1}$ move by at most, say, $n^{-n/3}$.  Further, after adding this random gaussian matrix, every minor of the resulting matrix is invertible with probability $1$.  Thus, we will assume without loss of generality throughout that every minor of $A_{n+1}$ is invertible.

In what follows, we let $A=A_n \sim \Sym_n(\zeta)$ and let $X \sim \col_{n}(\zeta)$ be a random vector, independent of $A$. 
Our first quasi-random property is standard from the concentration of the operator norm of a random symmetric matrix. We define $\cE_{1}$ by
\begin{equation} \cE_1  = \{\|A\|_{op} \leq 4 \sqrt{n} \} \label{it:cE-op}. 
\end{equation}

For the next property we need a definition. Let $X,X' \sim \col_n(\z)$ and define the random vector in $\R^n$ as $\tilde{X} := X_J - X'_J$, where $J \subseteq [n]$ is a 
$\mu$-random subset, i.e. for each $j \in [n]$ we have $j \in J$ independently with probability $\mu$. 
The reason behind this definition is slightly opaque at present, but will be clear in the context of Lemma~\ref{lem:small-ball-LD} in Section~\ref{sec:decouplingQForms}. Until we get there it is reasonable 
to think of $\tilde{X}$ as being essentially $X$; in particular, it is a random vector with i.i.d.\ subgaussian entries with mean $0$ and variance $\mu$. We now define $\cE_{2}$ to be the event in $A$ defined by  
\begin{equation} \cE_2 = \left\{\P_{\Xt}\left( A^{-1} \Xt / \|A^{-1} \Xt\|_2 \in \Comp(\delta,\rho) \right) \leq  e^{-c_2 n} \right\}\,. \label{it:cE-comp}
\end{equation}
We remind the reader that $\Comp(\delta,\rho)$ is defined in Section \ref{ss:comp}, and $\delta,\rho \in (0,1)$ are constants, fixed throughout the paper, and chosen according to Lemma \ref{lem:V-compressible}.  In the (rare) case that $\Xt = 0$, we interpret $\P_{\Xt}( A^{-1} \Xt / \|A^{-1} \Xt\|_2 \in \Comp(\delta,\rho) ) = 1$

Recalling the least common denominator defined at \eqref{eq:lcd-def}, we now define the event $\cE_3$ by 
\begin{equation}\cE_3 = \{ D_{\alpha,\gamma}(u) \geq e^{c_3 n} \text{ for every unit eigenvector }u \text{ of }A\}\,. \label{it:cE-evec}
\end{equation}

The next condition tells us that the random vector $A^{-1}\Xt$ is typically unstructured.   We will need a slightly stronger notion of structure than just looking at the LCD, in that we will need all sufficiently large subvectors to be unstructured.  For $\mu \in (0,1)$, define the \emph{subvector least common denominator} as 
$$ \hat{D}_{\alpha,\gamma,\mu}(v) :=\min_{\substack{I\subset [n]\\|I|\geq (1-2\mu)n}}D_{\alpha,\gamma}\left(v_I/\|v_I\|_2\right)\,.$$ 
We note that this is closely related to the notion of ``regularized least common denominator'' introduced by Vershynin in~\cite{vershynin-invertibility}.

Now, if we define the random vector
$v = v(\Xt) := A^{-1} \Xt$, then we define $\cE_4$ to be the event that $A$ satisfies 
\begin{equation} \label{it:cE-LCD} \cE_4 = \left\{\P_{\Xt}\left( \Dhat_{\alpha,\gamma,\mu}\left(v \right) < e^{c_4 n}    \right) \leq e^{-c_4n} \right\}\,. 
\end{equation}
As is the case for $\cE_2$, under the event that $\Xt = 0$, we interpret $\P_{\Xt}( \Dhat_{\alpha,\gamma,\mu}(v ) < e^{c_4 n}    )  = 1$.

We now define our main quasirandomness event $\cE$ to be the intersection of these events: 
\begin{align}\label{eq:cEdef}
\cE:= \cE_1 \cap \cE_2 \cap \cE_3 \cap \cE_4\,.
\end{align}

The following lemma essentially allows us to assume that $\cE$ holds in what follows. 

\begin{lemma}\label{lem:cE-exp}
For $B>0$, $\zeta \in \G_{B}$, and all sufficiently small $\alpha,\gamma,\mu \in (0,1)$, there exist constants $c_2,c_3,c_4 \in (0,1)$ appearing in \eqref{it:cE-comp}, \eqref{it:cE-evec} and \eqref{it:cE-LCD} so that
	\begin{equation} \P_A(\cE^c) \leq 2e^{-\Omega(n)}. \label{eq:lem-cE} \end{equation}
\end{lemma}

\begin{remark}[Choice of constants, $\alpha,\g, \mu$]
We take $\alpha,\g \in (0,1)$
to be sufficient small so that Lemma~\ref{lem:cE-exp} holds. For $\mu$ we will choose it to be sufficiently small so that (1) Lemma~\ref{lem:cE-exp} holds; (2) we have $\mu \in (0,2^{-15})$ and so that; (3) $\mu>0$ is small enough to guarantee that every set $I \subseteq [n]$ with $|I| \geq (1-2\mu)n$ satisfies 
\begin{equation} \label{eq:mu-choice} \|w\|_2 \leq c^{-2}_{\rho,\delta} \|w_I\|_2, \end{equation} for every $w \in \Inc(\delta,\rho)$.  
This is possible by Fact~\ref{fact:crhodelta}.	  
These constants $\alpha,\g,\mu$ will appear throughout the paper and will always be thought of as fixed according to this choice. 
\end{remark}

\subsection{Statement of our master quasi-randomness theorem and the deduction of Lemma \ref{lem:cE-exp}}

	We will deduce Lemma~\ref{lem:cE-exp} from a ``master quasi-randomness theorem'' together with a handful of now-standard results in the area.  
	
	For the purposes of the following sections, we shall informally consider a vector as ``structured'' if 
	\[ \hat{D}_{\alpha,\gamma,\mu}(v) \leq e^{c_\Sigma  n} \] where $c_\Sigma \in (0,1)$ is a small constant, to be chosen shortly.
	Thus it makes sense to define the set of ``structured directions'' on the sphere
	\begin{equation}\label{eq:Sigma-def}
		\Sigma = \Sigma_{\alpha,\g,\mu} := \{ v \in \S^{n-1} :  \hat{D}_{\alpha,\gamma,\mu}(v) \leq e^{c_{\Sigma} n} \}\,.
	\end{equation}
	
	We now introduce our essential quasi-randomness measure of a random matrix. For $\zeta \in \G$, 
	$A \sim \Sym_n(\zeta)$, and a given vector $w \in \R^n$, define \begin{equation}\label{eq:qw-def}
		q_n(w) = q_n(w;\alpha,\gamma,\mu) := \Pr_A\left(\exists v\in \Sigma  \text{ and } \exists s,t\in [-4\sqrt{n}, 4\sqrt{n}]:~Av=sv+tw \right)
	\end{equation}
	and set 
	\begin{equation} \label{eq:def-qn} 
		q_n = q_n(\alpha,\g,\mu) := \sup_{w\in \S^{n-1}} q_n(w)\,.
	\end{equation}
		We now state our ``master quasi-randomness theorem'', from which we deduce Lemma~\ref{lem:cE-exp}.
	
	\begin{theorem}[Master quasi-randomness theorem] \label{thm:qLCD}
		For $B >0$ and $\zeta \in \G_B$, there exist constants $\alpha,\gamma,\mu,c_{\Sigma},c \in (0,1)$ depending only on $B$ so that  
		\[q_{n}(\alpha, \gamma ,\mu) \leq  2e^{-cn}\,. \]
		
	\end{theorem} 

	The proof of Theorem \ref{thm:qLCD} is quite similar to the main theorem of \cite{RSM2}, albeit with a few technical adaptations, and is proved in 
	the Appendix.  Note that $q_n(\alpha,\gamma,\mu)$ is monotone decreasing as $\alpha,\gamma$ and $\mu$ decrease.  As such, Theorem \ref{thm:qLCD} implies that its conclusion holds for all sufficiently small $\alpha,\gamma,\mu$ as well.

	\vspace{3mm}
	
	We now prove that our pseudorandom event $\cE = \cE_1 \cap \cE_2 \cap \cE_3 \cap \cE_4$ holds with probability $1-e^{-\Omega(n)}$. 
	
	\begin{proof}[Proof of Lemma~\ref{lem:cE-exp}]
		
		\emph{The event} $\cE_1$: From \cite{feldheim-sodin} we may deduce\footnote{Technically, the result of \cite{feldheim-sodin} is sharper and for random matrices whose entries are symmetric random variables. However \eqref{eq:op-norm-bound} follows from \cite{feldheim-sodin} along with a ``symmetrization trick''. } the following concentration bound
	 \begin{equation} \label{eq:op-norm-bound}
			\P\big(\, \|A\|_{op} \geq  (3 + t)\sqrt{n}   \big) \ls e^{-ct^{3/2}n},
		\end{equation}
	which holds\footnote{ We use this bound rather than the more standard concentration bounds for the norm of subgaussian random matrices for aesthetic purposes: it allows us to use the absolute constant ``4'' in~\eqref{it:cE-op}. } for all $t \geq 0$ . Thus, by \eqref{eq:op-norm-bound}, the event $\cE_1$ at \eqref{it:cE-op} fails with probability $\ls e^{-\Omega(n)}$.

\noindent \emph{The event} $\cE_2$:  By Lemma \ref{lem:V-compressible} there is a $c > 0$ so that for each $u \neq 0$ we have 
		$$\P_A(A^{-1} u / \|A^{-1}u \|_2 \in \Comp(\delta,\rho)) \leq e^{-cn}\, .$$  
Applying Markov's inequality shows 
$$\P_A\left(\P_{\Xt}\left( A^{-1}\Xt / \|A^{-1}\Xt\|_2 \in \Comp(\delta,\rho), \Xt \neq 0 \right) > e^{-cn/2}\right) \leq e^{-cn/2}\,,$$ 
and so the event in~\eqref{it:cE-comp} fails with probability at most $O\left(e^{-\Omega(n)}\right)$, under the event $\Xt \neq 0$.  By Theorem 3.1.1 in \cite{vershynin2018high} we have that \begin{equation}\label{eq:Xt=0}
			\P_{\Xt}(\Xt = 0) \leq e^{-\Omega(\mu n)}\,.
		\end{equation}
		Choosing $c_2$ small enough shows an exponential bound on $\P(\cE_2^c)$.
\vspace{2mm}
		
\noindent \emph{The event} $\cE_3$: If $D_{\alpha,\g}(u) \leq e^{c_3n}$, for an $u$ an eigenvector $Au = \l v$, we have that 
		\[ \Dhat_{\alpha,\gamma,\mu}(u) \leq D_{\alpha,\gamma}(u) \leq e^{c_3 n} , \]
		where the first inequality is immediate from the definition.
		Now note that if $\cE_1$ holds then $\l \in [-4\sqrt{n},4\sqrt{n}]$ and so
		\[ \PP(\cE^{c}_3) \leq \PP\big( \exists u \in \Sigma, \l \in [-4\sqrt{n},4\sqrt{n}] : Au = \l u \big) + \PP(\cE_1^{c}) \leq q_n(0) + e^{-\Omega(n)},\]
		where the first inequality holds if we choose $c_3\leq c_\Sigma$. We now apply Theorem~\ref{thm:qLCD} to see $q_n(0) \leq q_n \ls e^{-\Omega(n)}$, yielding the 
		desired result.		
\vspace{2mm}

\noindent \emph{The event} $\cE_4$: Note first that by \eqref{eq:Xt=0}, we may assume $\Xt \neq 0$.  For a fixed instance of $\Xt \not = 0 $, we have 
\begin{equation} \label{eq:e4-check1} \PP_A\left( \hat{D}_{\alpha,\g,\mu}\left( A^{-1}\tilde{X}/\|\tilde{X}\|_2  \right) < e^{c_4n} \right)  \leq \PP_A\big( \exists v \in \Sigma : Av = \tilde{X}/\|\tilde{X}\|_2  \big) \leq q_n\left(\tilde{X}/\|\tilde{X}\|_2  \right), \end{equation} which is at most
$e^{-\Omega(n)}$, by Theorem~\ref{thm:qLCD}. Here the first inequality holds when $c_4 \leq c_{\Sigma}$.

We now write $v = A^{-1}\tilde{X}/\|\tilde{X}\|_2$ and apply Markov's inequality
\[ \PP(\cE_4^c) = \PP_{A}\left( \PP_{\tilde{X}}\left( \hat{D}_{\alpha,\g,\mu}(v) < e^{c_4n} \right) \geq e^{-c_4n} \right)\leq e^{c_4n} \EE_{\tilde{X}} \PP_{A}( \hat{D}_{\alpha,\gamma,\mu}(v) < e^{c_4}n) = e^{-\Omega(n)}, \] where the last line follows when $c_4$ is taken small relative to the implicit constant in the bound on the right-hand-side of \eqref{eq:e4-check1}.

Since we have shown that each of $\cE_1,\cE_2,\cE_3,\cE_4$ holds with probability $1-e^{-\Omega(n)}$, the intersection fails with exponentially small probability.
	\end{proof}

\section{Decoupling Quadratic Forms}\label{sec:decouplingQForms}

In this section we will prove our Esseen-type inequality that will allow us to deal with a small ball event and a large deviation event simultaneously.

\begin{lemma}\label{lem:tilted-Esseen}
For $B >0$, let $\zeta \in \G_B$ and $X \sim \col_n(\zeta)$. Let $M$ be an $n \times n $ symmetric matrix, $u\in \R^n$, $t \in \R$ and $s, \delta \geq 0$.  Then 
\begin{equation} \label{eq:tilted-Esseen} \P( |\langle M X, X \rangle - t| < \delta, \langle X,u \rangle \geq s  ) \ls  \delta e^{-s}\int_{-1/\delta}^{1/\delta} \left|\E\, e^{2\pi i \theta \langle M X, X \rangle + \langle X,u \rangle  }\right|\,d\theta\,.
\end{equation}
\end{lemma}

We will then bound the integrand (our so-called ``titled'' characteristic function) with a \emph{decoupling} maneuver, somewhat similar to a ``van der Corput trick''
in classical Fourier analysis.  This amounts to a clever application of Cauchy-Schwarz inspired by Kwan and Sauermann's work on Costello's conjecture \cite{kwan2020algebraic} (a similar technique appears in \cite{berkowitz} and \cite{nguyen-singularity}). 
We shall then be able to mix in our quasi-random conditions on our matrix $A$ to ultimately obtain Lemma~\ref{lem:small-ball-LD},
which gives us a rather tractable bound on the left-hand-side of \eqref{eq:tilted-Esseen}. To state this lemma, let us recall that $\cE$ (defined at~\eqref{eq:cEdef}) is the set 
of symmetric matrices satisfying the quasi-randomness conditions in the previous section, Section~\ref{sec:quasi-randomness}. Also recall that the constant $\mu \in (0,2^{-15})$ is defined in Section~\ref{sec:quasi-randomness} so that Lemma~\ref{lem:cE-exp} holds and is treated as fixed constant throughout this paper.

\begin{lemma}\label{lem:small-ball-LD}
	For $B >0$, let $\z \in\G_B$, $X \sim \col_n(\z)$ and let $A$ be a real symmetric $n\times n$ matrix with $A \in \cE$ and set $\mu_1 := \s_{\max}(A^{-1})$.  Also let $s \geq 0, \delta > e^{-c n}$ and 
$u \in \S^{n-1}$. Then
	\[ \P_X\left(\left|\langle A^{-1} X,X\rangle - t \right| \leq \delta \mu_1, \langle X, u \rangle \geq s \right)
	\ls \delta e^{-s} \int_{-1/\delta}^{1/\delta} I(\theta) ^{1/2}\,d\theta +  e^{-\Omega(n)}\, , \] 
	where
	\[ I(\theta) := \E_{J,X_J,X_J'} \, \exp\bigg( \langle (X + X')_J,u \rangle -c\theta^2 \mu_1^{-2} \|A^{-1}(X - X')_J \|_2^2\bigg)\,, \]
	 $X' \sim \col_n(\z)$ is independent of $X$, and $J \subseteq [n]$ is a $\mu$-random set. Here $c > 0$ is a constant depending only on $B$.
	 
\end{lemma}

While the definition of $I(\t)$ (and therefore the conclusion of the lemma) is a bit mysterious at this point, we assure the reader that this is a step in right
direction.

All works bounding the singularity probability for random symmetric matrices contain a related decoupling step \cite{nguyen-singularity,vershynin-invertibility,ferber-jain,CMMM,JSS-symmetric,RSM1,RSM2}, starting with Costello, Tao and Vu's breakthrough \cite{costello-tao-vu} building off of Costello's earlier work \cite{costello} on anticoncentration of bilinear and quadratic forms.  A subtle difference in the decoupling approach from \cite{kwan2020algebraic} used here is that the quadratic form is decoupled \emph{after} bounding a small ball probability in terms of the integral of a characteristic function rather than on the probability itself; the effect of this approach is that we do not lose a power of $\delta$, but only lose by a square root ``under the integral'' on the integrand $I(\theta)$.

\subsection{Proofs}
We now dive in and prove our Ess{e}en-type inequality. For this we shall appeal to the classical Ess{e}en inequality \cite{esseen}: if $Z$ is a random variable
taking values in $\R$ with characteristic function $\vp_Z(\t):= \E_Z\, e^{2\pi i \theta Z}$, then for all $t \in \R$ we have
\[  \PP_X( |Z - t| \leq \delta ) \ls \delta  \int_{-1/\delta}^{1/\delta}\, |\vp_Z( \t )| \, d\t . \]

We shall also use the following basic fact about subgaussian random vectors (see, for example, \cite[Prop.\ 2.6.1]{vershynin2018high}): If $\zeta \in \G_B$ and $Y \sim \col_n(\zeta)$ then for every vector $u \in \R^n$ we have 
\begin{equation} \label{eq:exp-moment} 
	\E_Y e^{\langle Y, u \rangle } \leq \exp(2B^2\|u\|_2^2)\, .
\end{equation} 

\vspace{1mm}

\begin{proof}[Proof of Lemma~\ref{lem:tilted-Esseen}]
	Since $\one\{ x \geq s \} \leq e^{x - s}$, we may bound \begin{equation}\label{eq:esseen1}
	\P_X( |\langle M X, X \rangle - t| < \delta, \langle X,u \rangle \geq s  ) \leq e^{-s}\E\left[\one\{|\langle M X, X \rangle - t| < \delta\} e^{\langle X,u\rangle } \right]\,.
	\end{equation}
	Define the random variable $Y \in \R^n$ by  
	\begin{equation}\label{eq:defY} \P(Y \in U) = (\E\, e^{\langle X,u\rangle })^{-1} \E [\one_U e^{\langle X,u \rangle}],
	\end{equation}
	for all open $U \subseteq \R^n$. Note that the expectation $\E_X e^{\langle X, u \rangle}$ is finite by \eqref{eq:exp-moment}. We now use this definition to rewrite the expectation on the right-hand-side of \eqref{eq:esseen1},
	 $$\E_X\left[\one\{|\langle M X, X \rangle - t| < \delta\} e^{\langle X,u\rangle }  \right] = \left( \E\, e^{\langle X ,u \rangle } \right) \P_Y( |\langle MY,Y\rangle - t| \leq \delta )\,.$$
	Thus, we may apply Esseen's Lemma to the random variable $Y$ to obtain 
	$$\P_Y( |\langle MY,Y\rangle - t| \leq \delta ) \ls  \delta \int_{-1/\delta}^{1/\delta} |\E_Y\, e^{2\pi i\theta \langle MY,Y\rangle}| \, d\theta\,. $$
	By the definition of $Y$ we have 
	$$\E_Y\, e^{2\pi i\theta \langle MY,Y\rangle} = \left(\E_X\, e^{\langle X,u\rangle}\right) ^{-1} \E\, e^{2\pi i\theta\langle MX,X\rangle + \langle X,u\rangle},$$
	completing the lemma.
\end{proof}

\vspace{2mm}

To control the integral on the right-hand-side of Lemma~\ref{lem:tilted-Esseen}, we will appeal to the following decoupling lemma,
which is adapted from Lemma 3.3 from \cite{kwan2020algebraic}.

\begin{lemma}[Decoupling with an exponential tilt]\label{lem:decoupling}
	Let $\zeta \in \Gamma$, let $X,X' \sim \col_n(\zeta)$ be independent and let $J\cup I = [n]$ be a partition of $[n]$.  
	Let $M$ be a $n \times n$ symmetric matrix and let $u\in \R^n$.  
	 Then 
	\begin{align*}
	\left|\E_X\, e^{2\pi i \theta \langle MX,X\rangle + \langle X,u \rangle }  \right|^2 
	\leq \E_{X_J,X_J'}\, e^{\langle (X + X')_J,u\rangle} \cdot \left|\E_{X_I} e^{4\pi i\theta \langle M(X - X')_J, X_I \rangle + 2\langle X_I,u\rangle  } \right|. 
	\end{align*}
\end{lemma}
\begin{proof}
	After partitioning the coordinates of $X$ according to $J$ and writing $\E_X = \EE_{X_I}\EE_{X_J}$, we apply Jensen's inequality to obtain
$$
 E := \left|\E_X\, e^{2\pi i \theta \langle MX,X\rangle + \langle X,u \rangle }  \right|^2
= \left|\E_{X_I} \EE_{X_J}\, e^{2\pi i \theta \langle MX,X\rangle + \langle X,u \rangle }  \right|^2 
	\leq \E_{X_I} \left|\E_{X_J}e^{2\pi i \theta \langle MX,X\rangle + \langle X,u \rangle }  \right|^2 . $$
	We now expand the square $\left|\E_{X_J}e^{2\pi i \theta \langle MX,X\rangle + \langle X,u \rangle }  \right|^2$  as \begin{align*}
	&\E_{X_J,X_J'} e^{2\pi i \theta \langle M(X_I + X_J),(X_I + X_J)\rangle + \langle (X_I + X_J),u\rangle - 2\pi i \theta \langle M(X_I + X_J'),(X_I + X_J')\rangle + \langle (X_I + X_J'),u\rangle } \\
	&= \E_{X_J,X_J'} e^{4\pi i \theta \langle M(X_J - X_J'),X_I\rangle + \langle X_J + X_J', u \rangle + 2\langle X_I,u\rangle + 2\pi i \langle M X_J, X_J \rangle - 2\pi i \langle M X_J',X_J'\rangle  }, \end{align*}
	where we used the fact that $M$ is symmetric.  Thus, swapping expectations yields  \begin{align*}
		E &\leq \E_{X_J,X_J'} \E_{X_I} e^{4\pi i \theta \langle M(X_J - X_J'),X_I\rangle + \langle X_J + X_J', u \rangle + 2\langle X_I,u\rangle + 2\pi i \langle M X_J, X_J \rangle - 2\pi i \langle M X_J',X_J'\rangle  } \\
	&\leq \E_{X_J,X_J'} \left|\E_{X_I} e^{4\pi i \theta \langle M(X_J - X_J'),X_I\rangle + \langle X_J + X_J', u \rangle + 2\langle X_I,u\rangle + 2\pi i \langle M X_J, X_J \rangle - 2\pi i \langle M X_J',X_J'\rangle  } \right| \\
	& = \E_{X_J,X_J'}\, e^{\langle X_J + X_J',u \rangle }  \left|\E_{X_I} e^{4\pi i\theta \langle M(X - X')_J, X_I \rangle + 2\langle X_I,u\rangle  } \right|,\, 
\end{align*} as desired. Here we could swap expectations since all expectations are finite, due to the subgaussian assumption on $\zeta$. \end{proof}

\vspace{2mm}

We need a basic bound that will be useful for bounding our tilted characteristic function.  This bound appears in the proof of Theorem 6.3 in Vershynin's paper \cite{vershynin-invertibility}.

\begin{fact}\label{fact:untilted-cf-bound}
	For $B>0$, let $\zeta \in \G_B$, let $\zeta'$ be an independent copy of $\zeta$ and set $\xi = \zeta - \zeta'$. Then for all $a \in \R^n$ we have $$\prod_{j}\E_{\xi}\, |\cos(2\pi \xi a_j)| \leq \exp\left(-c \min_{r \in [1,c^{-1}]} \| r a\|_{\T}^2\right)\,,$$
	where $c>0$ depends only on $B$.
\end{fact}

A simple symmetrization trick along with Cauchy-Schwarz will allow us to prove a similar bound for the tilted characteristic function.

\begin{lemma}\label{lem:char-fcn-bound}
	For $B >0$, let $\z \in \G_B$, $X \sim \col_n(\z)$ and let $u ,v \in \R^n$. Then
\begin{equation} \label{eq:char-fcn-bound} \left|\E_X e^{2\pi i  \langle X,v \rangle  + \langle X,u \rangle} \right| \leq \exp\left(-c\min_{r \in [1,c^{-1}]} \|rv\|_{\T}^2 + c^{-1} \|u \|_2^2\right)\, ,
\end{equation}
	where $c \in (0,1)$ depends only on $B$.
\end{lemma}
\begin{proof}
	 Let $\zeta'$ be an independent copy of $\zeta$ and note that
	  $$\left|\E_\zeta\, e^{2\pi i  \zeta v_j  + \zeta u_j} \right|^2 = \E_{\zeta,\zeta'}\, e^{2\pi i (\zeta-\zeta')v_j   +  (\zeta+\zeta')u_j } = \E_{\zeta,\zeta'}\left[ e^{ (\zeta + \zeta')u_j } \cos(2\pi(\zeta - \zeta')v_j)\right]\,. $$  
	Let $\Xt = (\Xt_i)_{i=1}^n$, $\Yt = (Y_i)_{i=1}^n$ denote vectors with i.i.d.\ coordinates distributed as $\xi:=\zeta - \zeta'$ and $\zeta + \zeta'$, respectively. We have 
	\begin{equation}\label{eq:char-fcn-bnd2}
	\left|\E_X e^{2\pi i  \langle X,v \rangle  + \langle X,u \rangle} \right|^2 \leq \E\, e^{\langle \Yt,u\rangle} \prod_{j } \cos (2\pi\Xt_jv_j) \leq \left(\E_{\Yt} e^{2\langle \Yt,u\rangle }\right)^{1/2} \left( \prod_{j} \E_{\xi} |\cos(2\pi \xi v_j)| \right)^{1/2},
	\end{equation}
	where we have applied the Cauchy-Schwarz inequality along with the bound $|\cos(x)|^2 \leq |\cos(x)|$ to obtain the last inequality. By \eqref{eq:exp-moment}, the first expectation on the right-hand-side of \eqref{eq:char-fcn-bnd2} is at most $\exp(O(\|u\|_2^2))$.  Applying Fact \ref{fact:untilted-cf-bound} completes the Lemma.
\end{proof}

\vspace{2mm}

\subsection{Quasi-random properties for triples $(J,X_J,X'_J)$}

We now prepare for the proof of Lemma~\ref{lem:small-ball-LD} by introducing a quasi-randomness notion on triples $(J,X_J,X'_J)$. Here
$J \subseteq [n]$ and $X,X' \in \R^n$. For this we fix a $n\times n$ real symmetric matrix $A \in \cE$
and define the event $\cF = \cF(A)$ as the intersection of the events $\cF_1,\cF_2,\cF_3$ and $\cF_4$, which are defined as follows. Given a triple $(J,X_J,X'_J)$, 
we write $\tilde{X} := X_J - X_J'$.

Define events $\cF_1,\cF_2,\cF_3(A)$ by 
\begin{align}
	\cF_1 &:= \left\{ |J| \in [\mu n/2, 2\mu n] \right\} \label{it:cF-J} \\ 
	\cF_2 &:= \{ \|\Xt \|_2 n^{-1/2} \in [c , c^{-1}]\} \label{it:cF-norm} \\
	\cF_3(A) &:= \{  A^{-1 } \Xt / \|A^{-1}\Xt\|_2 \in \Inc(\delta,\rho)  \}\,. \label{it:cF-Comp} 
\end{align}

Finally, we write $v = v(\tilde{X}) := A^{-1} \Xt$ and $I := [n] \setminus J$ and  then define $\cF_4(A)$ by 
\begin{equation} 
	\cF_4(A) := \left\{D_{\alpha,\gamma}\left(\frac{v_I}{\|v_I\|} \right) > e^{c n} \right\}\,. \label{it:cF-LCD} 
\end{equation}

We now define $\cF(A) := \cF_1 \cap \cF_2 \cap \cF_3(A) \cap \cF_4(A)$
and prove the following basic lemma that will allow us to essentially assume that \eqref{it:cF-J},\eqref{it:cF-norm},\eqref{it:cF-Comp},\eqref{it:cF-LCD}
hold in all that follows.  We recall that the constants $\delta,\rho,\mu,\alpha,\gamma$ were chosen in Lemma \ref{lem:V-compressible} and Lemma \ref{lem:cE-exp} as a function of the subgaussian moment $B$.  Thus the only new parameter in $\cF$ is the constant $c$ in lines \eqref{it:cF-norm} and \eqref{it:cF-LCD}.

\begin{lemma}\label{lem:cF-exp}
For $B >0$, let $\zeta \in \G_B$, let $X,X' \sim \col_n(\zeta)$ be independent and let $J \subseteq [n]$ be a $\mu$-random subset.
Let $A$ be a $n \times n$ real symmetric matrix with $A \in \cE$. 
 We may choose the constant $c \in (0,1)$ appearing in \eqref{it:cF-norm} and \eqref{it:cF-LCD} as a function of $B$ and $\mu$ so that 
	\[ \P_{J,X_J,X'_J}(\cF^c) \ls e^{-cn}\,.\]
\end{lemma}

\begin{proof}
	For $\cF_1$, we use Hoeffding's inequality to see $\P(\cF_1^c) \ls e^{-\Omega(n)}$.  To bound $\P(\cF_2^c)$, we note that the entries of $\Xt$ are independent, subgaussian, and have variance $2\mu$, and so $\Xt/(\sqrt{2\mu})$ has i.i.d. entries with mean zero, variance $1$ and subgaussian moment bounded by $B/\sqrt{2\mu}$. Thus from Theorem 3.1.1 in \cite{vershynin2018high} we have 
	\[ \PP\big( \, |\|\Xt\|_2 - \sqrt{2n\mu}| > t  \big) < \exp( -c\mu t^2/B^4 ). \]
	
	For $\cF_3(A), \cF_4(A)$, recall that $A \in \cE$ means that \eqref{it:cE-comp} and \eqref{it:cE-LCD} hold, thus exponential bounds on $\P(\cF_3^c)$ and $\P(\cF_4^c)$ follow from Markov's inequality.
\end{proof}

\subsection{Proof of Lemma~\ref{lem:small-ball-LD}}

\noindent We now prove Lemma~\ref{lem:small-ball-LD} by applying the previous three lemmas in sequence.  

\begin{proof}[Proof of Lemma~\ref{lem:small-ball-LD}]
	Let $\delta \geq e^{-c_1n}$ where we will choose $c_1>0$ to be sufficiently small later in the proof.  
	Apply Lemma \ref{lem:tilted-Esseen} to write	 \begin{equation}\label{eq:essen-app}
	\P_X\left(\left|\langle A^{-1} X,X\rangle - t \right| \leq \delta \mu_1, \langle X, u \rangle \geq s \right) 
	\ls \delta e^{-s}  \int_{-1/\delta}^{1/\delta}\left| \E_X \, e^{2\pi i \theta \frac{\langle A^{-1} X,X \rangle}{\mu_1} + \langle X,u\rangle } \right|  \,d\theta\, ,
	\end{equation}
	where we recall that $\mu_1 = \s_{\max}(A^{-1})$.
	We now look to apply our decoupling lemma, Lemma~\ref{lem:decoupling}.
	Let $J$ be a $\mu$-random subset of $[n]$, define $I:=[n] \setminus J$ and let $X'$ be an independent copy of $X$. By Lemma \ref{lem:decoupling} we have 
	\begin{align} \label{eq:J-decoup}
	&\left|\E_X\, e^{2\pi i \theta \frac{\langle A^{-1} X,X \rangle}{\mu_1} + \langle X,u\rangle } \right|^2
	\leq  \Ex_{J}\EE_{X_J, X'_J}\, e^{\langle (X + X')_J,u \rangle  } \cdot \left| \Ex_{X_I}\, e^{4\pi i \theta\left\langle\frac{A^{-1}\Xt}{\mu_1} ,X_I\right\rangle+2\langle X_I,u\rangle}\right| \,,
	\end{align}
	where we recall that $\Xt=(X-X')_J$.
	
	We first consider the contribution to the expectation on the right-hand-side of \eqref{eq:J-decoup} from triples $(J,X_J,X_J') \not\in \cF$. 
	For this let $Y$ be a random vector such that $Y_j=X_j+X'_j$, if $j\in J$, and $Y_j=2X_j$, if $j\in I$. Applying the triangle inequality, we have 
	\[
	\Ex_{J, X_J, X_J'}^{\cF^c}\, e^{\langle (X + X')_J,u \rangle  }
	 \cdot \left|\Ex_{X_I}\, e^{4\pi i \theta\langle \frac{A^{-1} \Xt }{\mu_1},X_I\rangle+2\langle X_I,u\rangle} \right|
	 \leq \Ex_{J, X_J, X_J'}^{\cF^c}\, e^{\langle (X + X')_J,u \rangle }\Ex_{X_I}\, e^{2\langle X_I,u\rangle} 
	= \E_{J,X,X'}^{\cF^c}e^{\langle Y,u \rangle } .\]
	By Cauchy-Schwarz, \eqref{eq:exp-moment} and  Lemma \ref{lem:cF-exp}, we have
	\begin{equation}\label{eq:Fc} \E_{J,X,X'}^{\cF^c}\,e^{\langle Y,u \rangle } \leq \E_{J, X, X'}\left[e^{\langle Y, 2u \rangle}\right]^{1/2} \P_{J, X_J, X_J'}(\cF^c)^{1/2}\ls e^{-\Omega(n)}\, .	\end{equation}
We now consider the contribution to the expectation on the right-hand-side of \eqref{eq:J-decoup} from triples $(J,X_J,X_J') \in \cF$.
For this, let $w=w(X):=\frac{A^{-1}\Xt}{\mu_1}$ and assume $(J,X_J,X_J') \in \cF$.  By Lemma \ref{lem:char-fcn-bound}, we have 
	\begin{equation}\label{eq:bnd-dist-to-integer} 
\big|\E_{X_I} e^{   4\pi i \t \la X_I, w\ra +\la X_I,2u \ra }\big| \ls \exp\left(-c \min_{r \in [1,c^{-1}]} \|2r\theta w_I\|_{\T}^2\right). \end{equation}
	Note that $\|w_I\|_2\leq \|\Xt\|_2\leq c^{-1}\sqrt{n}$, by the definition of $\mu_1 = \sigma_{\max}(A^{-1})$ and line \eqref{it:cF-norm} in the definition of 
	$\cF(A)$. 
	
	Now, from property \eqref{it:cF-LCD} in that definition and by the hypothesis $\delta > e^{-c_1 n}$, we may choose $c_1 > 0$ small enough so that  
	\[D_{\alpha,\gamma}(w_I/\|w_I\|_2)\geq 2c^{-2}n^{1/2} /\delta \geq  2c^{-1}\|w_I\|_2/\delta .\]
	By the definition of the least common denominator, for $|\theta| \leq 1/\delta $ we have
	 \begin{equation}\label{eq:lcd-calcs}
	 \min_{r \in [1,c^{-1}]} \| 2r\theta w_I \|_{\T} =   \min_{r \in [1,c^{-1}]} \left\|2r\theta\|w_I\|_2\cdot  \frac{w_I}{\|w_I\|_2}\right\|_{\T}
	 \geq \min\left\lbrace \gamma\theta\|w_I\|_2, \sqrt{\alpha |I|}\right\rbrace.
	 \end{equation}
	So for $|\theta|\leq 1/\delta$ we use \eqref{eq:lcd-calcs} in \eqref{eq:bnd-dist-to-integer} to bound the right-hand-side of \eqref{eq:J-decoup} as
	\begin{equation}\label{eq:sm-ball} 
	\E_{J,X_J,X_J'}^{\cF} e^{\langle (X + X')_J,u \rangle  } \cdot \left|\E_{X_I}\, e^{4\pi i \theta\langle w,X_I\rangle+2\langle X_I,u\rangle} \right|  \ls \E_{J,X_J,X_J'}^{\cF}\, e^{\langle (X + X')_J,u \rangle  }e^{- c \min\{\gamma^2\theta^2\|w_I\|_2^2,\alpha |I|\}} . 
\end{equation}
We now use that $(J,X_J,X_J') \in \cF$ to see that $w \in \Inc(\delta,\rho)$ and that we chose $\mu$ to be sufficiently small, compared to $\rho,\delta$, to guarantee that 
\[ \|w\|_2 \leq C\|w_I\|_2, \] for some $C > 0$ (see \eqref{eq:mu-choice}).
Thus the right-hand-side of \eqref{eq:sm-ball} is
	\[ \ls \E_{J,X_J,X_J'}^{\cF} e^{\la (X + X')_J,u \ra } e^{-c'  \theta^2 \|w \|_2^2} + e^{-\Omega( n)} \, . \]

Combining this with \eqref{eq:sm-ball},  \eqref{eq:J-decoup} obtains the desired bound in the case in the case $(J,X_J,X'_J) \in \cF$. Combining this with 
\eqref{eq:Fc} completes the proof of Lemma~\ref{lem:small-ball-LD}.
	
\end{proof}

\section{Preparation for the ``Base step'' of the iteration}\label{sec:baseprep}

As we mentioned at \eqref{eq:sketch-1}, Vershynin \cite{vershynin-invertibility}, gave a natural way of bounding the least singular value of a random symmetric matrix:  
\[ \PP( \sigma_{\min}(A_{n+1}) \leq \eps n^{-1/2} ) 
\ls \sup_{r \in \R} \PP_{A_n,X}\big( |\la A_n^{-1}X, X \ra - r| \leq  \eps \|A_n^{-1}X\|_2 \big)\, ,  \]
 where we recall that $A_n$ is obtained from $A_{n+1}$ by deleting its first row and column.
  The main goal of this section is to prove the following lemma which tells us that we may intersect with the event $\sigma_{\min}(A_{n}) \geq \eps n^{-1/2}$ in the probability on the right-hand-side, at a loss of $C\eps$. This will be crucial for the base step in our iteration, since the bound we obtain on $\PP( \sigma_{\min}(A_{n+1}) \leq \eps n^{-1/2} )$ deteriorates as $\s_{\min}(A_n)$ decreases.
 
 \begin{lemma}\label{lem:distance-conditioned}
For $B > 0$, $\z \in \G_B$, let $A_{n+1} \sim \Sym_{n+1}(\zeta) $ and let $X \sim \col_n(\zeta)$. Then 
	$$\P\left(\sigma_{\min}(A_{n+1}) \leq \eps n^{-1/2} \right) 
\ls  \eps  + \sup_{r \in \R}\, \P\left(\frac{|\langle A_n^{-1}X, X\rangle  - r|}{ \|A_n^{-1} X \|_2} \leq C \eps , \sigma_{\min}(A_{n}) \geq \eps n^{-1/2} \right) + e^{-\Omega(n)} \,,$$
	for all $\eps>0$. Here $C > 0$ depends only on $B$.
\end{lemma}


We deduce Lemma~\ref{lem:distance-conditioned} from a geometric form of the lemma which we state here.  Let $X_j$ denote the $j$th column of $A_{n+1}$, let 
\[ H_j = \Span\{ X_1,\ldots,X_{j-1},X_{j+1},\ldots,X_{n+1}\} \text{ and }d_j(A_{n+1}) := \dist(X_j,H_j).\]
We shall prove the following ``geometric'' version of Lemma~\ref{lem:distance-conditioned}. 
\begin{lemma}\label{lem:distance-conditioned0}
For $B>0$, $\zeta \in \G_B$, let $A_{n+1} \sim \Sym_{n+1}(\zeta)$. Then for all $\eps>0$,
	 \[
	\P(\sigma_{\min}(A_{n+1}) \leq \eps n^{-1/2} ) \ls \eps  + \P\left(d_1(A_{n+1}) \leq C\eps \text{ and } \sigma_{\min}(A_n) \geq \eps n^{-1/2} \right) + e^{-\Omega(n)}\,,
	\] where $C > 0$ depends only on $B$. 
	\end{lemma}
The deduction of Lemma~\ref{lem:distance-conditioned} from Lemma~\ref{lem:distance-conditioned0} is straight-forward given the ideas from \cite{vershynin-invertibility}; so we turn to discuss the proof of Lemma~\ref{lem:distance-conditioned0}.
	
For this, we want to intersect the event $\sigma_{\min}(A_{n+1}) \leq \eps n^{-1/2}$  
with the event $\sigma_{\min}(A_n) \geq \eps n^{-1/2}$, where we understand
$A_n$ to the be principal minor $A_{n+1}^{(n+1)}$ of $A_{n+1}$. To do this we first consider the related ``pathological'' event 
\[ \cP := \left\lbrace
 \sigma_{\min}(A_{n+1}^{(i)})\leq \eps n^{-1/2} \text{ for at least } cn \text{ values of }  i \in [n+1] 
\right\rbrace \] and then split our probability of interest into the sum
\begin{equation} \label{eq:split} \P(\sigma_{\min}(A_{n+1}) \leq \eps n^{-1/2} \cap \cP ) 
+ \P(\sigma_{\min}(A_{n+1}) \leq \eps n^{-1/2} \cap \cP^c) , \end{equation} and work with each term separately. Here $c = c_{\rho,\delta}/2$, where $c_{\rho,\delta}$ is the constant defined in Section~\ref{sec:quasi-randomness}.

We deal with the second term on the right-hand-side by showing
\begin{equation}\label{eq:distance-first-term}  \P(\sigma_{\min}(A_{n+1}) \leq \eps n^{-1/2} \cap \cP^c) \leq \PP( d_1(A_{n+1}) \ls \eps \text{ and } \sigma_{\min}(A_n) \geq \eps n^{-1/2} ) + e^{-\Omega(n)}\, ,\end{equation}
by a straight-forward argument in a manner similar to Rudelson and Vershynin in \cite{RV}. We then deal with first term on the right-hand-size of \eqref{eq:split} by showing that 
\begin{equation} \label{eq:distance-second-term} \P(\sigma_{\min}(A_{n+1}) \leq \eps n^{-1/2} \cap \cP ) \ls \eps  + e^{-\Omega(n)}.\end{equation} Putting these two inequalities together then implies Lemma~\ref{lem:distance-conditioned0}. 

\subsection{Proof of the inequality at \eqref{eq:distance-first-term}}

Here we prove \eqref{eq:distance-first-term} in the following form. 

\begin{lemma}\label{lem:distance-first-term} For $B>0$, $\z \in \Gamma_B$, let $A_{n+1} \sim \Sym_{n+1}(\z)$. Then, for all $\eps>0$, we have 
\[\P(\sigma_{\min}(A_{n+1}) \leq \eps n^{-1/2} \cap \cP^c) \ls \PP\big( d_1(A_{n+1}) \ls \eps \text{ and } \sigma_{\min}(A_n) \geq \eps n^{-1/2} \big) + e^{-\Omega(n)}.\]
\end{lemma}

For this we use a basic but important fact which is at the heart of the geometric approach of Rudelson and Vershynin (see, e.g.,~\cite[Lemma 3.5]{RV}).
\begin{fact}\label{fact:dist-leastsing}
	Let $M$ be an $n \times n$ matrix and $v$ be a unit vector satisfying $\| M v \|_2 = \sigma_{\min}(M)$.  Then
	\[
	\sigma_{\min}(M) \geq |v_j| \cdot d_j(M) \quad \text{ for each } j \in [n]\, .
	\]
\end{fact}

We are now ready to prove the inequality mentioned at \eqref{eq:distance-first-term}.

\begin{proof}[Proof of Lemma \ref{lem:distance-first-term}]
We rule out another pathological event: Let $v$ denote a unit eigenvector corresponding to the least singular value of $A_{n+1}$ and let $\cC$ denote
the event that $v$ is $(\rho,\delta)$-compressible\footnote{See Section~\ref{sec:defsAndPrelims} for definition and Section~\ref{sec:quasi-randomness} for choice of $\delta,\rho$.}. By Lemma~\ref{lem:V-compressible}, $\P(\cC)\leq e^{-\Omega(n)}$. Thus 
\begin{equation}\label{eq:dist-lem-quasi-rand} \PP( \s_{\min}(A_{n+1}) \leq \eps n^{-1/2} \text{ and } \cP^c) \leq \PP( \s_{\min}(A_{n+1}) \leq \eps n^{-1/2} \text{ and } \cC^c \cap \cP^c)  + e^{-\Omega(n)}.  \end{equation}
We now look to bound this event in terms of the distance of the column $X_j$ to the subspace $H_j$, in the style of \cite{RV}. For this, we define
\[ S := \{j :  d_j(A_{n+1}) \leq \eps/c_{\rho, \delta} \text{ and } \sigma_{\min}(A_{n+1}^{(j)}) \geq \eps n^{-1/2}  \} . \]
We now claim  
\begin{equation}\label{eq:claim-in-distance1}
\{ \s_{\min}(A_{n+1}) \leq \eps n^{-1/2} \} \cap \cC^c \cap \cP^c \Longrightarrow |S| \geq c_{\rho,\delta} n/2 .\end{equation}
To see this, fix a matrix $A$ satisfying the left-hand-side of \eqref{eq:claim-in-distance1} and let $v$ be a 
eigenvector corresponding to the least singular value. Now, since 
$v$ is not compressible, there are $\geq c_{\rho,\delta} n$ values of $j \in [n+1]$ 
for which $|v_j| \geq c_{\rho,\delta}n^{-1/2}$.
Thus, Fact~\ref{fact:dist-leastsing} immediately tells us there are $\geq c_{\rho,\delta} n$ values of $j \in [n+1]$ for which 
$d_{j}(A) \leq \eps/c_{\rho,\delta}$. Finally, by definition of $\cP^c$,
at most $c_{\rho,\delta}n/2$ of these values of $j$ satisfy 
$\sigma_{n+1}(A^{(j)}) \leq \eps n^{-1/2}$ and so \eqref{eq:claim-in-distance1} is proved. 
  
We now use \eqref{eq:claim-in-distance1} along with with Markov's inequality to bound
	\begin{equation}\label{eq:dist-cond-1}
	\P(\sigma_{\min}(A_{n+1}) \leq \eps n^{-1/2} \text{ and } \cC^c \cap \cP^c) \leq \PP( |S| \geq c_{\rho, \delta}n/2 ) \leq \frac{2}{c_{\rho, \delta}n} \EE |S|.\end{equation}
Now by definition of $S$ and symmetry of the coordinates, we have
\begin{align*}
	 \label{eq:dist-cond-2}
\EE |S| &= \sum_j \PP\big( d_j(A_{n+1}) \leq \eps/c_{\rho, \delta}, ~\sigma_{\min}(A_{n+1}^{(j)}) \geq \eps n^{-1/2} \big) \\
&= n\cdot \P\big(d_1(A_{n+1}) \leq \eps/c_{\rho, \delta},~\sigma_{\min}(A_{n+1}^{(1)}) \geq \eps n^{-1/2} \big)\, .
\end{align*}	Putting this together with \eqref{eq:dist-cond-1} and \eqref{eq:claim-in-distance1} finishes the proof.
\end{proof}

\subsection{Proof of the inequality at \eqref{eq:distance-second-term}}

We now prove the inequality discussed at \eqref{eq:distance-second-term} in the following form.
\begin{lemma}\label{lem:poorly-behaved minors} For $B>0$, $\z \in \Gamma_B$, let $A_{n+1} \sim \Sym_{n+1}(\z)$. Then, for all $\eps>0$, we have 
	\begin{equation}\label{eq:badminors}
\P\left(\sigma_{\min}( A_{n+1}) \leq \eps n^{-1/2} \text{ and } \cP \right) \ls \eps + e^{-\Omega(n)}\,.\end{equation}
\end{lemma}

For the proof of this lemma we will need a few results from the random matrix literature. The first such result is a more sophisticated version of  
Lemma~\ref{lem:V-compressible}, which tells us that the mass of the eigenvectors of $A$
does not ``localize'' on a set of coordinates of size $o(n)$. The theorem we need, 
due to Rudelson and Vershynin (Theorem 1.5 in \cite{RV-nogaps}), tells us that the mass of the eigenvectors of our random matrix does not ``localize'' on a set of coordinates of size $(1-c)n$, for any fixed $c >0$. We state this result in a way to match our application.

\begin{theorem} \label{lem:no-gaps}
For $B>0$, $\z \in \Gamma_B$, let $A \sim \Sym_{n}(\zeta)$ and let $v$ denote the unit eigenvector of $A$ corresponding to the least singular value of $A$.
Then there exists $c_2>0$ such that for all sufficiently small $c_1>0$ we have 
$$\P\big(\,  |v_j| \geq (c_2c_1)^6 n^{-1/2} \text{ for at least }  (1-c_1)n  \text{ values of } j \big)  \geq 1- e^{-c_1 n}\, , $$
for $n$ sufficiently large.
\end{theorem}

We also require an elementary, but extremely useful, fact from linear algebra. This fact is a key step in the work of Nguyen, Tao and Vu on eigenvalue repulsion in random matrices (see \cite[Section 4]{nguyen-tao-vu-repulsion}); we state it here in a form best suited for our application.

\begin{fact}\label{fact:repulsion}
	Let $M$ be a $n\times n$ real symmetric matrix and let $\l$ be an eigenvalue of $M$ with corresponding unit eigenvector $u$. Let $j\in [n]$ and let 
	$\l'$ be an eigenvector of the minor $M^{(j)}$ with corresponding unit eigenvector $v$.  Then 
	\[ |\langle v, X^{(j)} \rangle| \leq |\lambda - \lambda'|/ |u_j |,\] where $X^{(j)}$ is the $j$th column of $M$ with the $j$th entry removed. 
\end{fact}
\begin{proof}
	Without loss of generality, take $j = n$ and express $u = (w,u_{n})$ 
where $w \in \R^{n-1}$.  Then we have $(M^{(n)} - \lambda I )w + X^{(n)} u_{n} = 0$.  Multiplying on the left by $v^T$ yields \[|u_{n} \langle v,X^{(n)} \rangle | = |\lambda -\lambda'| |\langle v,w\rangle | \leq |\lambda- \lambda'|\,.\] \end{proof}

We shall also need the inverse Littlewood-Offord theorem of Rudelson and Vershynin~\cite{RV}, which we have stated here in simplified form. Recall that $D_{\alpha,\gamma}(v)$ is the least common denominator of the vector $v$, as defined at \eqref{eq:lcd-def}.

\begin{theorem}\label{lem:RVsimple}
For $n\in \N$, $B>0$, $\gamma,\alpha\in (0,1)$ and $\eps > 0$, let $v\in \S^{n-1}$ satisfy $D_{\alpha,\gamma}(v)> c\eps^{-1}$ and let $X \sim \col_n(\z)$,
where $\zeta \in \G_B$. Then 
\[
\P(|\langle X, v \rangle|\leq \eps)\ls \eps + e^{-c\alpha n}\, .
\] Here $c >0$ depends only on $B$ and $\g$. \end{theorem}

We are now in a position to prove Lemma~\ref{lem:poorly-behaved minors}.

\begin{proof}[Proof of Lemma~\ref{lem:poorly-behaved minors}] Let $A$ be an instance of our random matrix and let $v$ be the unit eigenvector corresponding to the least singular value of $A$. Let $w_j = w(A^{(j)})$ denote a unit eigenvector of $A^{(j)}$ corresponding to the least singular value of $A^{(j)}$. 

We introduce two ``quasi-randomness'' events $\cQ$ and $\cA$, that will hold with probability $1-e^{\Omega(n)}$. Indeed, define 
\[ \cQ_j = \{ D_{\alpha,\gamma}(w_j)\geq e^{c_3 n} \} \text{ for all } j \in [n+1] \text{ and set } \cQ = \bigcap \cQ_j . \] Here $\alpha, \gamma, c_3$ are chosen according to Lemma~\ref{lem:cE-exp}, which tells us that $\P(\cQ^c) \leq e^{-\Omega(n)}$. Define 
\[S_1 = \{j : \sigma_n(A_{n+1}^{(j)}) \leq \eps n^{-1/2} \ \}\, \text{  and  } \, S_2 = \{ j : |v_j| \geq (cc_2/2)^6n^{-1/2}  \}.\] 
Note that $\cP$ holds exactly when $|S_1| \geq cn $. Let $\cA$ be the ``non-localization'' event that $|S_2| \geq (1-c/2)n$. By Theorem~\ref{lem:no-gaps}, we  have $\P(\cA^c)\leq e^{-\Omega(n)}$. Here $c/2 = c_{\rho,\delta}/4$. 
Now, if we let $X^{(j)}$ denote the $j$th column of $A$ with the $j$th entry removed, we define 
\[ T = \{ j : |\la w_j, X^{(j)}\ra| \leq C\eps \}, \]
where $C = 2^7/(c_2c)^6$. We now claim 
\begin{equation}\label{eq:claim-in-distance2} \{ \sigma_{\min}( A) \leq \eps n^{-1/2} \} \cap \cP \cap \cA  \Longrightarrow |T| \geq cn/2. \end{equation}
To see this, first note that if $\cP \cap \cA$ holds then $|S_1\cap S_2| \geq cn/2$. Also, for each $j \in S_1 \cap S_2$  we may apply Fact~\ref{fact:repulsion} to see that $|\la w_j, X^{(j)}\ra| \leq C \eps$ since $j$ is such that 
$\sigma_{\min}( A^{(j)}) \leq \eps n^{-1/2}$ and $\sigma_{\min}(A) \leq \eps n^{-1/2}$. This proves \eqref{eq:claim-in-distance2}.

To finish the proof of Lemma~\ref{lem:poorly-behaved minors}, we define the random variable
\[ R = n^{-1} \sum_j \1\left(  |\la w_j , X^{(j)} \ra| \leq C\eps \text{ and } \cQ_j \right),\] 
and observe that $ \PP(\sigma_{\min}(A_{n+1}) \leq \eps n^{-1/2} \cap \cP ) $ is at most
	  \[ \PP(\sigma_{\min}(A_{n+1}) \leq \eps n^{-1/2} \text{ and } \cA \cap \cQ \cap \cP )  + e^{-\Omega(n)} \leq \PP( R \geq c/4 ) + e^{-\Omega(n)} .  \]
We now apply Markov and expand the definition of $R$ to bound
\[\PP( R \geq c/4) \ls n^{-1} \sum_{j} \EE_{A^{(j)}_{n+1}}\PP_{X^{(j)}}\left( |\langle w_j , X^{(j)} \rangle| \leq C\eps \cap \cQ_j \right) \ls \eps + e^{-\Omega(n)},\]
 where the last inequality follows from the fact that $X^{(j)}$ is independent of the event $Q_j$ and $w_j$ and therefore we may put the property $\cQ_j$ to use
 by applying the inverse Littlewood-Offord theorem of Rudelson and Vershynin, Theorem~\ref{lem:RVsimple}.\end{proof}

\subsection{Proofs of Lemma~\ref{lem:distance-conditioned0} and Lemma~\ref{lem:distance-conditioned}}

All that remains is to put the pieces together and prove Lemma~\ref{lem:distance-conditioned0} and Lemma~\ref{lem:distance-conditioned}. 
\begin{proof}[Proof of Lemma~\ref{lem:distance-conditioned0}]
As we saw at \eqref{eq:split} we simply express $\PP( \s_{\min}(A_{n+1})\leq \eps n^{-1/2} )$ as 
\[\P(\sigma_{\min}(A_{n+1}) \leq \eps n^{-1/2} \text{ and } \cP ) 
+ \P(\sigma_{\min}(A_{n+1}) \leq \eps n^{-1/2} \text{ and } \cP^c), \]
and then apply Lemma~\ref{lem:poorly-behaved minors} to the first term and Lemma~\ref{lem:distance-first-term} to the second term. 
\end{proof}

\begin{proof}[Proof of Lemma~\ref{lem:distance-conditioned}]
If we set $a_{1,1}$ to be the first entry of $A = A_{n+1}$ then, by \cite[Prop.\ 5.1]{vershynin-invertibility}, we have that 
$$d_1(A_{n+1}) = \frac{|\langle A^{-1} X, X \rangle - a_{1,1}  |}{\sqrt{1  + \|A^{-1}X\|_2^2 } }\,.$$
	Additionally, by \cite[Prop.\ 8.2]{vershynin-invertibility}, we have $\|A^{-1}X\|_2 > 1/15$ with probability at least $1 - e^{-\Omega(n)}$.  Replacing $a_{1,1}$ with $r$ and taking a supremum completes the proof of Lemma~\ref{lem:distance-conditioned}.
\end{proof}

\section{Eigenvalue crowding (and the proofs of Theorem~\ref{thm:simple-spectrum} and Theorem~\ref{th:repulsion})} \label{sec:evalcrowding}

The main purpose of this section is to prove the following theorem which gives an upper-bound on the probability that $k \geq 2$ eigenvalues of a random matrix fall in an interval of length $\eps$. The case $\eps = 0$ case of this theorem tells us that the probability that a random symmetric matrix has \emph{simple} spectrum  (that is, has no repeated eigenvalue) is $1-e^{-\Omega(n)}$, which is sharp
and confirms a conjecture of Nguyen, Tao and Vu \cite{nguyen-tao-vu-repulsion}. 

Given an $n\times n$ real symmetric matrix $M$, we let $\lambda_1(M)\geq\ldots\geq \lambda_n(M)$ denote its eigenvalues.

	\begin{theorem} For $B>0$, $\z \in \G_B$, let $A_{n+1} \sim \Sym_{n+1}(\zeta) $. Then for each $j \leq cn$ and all $\eps \geq 0$ we have 
		$$\max_{k \leq n-j} \, \P( |\lambda_{k+j}(A_n) - \lambda_{k}(A_n)| \leq \eps n^{-1/2}   ) \leq \left(C\eps \right)^j + 2e^{-cn} \, ,$$
		where $C,c>0$ are constants depending on $B$.
	\end{theorem}

We suspect that the bound in Lemma~\ref{th:repulsion} is actually \emph{far} from the truth, for $\eps > e^{-cn}$ and $j \geq 1 $. In fact, one expects \emph{quadratic} dependence on $j$ in the exponent of $\eps$. This type of dependence was recently confirmed by Nguyen \cite{nguyen-overcrowding} for $\eps > e^{-n^{c}}$.  

For the proof of Lemma~\ref{th:repulsion}, we remind the reader that if $u \in \R^n \cap \Inc(\rho,\delta)$ then
at least $c_{\rho,\delta}n$ coordinates of $u$ have absolute value at least $c_{\rho,\delta}n^{-1/2}$.

In what follows, for a $n \times n$ symmetric matrix $A$, we use the notation $A^{(i_1,\ldots, i_r)}$ to refer to the minor of $A$ for which the rows and columns
indexed by $i_1,\ldots,i_r$ have been deleted. We also use the notation $A_{S \times T}$ to refer to the $|S| \times |T|$ submatrix of $A$ defined
by $(A_{i,j})_{i \in S, j\in T}$.

The following fact contains the key linear algebra required for  the proof of Theorem~\ref{th:repulsion}.

\begin{fact}\label{fact:interlacing} For $1\leq k +j < n$, let $A$ be a $n \times n$ symmetric matrix for which 
\[ | \l_{k+j}(A) - \l_k(A)| \leq \eps n^{-1/2}. \]
Let $(i_1,\ldots,i_j) \in [n]^j$ be such that $i_1,\ldots, i_j$ are distinct. 
Then there exist unit vectors $w^{(1)},\ldots,w^{(k)}$ for which 
\[ \la w^{(r)}, X_r \ra \leq (\eps n^{-1/2} ) \cdot (1/|w_{i_r}^{(r-1)}|), \]
where $X_r \in \R^{n-r} $ is the $i_r$th column of $A$ with coordinates indexed by $i_1,\ldots,i_r$ removed. That is,
$X_r := A_{ [n] \setminus \{i_1,\ldots , i_r \} \times \{i_r\} }$ and 
$w^{(r)}$  is a unit eigenvector corresponding to $\l_{k}(A^{(i_1,\ldots, i_r)})$.
\end{fact}
\begin{proof}
For $(i_1,\ldots,i_j)\in [n]^j$, define the matrices $M_0,M_1,\ldots,M_j$ by setting $M_r = A^{(i_1,\ldots,i_r)}$ for $r = 1,\ldots, j$ and
then $M_0 := A$. Now if 
\[ |\lambda_{k+j}(A) - \lambda_{k}(A)| \leq \eps n^{-1/2},\] then Cauchy's interlacing theorem implies
\[ |\lambda_{k}(M_r) - \lambda_k(M_{r-1})| \leq \eps n^{-1/2}, \] for all $r = 1,\ldots,j$.  
So let $w^{(r)}$ denote a unit eigenvector of $M_r$ corresponding to eigenvalue $\lambda_k(M_r)$. Thus, by Fact \ref{fact:repulsion}, we see that
\[ |\la w^{(r)} , X_r \ra| \leq (\eps n^{-1/2} ) \cdot (1/|w^{(r-1)}_{i_r}| ),\]
for $r=1, \ldots, j$, where $X_r \in \R^{n-r}$ is the $i_r$th column of $M_{r-1}$, with the diagonal entry removed. 
In other words, $X_r \in \R^{n-r} $ is the $i_r$th column of $A$ with coordinates indexed by $i_1,\ldots,i_r$ removed. This completes the proof of Fact~\ref{fact:interlacing}.
\end{proof}

\begin{proof}[Proof of Theorem \ref{th:repulsion}]
	Note may assume that $\eps > e^{-cn}$; the general case follows by taking $c$ sufficiently small.  
Now, define $\cA$ to be the event that \emph{all} unit eigenvectors $v$ of \emph{all} $\binom{n}{j}$ of the minors $A^{(i_1,\ldots,i_j)}_n$ lie in $\Inc(\rho,\delta)$ and satisfy $D_{\alpha, \gamma}(v)>e^{c_3 n}$, where $\alpha, \gamma, c_3$ are chosen according to Lemma~\ref{lem:cE-exp}. Note that by Lemma~\ref{lem:cE-exp} and Lemma~\ref{lem:V-compressible}, we have 
\[\P(\cA^c)  \leq  \binom{n}{j+1} e^{-\Omega(n)} \leq n\left(\frac{en}{j} \right)^{j} e^{-\Omega(n)} \ls e^{-cn},\] by taking $c$ small enough, so that 
$j\log (en/j) < cn$
is smaller than the $\Omega(n)$ term.

With Fact~\ref{fact:interlacing} in mind, we define the event, $ \cE_{i_1,\ldots,i_j}$,  for each $(i_1,\ldots,i_j) \in [n]^j$, $i_r$ distinct, to be the event that 
	$$|\langle w^{(r)}, X_r \rangle | \leq \eps/c_{\rho,\delta} \quad \text{for all } r\in [j]\, ,$$
where $X_r \in \R^{n-r} $ is the $i_r$th column of $A$ with coordinates indexed by $i_1,\ldots,i_r$ removed and 
$w^{(r)}$  is a unit eigenvector corresponding to $\l_{k}(A^{(i_1,\ldots, i_r)})$.	
	
If $\cA$ holds then each $w^{(r)}$ has at least $c_{\rho,\delta}n$ coordinates with absolute value at least $c_{\rho,\delta}n^{-1/2}$. Thus, if additionally we have 
\[ |\lambda_{k+j}(A_n) - \lambda_k(A_n)| \leq \eps n^{-1/2}, \] Fact~\ref{fact:interlacing} tells us that $\cE_{i_1,\ldots,i_j}$ occurs for at least 
$(c_{\rho,\delta}n /2)^j$ tuples $(i_1,\ldots,i_j)$. 

Define $N$ to be the number of indices $(i_1,\ldots,i_j)$ for which $\cE_{i_1,\ldots,i_j}$ occurs, and note \begin{align}
	\P(\, |\lambda_{k+j}(A_n) - \lambda_k(A_n)| \leq \eps n^{-1/2} )
 &\leq  \P\big( N \geq (c_{\rho,\delta} n/2)^j \text{ and } \cA \big) + O(e^{-cn}) \\
	&\leq   \left(\frac{2}{c_{\rho,\delta}}\right)^j\P(\cE_{1,\ldots,j} \cap \cA ) + O(e^{-cn}) \label{eq:VerTrick}
	\end{align}
	where, for the second inequality, we applied Markov's inequality and used the symmetry of the events $\cE_{i_1,\ldots,i_j}$.
	
Thus we need only show that there exists $C>0$ such that $\P(\cE_{1,\ldots,j} \cap \cA ) \leq (C\eps)^j$. To use independence, we replace
each of $w^{(r)}$ with the worst case vector, under $\cA$
\begin{align} \P(\cE_{1,\ldots,j} \cap \cA )  
&\leq \max_{w_1,\ldots,w_j : D_{\alpha, \gamma}(w_i) > e^{c_3n}} \P_{X_1,\ldots,X_r}\big(\, |\langle w_r, X_r \rangle | \leq \eps/c_{\rho,\delta}  \text{ for all } r\in [j]\, \big)\\
&\leq \max_{w_1,\ldots,w_j : D_{\alpha, \gamma}(w_i) > e^{c_3n}} \prod_{r=1}^j \P_{X_r}\big(\, |\langle w_r, X_r \rangle | \leq \eps/c_{\rho,\delta}\, \big),  \leq (C\eps)^j, \label{eq:RVapp}
\end{align}
where the first inequality follows from the independence of the the vectors $\{X_r \}_{r\leq j}$ and the last inequality follows from the fact that 
$D_{\alpha,\gamma}(w_r)> e^{c_3 n}\gtrsim 1/\eps$ (by choosing $c >0$ small enough relative to $c_3$), and the Littlewood-Offord theorem of Rudelson and Vershynin, Lemma~\ref{lem:RVsimple}. Putting \eqref{eq:VerTrick} and \eqref{eq:RVapp} together completes the proof of Theorem~\ref{th:repulsion}.
\end{proof}

\vspace{2mm}

\noindent Of course, the proof of Theorem~\ref{thm:simple-spectrum} follows immediately.

\begin{proof}[Proof of Theorem~\ref{thm:simple-spectrum}] Simply take $\eps =0$ in Theorem~\ref{th:repulsion}. 
\end{proof}

\section{Properties of the spectrum} \label{ss:spectrum}

In this section we describe and deduce Lemma~\ref{lem:upper-moments} and  Corollary~\ref{cor:distortion}, which are the tools we will use to control the ``bulk '' of the eigenvalues of $A^{-1}$. Here we understand ``bulk'' relative to the spectral measure of $A^{-1}$: our interest in an eigenvalue $\l$ of $A^{-1}$ is proportional to 
its contribution to $\|A^{-1}\|_{\HS}$. Thus the behaviour of \emph{smallest} singular values of $A$ are of the highest importance for us.

For this we let $\sigma_n \leq \sigma_{n-1} \leq \cdots \leq \sigma_1$ be the singular values of $A$ and let 
$\mu_1 \geq \ldots \geq \mu_n$ be the singular values of $A^{-1}$. Of course, we have $\mu_k=1/\sigma_{n-k+1}$ for $1\leq k \leq n$.

In short, these two lemmas, when taken together, tell us that 
\begin{equation}\label{eq:sigma-huristic} \sigma_{n - k+1} \approx k n^{-1/2}, \end{equation} 
for all $n \geq k \gg 1$ in some appropriate sense.  

\begin{lemma} \label{lem:upper-moments} For $p > 1$, $B>0$ and $\z \in \Gamma_B$, let $A \sim \Sym_n(\zeta)$. There is a constant $C_p$ depending on $B,p$ so that 
	$$\E\, \left( \frac{\sqrt{n}}{\mu_k k}\right)^p   \leq C_p\,,$$
	for all $k$. 
\end{lemma}

We shall deduce Lemma~\ref{lem:upper-moments} from the ``local semicircular law'' of Erd\H{o}s, Schlein and Yau \cite{ESY-IMRN}, which gives us good control of the 
bulk of the spectrum at ``scales'' of size $\gg n^{-1/2}$.  

We also record a useful corollary of this lemma. For this, we define the function $\| \cdot \|_{\ast} $ for a $n \times n$ symmetric matrix $M$ to be
\begin{equation}\label{eq:norm-star-def} \|M\|_\ast^2 = \sum_{k = 1}^n \s_k(M)^2 (\log(1 + k))^2. \end{equation}
The point of this definition is to give some measure to how the spectrum of $A^{-1}$ is ``distorted'' from what it ``should be'', according to the heuristic at \eqref{eq:sigma-huristic}. 
Indeed if we have $\sigma_{n - k+1} = \Theta( k/\sqrt{n})$ for all $k$, say, then we have that
\[ \|A^{-1}\|_{\ast}  = \Theta( \mu_1 ). \]
Conversely, any deviation from this captures some macroscopic misbehavior on the part of the spectrum. In particular, the ``weight function'' $k \mapsto (\log(1+k))^2$ is designed to bias the smallest singular values, and thus we are primarily looking at this range for any poor behavior.

\begin{corollary}\label{cor:distortion}
	For $p > 1$, $B>0$ and $\z \in \Gamma_B$, let $A \sim \Sym_n(\zeta)$.  Then there exists constants $C_p, c_p>0$ depending on $B,p$ such that
	$$\E \left[\left(\frac{\|A^{-1}\|_\ast}{\mu_1} \right)^p  \right] \leq C_p\,.$$ 
\end{corollary}

In the remainder of this section we describe the results of Erd\H{o}s, Schlein and Yau \cite{ESY-IMRN} and deduce Lemma~\ref{lem:upper-moments}. We then deduce Corollary~\ref{cor:distortion}.

\subsection{The local semi-circular law and Lemma~\ref{lem:upper-moments}}
For $ a < b $ we define $N_A(a,b)$ to be the number of eigenvalues of $A$ in the interval $(a,b)$. One of the most fundamental results in the theory of
random symmetric matrices is the \emph{semi-circular law} which says that  
\[ \lim_{n \rightarrow \infty} \frac{N_A(a\sqrt{n},b\sqrt{n})}{n} = \frac{1}{2\pi}\int_{a}^b(4 - x^2)^{1/2}_+\,dx, \]
almost surely, where $A \sim \Sym_n(\zeta)$.

We use a powerful ``local'' version of the semi-circle law developed by Erd\H{o}s, Schlein and Yau in a series of important papers \cite{ESY-AOP,ESY-CMP,ESY-IMRN}. Their results show that the spectrum of a random symmetric matrix actually adheres surprisingly closely to the semi-circular law. In this paper, we need control on the number of eigenvalues in intervals of the form $[-t,t]$, where $1/n^{1/2} \ll t \ll n^{1/2}$. The semi-circular law predicts that 
\[ N_A(-t,t)  \approx \frac{n}{2\pi} \int_{-t n^{-1/2}}^{tn^{-1/2}}(4 - x^2)^{1/2}_+\,dx = \frac{2t n^{1/2}}{\pi}(1+o(1)). \]
Theorem 1.11 of~\cite{erdHos2011universality} makes this prediction rigorous\footnote{
Theorem 1.11 of the survey \cite{erdHos2011universality} is based on Corollary 3.2 in \cite{ESY-IMRN}. In the paper~\cite{ESY-IMRN} the results are technically stated for (complex) Hermitian random matrices. However, the same proof goes through for real symmetric matrices. This is why we cite the later survey \cite{erdHos2011universality}, where this more general version is stated.}.

\begin{theorem}\label{lem:ESY} Let $B>0$, $\z \in \Gamma_B$, and let $A \sim \Sym_n(\zeta)$. Then, for $t \in [C n^{-1/2}, n^{1/2}]$,
	\begin{equation}
		\P\left(\,  \left| \frac{N_A(-t,t)}{n^{1/2}t} - 2\pi^{-1} \right| > \pi \right) \ls  \exp\left(-c_1(t^2n)^{1/4} \right)\, ,
	\end{equation} where $C,c_1>0$ are absolute constants.
\end{theorem}

Lemma \ref{lem:upper-moments} follows quickly from Theorem~\ref{lem:ESY}. In fact we shall only use two corollaries. 

\begin{corollary}\label{cl:ESY-consequence}
	Let $B>0$, $\z \in \Gamma_B$, and let $A \sim \Sym_n(\zeta)$. Then for all $s \geq C$ and $k \in \N$ satisfying $sk \leq n$ we have $$\P\left( \frac{\sqrt{n}}{\mu_k k} \geq s\right) \ls  \exp\big(-c(sk)^{1/2}\big)\,,$$
	where $C,c>0$ are absolute constants. 
\end{corollary}
\begin{proof}
	Let $C$ be the maximum of the constant $C$ from Lemma~\ref{lem:ESY} and $\pi$.  
	If $\frac{\sqrt{n}}{\mu_k k} \geq s$ then $N_A(-sk n^{-1/2},skn^{-1/2}) \leq k$.  We now apply 
	Lemma \ref{lem:ESY} with $t = sk n^{-1/2} \geq sn^{-1/2} \geq Cn^{-1/2}$ to see that this event
	occurs with probability $\ls \exp(-c\sqrt{sk})$.
\end{proof}

An identical argument provides a similar bound in the other direction.  
\begin{corollary}\label{cl:ESY-other-direction}
	Let $B>0$, $\z \in \Gamma_B$, and let $A \sim \Sym_n(\zeta)$. Then for all  $k \in \N$ we have $$\P\left( \mu_k \geq \frac{C \sqrt{n}}{k}\right) \ls  \exp\big(-c k^{1/2}\big)\,,$$
	where $C,c>0$ are absolute constants. 
\end{corollary}

\begin{proof}[Proof of Lemma \ref{lem:upper-moments}]
	Let $ C $ be the constant from Corollary~\ref{cl:ESY-consequence}.
From the standard tail estimates on $\|A\|_{op}$, like  \eqref{eq:op-norm-bound} for example, we immediately see that for all  $k \geq n/C$ we have
	\[ \E\, \left(\frac{\sqrt{n}}{\mu_kk} \right)^p  \leq \EE_A\left( \frac{\s_{1}(A)\sqrt{n}}{k} \right)^p  = O_p((n/k)^p) = O_p(1).\]  
	Thus we can restrict our attention to the case when $k \leq n/C$.  Define the events
	$$E_1 = \left\{\frac{\sqrt{n}}{\mu_kk} \leq C\right\}, \quad E_2 = \left\{\frac{\sqrt{n}}{\mu_kk} \in [C, n/k ] \right\}, \quad E_3 = \left\{\frac{\sqrt{n}}{\mu_kk} \geq \frac{n}{k}  \right\}.$$  We may bound 
	\begin{equation}\label{eq:three-ranges} \E\, \left(\frac{\sqrt{n}}{\mu_kk} \right)^p \leq C^p+ \E\, \left(\frac{\sqrt{n}}{\mu_kk}\right)^p \1_{E_2}
		+\E \left(\frac{\sqrt{n}}{\mu_kk}\right)^p \1_{E_3} \,. \end{equation}
	To deal with the second term in \eqref{eq:three-ranges}, we use Corollary \ref{cl:ESY-consequence} to see that 
	
	\[\E\, \left(\frac{\sqrt{n}}{\mu_kk}\right)^p \1_{E_2}  \ls  \int_{C}^{n/k} ps^{p-1}e^{-c\sqrt{sk}} ds = O_p(1).\]
	To deal with the third term in \eqref{eq:three-ranges}, we note that since $n/k \geq C$ we may apply Corollary~\ref{cl:ESY-consequence}, with $s=n/k$, to conclude that $\P(E_3) \ls e^{-c\sqrt{n}}$.  
	Thus, by Cauchy-Schwarz, we have 
	\[\E \left(\frac{\sqrt{n}}{\mu_kk}\right)^p \1_{E_3}
	\leq \left(\E \left(\frac{\s_{1}\sqrt{n}}{k}\right)^{2p}\right)^{1/2} \P(E_3)^{1/2} 
	\leq O_p(1) \cdot n^{p}  e^{-c\sqrt{n}} = O_p(1), \]
	where we have used the upper tail estimate in $\sigma_1$ from \eqref{eq:op-norm-bound} to see $\EE\, \s_1^{2p} = O_p(n^{p})$.
\end{proof}

	\subsection{Deduction of Corollary~\ref{cor:distortion}}
	
	We now conclude this section by deducing Corollary~\ref{cor:distortion} from Lemma~\ref{lem:upper-moments} and Corollary~\ref{cl:ESY-other-direction}. 
	
	\begin{proof}[Proof of Corollary~\ref{cor:distortion}]
		Recall
		\[ \|A^{-1}\|_{\ast}^2 = \sum_{k = 1}^n \mu_k^2 (\log(1 + k))^2.\]
		By H\"older's inequality we may assume without loss of generality that $p \geq 2$.  Applying the triangle inequality for the $L^{p/2}$ norm gives
		\[ \left[ \E \left( \sum_{k = 1}^n \frac{\mu_k^2 (\log (1 + k))^2}{\mu_1^2}  \right)^{p/2}\right]^{2/p} \leq \sum_{k = 1}^n (\log(1 + k))^2  \E \left[\frac{\mu_k^{p}}{\mu_1^{p}}\right]^{2/p }\,. \]
Taking $C$ to be the constant from Corollary~\ref{cl:ESY-other-direction} bound \begin{align*}
			 \E \left[\frac{\mu_k^{p}}{\mu_1^{p}}\right]&\leq C^pk^{-p} \E\left[\left(\frac{\sqrt{n}}{\mu_1}\right)^p \right] + \P\left(\mu_k \geq C \frac{\sqrt{n}}{k}  \right) \lesssim C^p k^{-p}
		\end{align*}
	where we used Lemma \ref{lem:upper-moments} and Corollary~\ref{cl:ESY-other-direction} for the second inequality. Combining the previous two equations completes the proof.
	\end{proof}

\section{Controlling small balls and large deviations}\label{sec:small-ball-large-dev}

The goal of this section is to prove the following lemma, which will be a main ingredient in our iteration in Section~\ref{sec:intermediate-bounds}. We shall then use it again in the final step and proof of Theorem~\ref{thm:main}, in Section~\ref{sec:proof-main}.

\begin{lemma}\label{cor:inductive-step}
For $B >0$ and $\z \in \Gamma_B$, let $A = A_n \sim \Sym_{n}(\z)$ and let $X \sim \col_n(\z)$. Let $u\in \R^{n-1}$ be a random vector with $\|u\|_2 \leq 1$ that depends only on $A$.
Then, for $\delta, \eps > e^{-cn}$ and $s\geq 0$, we have
	\begin{align} 
	&\EE_A \sup_r\P_X\left(\frac{|\langle A^{-1}X,X\rangle -r|}{\|A^{-1}\|_{\ast}} \leq \delta,~\langle X, u\rangle\geq s,~ \frac{\mu_1}{\sqrt{n}} \leq \eps^{-1} \right)
	\nonumber \\
	&\qquad \ls  \delta e^{-s} \left[ \E_{A}  \left(\frac{\mu_1}{\sqrt{n}}\right)^{7/9} \one\left\{\frac{\mu_1}{\sqrt{n}} \leq \eps^{-1} \right\}   \right]^{6/7} + e^{-cn}\,, \label{eq:cor-inductive-step}
	\end{align} where $c>0$ depends only on $B>0$.
\end{lemma}

Note that with this lemma we have eliminated all ``fine-grained'' information about the spectrum of $A^{-1}$ and all that remains is $\mu_1$, which is the reciprocal of the least singular value of the matrix $A$. We also note that we will only need the full power of Lemma~\ref{cor:inductive-step} in Section~\ref{sec:proof-main};  until then, we will apply it with $s=0, u=0$. 

We now turn our attention to proving Lemma~\ref{cor:inductive-step}. We start with an application of Theorem~\ref{thm:invLwO}, our negative correlation theorem, which we restate here in its full-fledged form.

\begin{theorem}\label{thm:ILwO2}
For  $n \in \N$, $\alpha,\gamma \in (0,1), B > 0$ and $\mu \in (0,2^{-15})$, there are constants $c,R > 0$ depending only on $\alpha,\gamma,\mu,B$ so that the following holds.  Let $0\leq k \leq c \alpha n$ and $\eps \geq \exp(-c\alpha n)$, let $v \in \S^{n-1}$, and let $w_1,\ldots,w_k \in \S^{n-1}$ be orthogonal.
For $\zeta \in \G_B$, let $\zeta'$ be an independent copy of $\zeta$ and $Z_\mu$ a Bernoulli variable with parameter $\mu$;  let $\Xt \in \R^n$ be a random vector whose coordinates are i.i.d.\ copies of the random variable $(\z - \z')Z_\mu$.

If $D_{\alpha,\gamma}(v) > 1/\eps$ then
\begin{equation}\label{eq:thminvLwO} 
\PP_X\left( |\la \Xt, v \ra| \leq \eps\, \text{ and }\, \sum_{j = 1}^k \langle w_j, \Xt\rangle^2 \leq c k \right) \leq R \eps \cdot e^{-c k}\,.
\end{equation}
\end{theorem}

The proof of Theorem~\ref{thm:ILwO2} is provided in the Appendix. We now prove Lemma~\ref{lem:small-ball-with-LwO}.

\begin{lemma}\label{lem:small-ball-with-LwO}
	Let $A$ be a $n \times n$ real symmetric matrix with $A \in \cE$ and set $ \mu_i := \sigma_{i}(A^{-1})$, for all $i \in [n]$.
	For $B>0$, $\z\in \Gamma_B$, let $X,X' \sim \col_n(\zeta)$ be independent, let $J \subseteq [n]$ be a $\mu$-random subset with $\mu \in (0, 2^{-15})$, and set $\Xt := (X - X')_J$.
	If $k \in [1,c n]$ is such that $s \in (e^{-c n} , \mu_k/\mu_1)$ then
	\begin{equation}\label{eq:lem-small-ball-wLwO} \P_{\Xt}\left( \|A^{-1} \Xt\|_2  \leq s\mu_1 \right) \ls s e^{-ck}\,, 
	 \end{equation}
where $c > 0$ depends only on $B$.
\end{lemma}
\begin{proof}
For each $j\in [n]$ we let $v_j$ denote a unit eigenvector of $A^{-1}$ corresponding to $\mu_j$.
Using the resulting singular value decomposition of $A^{-1}$, we may express 
\[ \|A^{-1} \Xt\|^2_2  = \la  A^{-1} \Xt , A^{-1}\Xt \ra = \sum_{j=1}^n \mu_j^2 \la \Xt, v_j \ra^2 \]
and thus 
\begin{equation}\label{eq:wLwO} \P_{\Xt}\left( \|A^{-1} \Xt\|_2 \mu_1^{-1} \leq s \right)
 \leq \P_{\Xt}\left( |\langle v_1 , \Xt \rangle| \leq s \text{ and } \sum_{j = 2}^k \frac{\mu_j^2}{\mu_1^2}\langle v_j, \Xt \rangle^2 \leq s^2 \right). \end{equation}
	 We now use that $s \leq 1$ and $\mu_k/\mu_1 \leq 1$ in \eqref{eq:wLwO} to obtain
  \begin{equation} \label{eq:wLwo2} \P_{\Xt}\left( \|A^{-1} \Xt\|_2 \mu_1^{-1} \leq s \right) \leq \P_{\Xt}\left( |\langle v_1 , \Xt \rangle| \leq s \text{ and } \sum_{j = 2}^k \langle v_j, \Xt \rangle^2 \leq 1 \right) \,. \end{equation}
We now carefully observe that we are in a position to apply Theorem~\ref{thm:invLwO} to the right-hand-side of \eqref{eq:wLwo2}. The coordinates of $\Xt$ are of the form $(\zeta-\zeta')Z_{\mu}$,
where $Z_{\mu}$ is a Bernoulli random variable taking $1$ with probability $\mu \in (0,2^{-15})$ and $0$ otherwise. Also, the 
$ v_2,\ldots,v_k$ are orthogonal
and, importantly, we use that $A \in \cE$ to learn that\footnote{Recall here that the constants $\alpha,\g>0$ are implicit in the definition of $\cE$
	and are chosen so that Lemma~\ref{lem:cE-exp} holds.}  $D_{\alpha,\g}(v_1)>1/s$ by property~\eqref{it:cE-evec}, provided we choose the constant $c>0$ (in the statement of Lemma~\ref{lem:small-ball-with-LwO}) to be sufficiently small, depending on $\mu,B$. Thus we may apply Theorem~\ref{thm:invLwO} and complete the proof of the Lemma~\ref{lem:small-ball-with-LwO}.
\end{proof}

\vspace{3mm}

\noindent With this lemma in hand, we establish the following corollary of Lemma~\ref{lem:small-ball-LD}.

\begin{lemma}\label{lem:small-ball-app}
	For $B>0$ and $\zeta \in \G_B$, let $X \sim \col_n(\z)$ and let $A$ be a $n\times n$ real symmetric matrix with $A \in \cE$.  
	If $s >0$, $\delta \in (e^{-c n},1)$ and $u \in \S^{n-1}$ then
	\begin{equation} \label{eq:small-ball-app}
	\sup_r\P_{X}\big( \left|\langle A^{-1} X, X \rangle -r \right| \leq \delta \mu_1 ,\langle X, u \rangle \geq s \big) 
	\ls    \delta e^{-s}  \sum_{k = 2}^{cn}  e^{-ck}\left(\frac{\mu_1}{\mu_k} \right)^{2/3} +  e^{-cn}\, ,
	\end{equation} where $c>0$ is a constant depending only on $B$.
\end{lemma}
\begin{proof}
	We apply Lemma \ref{lem:small-ball-LD} to the left-hand-side of \eqref{eq:small-ball-app} to get
\begin{equation}\label{eq:small-ball-app2} \sup_r\P_{X}\big( \left|\langle A^{-1} X, X \rangle - r \right| \leq \delta \mu_1 ,\langle X, u \rangle \geq s \big)
\ls  \delta e^{-s}  \int_{-1/\delta}^{1/\delta} I(\t)^{1/2} \,d\t + e^{-\Omega(n)} \, , \end{equation}
where 
\[ I(\t) := \E_{J,X_J,X_J'} \exp\left( \langle (X + X')_J,u \rangle -c' \theta^2 \mu_1^{-2} \| A^{-1}(X - X')_J \|_2^2 \right)  ,\]
and  $c' = c'(B) >0 $ is a constant depending only on $B$ and $J \subseteq [n]$ is a $\mu$-random subset. Set 
\[ \Xt=(X-X')_J \qquad \text{  and  } \qquad v = A^{-1}\Xt,\] and apply H\"older's inequality
\begin{equation}\label{eq:I} I(\t) = \E_{J,X_J,X_J'} \left[e^{\langle (X + X')_J,u \rangle} e^{-c' \theta^2 \|v \|_2^2/\mu_1^2 } \right] \ls \left(\E_{\Xt} e^{-c'' \theta^2 \|v \|_2^2/\mu_1^2} \right)^{8/9}\left( \E_{J,X_J,X_J'}\, e^{9\langle (X + X')_J,u \rangle} \right)^{1/9} .\end{equation}
Thus we apply \eqref{eq:exp-moment} to see that the second term on the right-hand-side of \eqref{eq:I} is $O(1)$. Thus, for each $\t > 0$ we have
$$ I(\theta)^{9/8} \ls_{B} \E_{\Xt} e^{-c'' \theta^2 \|v\|_2^2/\mu_1^2} \leq e^{-c'' \theta^{1/5}} + \P_{\Xt}( \|v\|_2 \leq \mu_1\theta^{-9/10})\,.$$
	
As a result, we have 
$$\int_{-1/\delta}^{1/\delta} I(\t)^{1/2} \,d\theta \ls 1 + \int_{1}^{1/\delta} \P_{\Xt}(\|v\|_2 \leq \mu_1 \theta^{-9/10} )^{4/9}\,d\theta\,
\ls 1 + \int_{\delta}^{1} s^{-19/9} \P_{\Xt}(\|v\|_2 \leq \mu_1 s )^{4/9}\, ds .$$
To bound this integral, we partition $ [\delta,1] = [\delta, \mu_{c n}/\mu_1 ] \cup \bigcup_{k=2}^{c n} [\mu_{k}/\mu_1,\mu_{k-1}/\mu_1]$
and apply Lemma \ref{lem:small-ball-with-LwO} to bound the integrand depending on which interval $s$ lies in. Note this lemma is applicable since
$A \in \cE$. We obtain
\[ \int_{\mu_k/\mu_1}^{\mu_{k-1}/\mu_1}  s^{-19/9} \P_{\Xt}(\|v\|_2 \leq \mu_1 s )^{4/9} \leq e^{-ck} \int_{\mu_k/\mu_1}^{\mu_{k-1}/\mu_1}  s^{-15/9} \, ds \leq e^{-ck}(\mu_1/\mu_k)^{2/3}, \]
while 
\[ \int_{\delta}^{\mu_{c n}/\mu_1} s^{-19/9} \P_{\Xt}(\|v\|_2 \leq \mu_1 s )^{4/9}  \leq e^{-c n} \delta^{-3/2} \le e^{-\Omega(n)}. \]
Summing over all $k$ and plugging the result into \eqref{eq:small-ball-app2} completes the lemma.
\end{proof}

\vspace{2mm}

 We may now prove Lemma~\ref{cor:inductive-step} by using the previous Lemma~\ref{lem:small-ball-app} along with the properties of the spectrum of $A$ established in Section~\ref{ss:spectrum}. 

\begin{proof}[Proof of Lemma~\ref{cor:inductive-step}]
Let $\cE$ be our quasi-random event as defined in Section~\ref{sec:quasi-randomness} and let 
\[ \cE_0=\cE\cap \left\lbrace\frac{\mu_1}{\sqrt{n}} \leq \eps^{-1}\right\rbrace.\]
For fixed $A \in \cE_0$ and $u = u(A) \in \R^{n}$ with $\|u\|_2\leq 1$, we may apply Lemma \ref{lem:small-ball-app} with $\delta' =\delta \frac{\|A^{-1}\|_{\ast}}{\mu_1}$ to see that 
\[\sup_{r \in \R} \P_{X}\big( \left|\langle A^{-1} X, X \rangle -r \right| \leq \delta \|A\|_{\ast} ,\langle X, u \rangle \geq s \big)
	 \ls \delta e^{-s}  \left(\frac{\|A^{-1}\|_{*}}{\mu_1}\right) \sum_{k = 2}^{cn} e^{-ck}\left(\frac{\mu_1}{\mu_k} \right)^{2/3}  + e^{-cn}\, .\] 
By Lemma \ref{lem:cE-exp}, $\Pr_A(\cE^c)\ls \exp(-\Omega(n))$. Therefore it is enough to show that 
\begin{equation} \label{eq:corinductivestep-goal}\Ex_A^{\cE_0} \left(\frac{\|A^{-1}\|_{*}}{\mu_1}\right) 
	\left(\frac{\mu_1}{\mu_k} \right)^{2/3}
\ls k\cdot \E_{A}^{\cE_0} \left[ \left(\frac{\mu_1}{\sqrt{n}}\right)^{7/9}\right]^{6/7},\end{equation}
	for each $k \in [2,c n]$. For this, apply H\"older's inequality to the left-hand-side of \eqref{eq:corinductivestep-goal} to get
	 \[\Ex_A^{\cE_0} \left(\frac{\|A^{-1}\|_{*}}{\mu_1}\right)\left(\frac{\mu_1}{\mu_k} \right)^{2/3}\leq \Ex_A^{\cE_0} \left[\left(\frac{\|A^{-1}\|_{*}}{\mu_1}\right)^{14}\right]^{1/14}\Ex_A^{\cE_0} \left[\left(\frac{\sqrt{n}}{\mu_k}\right)^{28/3}\right]^{1/14}\Ex_A^{\cE_0} \left[\left(\frac{\mu_1}{\sqrt{n}} \right)^{7/9}\right]^{6/7}.\] We now apply Corollary~\ref{cor:distortion} to see the first term is $O(1)$ and Lemma \ref{lem:upper-moments} to see that the second term is $O(k)$. This establishes \eqref{eq:corinductivestep-goal} and thus Lemma~\ref{cor:inductive-step}.
\end{proof}

\section{Intermediate bounds: Bootstrapping the lower tail} \label{sec:intermediate-bounds}

In this short section we will use the tools developed so far to prove an ``up-to-logarithms'' version of Theorem \ref{thm:main}. In the next section, Section~\ref{sec:proof-main}, we will bootstrap this result (once again) to prove Theorem~\ref{thm:main}.

\begin{lemma} For $B >0 $, let $\zeta \in \G_B$, and let $A_n \sim \Sym_{n}(\zeta)$. Then for all $\eps >0$
	\label{lem:4/5}
	$$\P(\sigma_{\min}(A_{n}) \leq \eps n^{-1/2} ) \ls  \eps \cdot (\log \eps^{-1})^{1/2} +  e^{-\Omega(n)}\,.$$
\end{lemma}

To prove Lemma~\ref{lem:4/5}, we first prove the following ``base step'' (Lemma~\ref{lem:polynomial-LSV}) which we then improve upon in three increments,
ultimately arriving at Lemma~\ref{lem:4/5}. 

The ``base step'' is an easy consequence of Lemma~\ref{lem:distance-conditioned0} and Lemma~\ref{cor:inductive-step} and actually already improves upon 
the best known bounds on the least-singular value problem for random symmetric matrices. For this we will 
need the well-known theorem due to Hanson and Wright \cite{Hanson-Wright,Wright}. See \cite[Theorem 6.2.1]{vershynin2018high}) for a 
modern exposition.

\begin{theorem}[Hanson-Wright]\label{thm:HW}
For $B>0$, let $\z \in \G_B$, let $X \sim \col_n(\z)$ and let $M$ be a $m\times n$ matrix. Then for any $t\geq 0$, we have
\[
\P_X\big( \left| \|MX\|_2 - \|M\|_{\HS} \right| >t \big) \leq 2 \exp\left(- \frac{ct^2}{B^4\|M\|^2} \right)\, ,
\] where $c>0$ is absolute constant.
\end{theorem}

We now prove the base step of our iteration.

\begin{lemma}[Base step]\label{lem:polynomial-LSV}
	For $B>0$, let $\zeta \in \G_B$ and let $A_{n+1} \sim \Sym_{n+1}(\zeta)$. Then 
	$$\P(\sigma_{\min}(A_{n+1}) \leq \eps n^{-1/2} ) \ls \eps^{1/4} +  e^{-\Omega(n)}\,, $$
	for all $\eps >0$.
\end{lemma}
\begin{proof} As usual, we let $A := A_{n}$. By Lemma~\ref{lem:distance-conditioned}, it will be sufficient to show that for $r\in \R$, 
	\begin{align}\label{eq:Ax-small-ball-1/4}
	\P_{A,X}\left(\, \frac{|\langle A^{-1} X, X \rangle - r|}{\|A^{-1}X\|_2} \leq C\eps,\, \sigma_n(A)\geq \eps n^{-1/2} \right) \ls \eps^{1/4} + e^{-\Omega(n)}\,.
	\end{align}
By the Hanson-Wright inequality (Theorem~\ref{thm:HW}), there exists $C'>0$ so that 
	\begin{equation}\label{eq:norm-by-HS}
\P_X\big(\, \|A^{-1} X\|_2 \geq C'(\log \eps^{-1} )^{1/2} \cdot \| A^{-1} \|_{\HS}\, \big) \leq \eps\,\end{equation} 
	and so the left-hand-side of \eqref{eq:Ax-small-ball-1/4} is bounded above by
	\[
	 \eps + \P_{A,X}\left(\, \frac{|\langle A^{-1} X, X \rangle - r|}{\|A^{-1}\|_{\HS}} \leq \delta,\,  \sigma_n(A)\geq \eps n^{-1/2} \right) \, ,
	\]
	where $\delta :=C'' \eps \cdot ( \log \eps^{-1} )^{1/2}$. Now, by Lemma~\ref{cor:inductive-step} with the choice of $u=0, s=0$, we have
	\begin{equation}\label{eq:basis-step-last}
	 \P_{A,X}\left(\, \frac{|\langle A^{-1} X, X \rangle - r|}{\|A^{-1}\|_{\HS}} \leq \delta,\,  \sigma_n(A)\geq  \eps n^{-1/2} \right) 
	 \ls
	 \delta \eps^{-2/3} +  e^{-\Omega(n)}
	 \ls 
	  \eps^{1/4} +  e^{-\Omega(n)}\, ,
	\end{equation} where we have used that $\|A^{-1}\|_{\ast}\geq \|A^{-1}\|_{\HS}$. We also note that Lemma~\ref{cor:inductive-step} actually gives an upper bound on 
	$\E_A \sup_r \PP_X( \cA)$, where $\cA$ is the event on the left-hand-side of \eqref{eq:cor-ind-step-app}. Since $\sup_r \P_{A,X}(\cA) \leq \E_A \sup_r \PP_X( \cA) $, the bound \eqref{eq:basis-step-last}, and thus Lemma~\ref{lem:polynomial-LSV}, follows. 
\end{proof}

\vspace{2mm}

\noindent The next lemma is our ``bootstrapping step'': given bounds of the form 
\[ \P(\sigma_{\min}(A_n)\leq \eps n^{-1/2} ) \ls \eps^{\kappa}+ e^{-cn}, \]
this lemma will produce better bounds for the same problem with $A_{n+1}$ in place of $A_n$.

\begin{lemma}(Bootstrapping step) \label{lem:bootstrap}
For $B >0$, let $\zeta \in \Gamma_B$, let $A_{n+1} \sim \Sym_{n+1}(\z)$ and let $\k \in (0,1) \setminus \{7/10\}$. If for all $\eps>0$, and all $n$ we have 
\begin{align}\label{eq:lsv-hyp}
\P\big(\s_{\min}(A_n)\leq \eps n^{-1/2} \big )\ls \eps^{\kappa}+ e^{-\Omega(n)}\, ,
\end{align}
then for all $\eps>0$ and all $n$ we have
\[
\P(\s_{\min}(A_{n+1})\leq \eps n^{-1/2} )\ls (\log \eps^{-1} )^{1/2} \cdot \eps^{\min\left\{1, 6\kappa/7+1/3\right\}}+ e^{-\Omega(n)}\, .
\]
\end{lemma}
\begin{proof}
Let $c>0$ denote the implicit constant in the exponent on the right-hand-side of~\eqref{eq:lsv-hyp}.
Note that if $0< \eps <e^{-cn}$, by the assumption of the lemma, then we have  
\[ \P(\s_{\min}(A_{n})\leq \eps n^{-1/2}  )\ls e^{-\Omega(n)},\]
for all $n$, in which case we are done. So we may assume $\eps > e^{-cn}$.

As in the proof of the ``base step'', Lemma~\ref{lem:polynomial-LSV}, we look to apply Lemma~\ref{lem:distance-conditioned0} and Lemma~\ref{cor:inductive-step} in sequence. For this we write $A = A_n$ and bound \eqref{eq:cor-inductive-step} as in the conclusion of Lemma~\ref{cor:inductive-step}
\begin{equation}\label{eq:boot1} 
\E_{A}\,  \left(\frac{\mu_1}{\sqrt{n}}\right)^{7/9} \one\left\{\frac{\mu_1}{\sqrt{n}} \leq \eps^{-1} \right\} 
\leq \int_{0}^{\eps^{-7/9}}\P\left(\sigma_{\min}(A)\leq x^{-9/7} n^{-1/2} \right)\, dx  ,\end{equation}
where we used that $\sigma_{\min}(A)=1/\mu_1(A)$. 
Now use assumption~\eqref{eq:lsv-hyp} to see the right-hand-side of \eqref{eq:boot1} is
\begin{equation} \ls 1 + \int_{1}^{\eps^{-7/9}} (x^{-9\kappa/7}+ e^{-cn})\,dx 
	\ls \max\left\{1, \eps^{\kappa-7/9}\right\}\, . \label{eq:mu-moment} \end{equation}

	Now we apply Lemma~\ref{cor:inductive-step} with $\delta = C\eps \cdot (\log \eps^{-1} )^{1/2}$, $s=0$ and $u=0$ to see that 
	 \begin{align}
	\P_{A,X}\left(\frac{|\langle A^{-1}X,X\rangle - r |}{\|A^{-1}\|_{\HS}} \leq \delta,\, \frac{\mu_1}{\sqrt{n}} \leq \eps^{-1} \right) 
	&\ls  \max\left\{\eps, \eps^{6\kappa/7+1/3}\right\} \cdot (\log \eps^{-1} )^{1/2} +  e^{-\Omega(n)}\, , \label{eq:cor-ind-step-app}
	\end{align} 
	for all $r$. Here we used that $\|A^{-1}\|_{\HS} \leq \|A^{-1}\|_{\ast}$. 
	
	Now, by Hanson-Wright (Theorem~\ref{thm:HW}), there exists $C'>0$ such that 
	\[\P_X\big(\| A^{-1} X \|_2 \geq  C' \|A^{-1}\|_{\HS}\cdot (\log \eps^{-1} )^{1/2} \big) \leq \eps.\] 
	Thus we choose $C''$ to be large enough, so that  
	$$\P_{A,X}\left(\, \frac{|\langle A^{-1}X,X\rangle -r |}{\|A^{-1} X\|_2} \leq  C''\eps , \sigma_{n}(A) \geq \eps n^{-1/2}\, \right) 
	\ls \max\left\{\eps, \eps^{6\kappa/7+1/3}\right\} \cdot (\log \eps^{-1} )^{1/2} +  e^{-\Omega(n)} \, ,$$ for all $r$. Lemma~\ref{lem:distance-conditioned} now completes the proof of Lemma~\ref{lem:bootstrap}.
\end{proof}

\vspace{2mm}

\noindent Lemma \ref{lem:4/5} now follows by iterating Lemma~\ref{lem:bootstrap} three times.

\begin{proof}[Proof of Lemma \ref{lem:4/5}]
By Lemma~\ref{lem:polynomial-LSV} and Lemma~\ref{lem:bootstrap} we have
\[
\P(\sigma_{\min}(A)\leq \eps n^{-1/2} )\ls \eps^{13/21} \cdot (\log \eps^{-1} )^{1/2}+ e^{-\Omega(n)} \ls  \eps^{13/21-\eta}+ e^{-\Omega(n)}\, ,
\] for some small $\eta>0$. Applying Lemma~\ref{lem:bootstrap} twice more gives an exponent of $\frac{127}{147}-\frac{6}{7}\eta$ and then $1$, for $\eta$ small,
thus completing the proof.
\end{proof}

\section{Proof of Theorem~\ref{thm:main}}\label{sec:proof-main}
We are now ready to prove our main result, Theorem~\ref{thm:main}.
We use Lemma~\ref{lem:distance-conditioned} (as in the proof of Lemma \ref{lem:4/5}) and the inequality at~\eqref{eq:cEdef} to see that it is enough to prove
\begin{equation}\label{eq:sufficient}
\P^{\cE}\left(\, \frac{|\langle A^{-1}X,X\rangle-r|}{\|A^{-1}X\|_2}\leq C\eps, \text{ and } \sigma_n(A) \geq \eps n^{-1/2} \right) \ls \eps + e^{-\Omega(n)} \,,
\end{equation} 
where $C$ is as in Lemma \ref{lem:distance-conditioned} and the implied constants do not depend on $r$. Recall that $\cE$ is the quasi-random event 
defined in Section~\ref{sec:quasi-randomness}.

To prepare ourselves for what follows, we put $\cE_0 := \cE \cap \{\sigma_{\min}(A) \geq \eps n^{-1/2} \}$  and 
\[
Q(A, X):=\frac{|\langle A^{-1}X,X\rangle-r|}{\|A^{-1}X\|_2}\, \, \text{   and   } \,  Q_{\ast}(A, X):=\frac{|\langle A^{-1}X,X\rangle-r|}{\|A^{-1}\|_{\ast}}
\]
where
\[\|A^{-1}\|_{*}^2 =\sum_{k=1}^n \mu_k^{2}(\log (1 + k) )^2 \, ,\] 
as defined in Section~\ref{ss:spectrum}. We now split the left-hand-side of~\eqref{eq:sufficient} as
 \begin{align}
 \P^{\cE_0}\left( Q(A, X)\leq C\eps \right) &\leq   
\P^{\cE_0}\left(Q_{\ast}(A, X)\leq 2C\eps \right)  
+ \P^{\cE_0}\left(Q(A, X)\leq C\eps, \frac{\|A^{-1}X\|_2}{\|A^{-1}\|_{\ast}}\geq 2 \right)\,.\label{eq:two-cases}
\end{align}

We can take care of the first term easily by combining Lemma \ref{cor:inductive-step} and Lemma~\ref{lem:4/5}.

\begin{lemma}\label{lem:small-t}
	For $\eps>0$, 
$$ \P^{\cE_0}(Q_{\ast}(A, X) \leq 2 C\eps   ) \ls  \eps +  e^{-\Omega(n)}\,.$$
\end{lemma}
\begin{proof}
	Apply Lemma~\ref{cor:inductive-step}, with $\delta=2C\eps$, $u=0$ and $s=0$ to obtain
	$$\P^{\cE_0}(Q_{\ast}(A, X)\leq 2 C \eps    ) \ls   \eps \left( \E_{A} \left(\frac{\mu_1}{\sqrt{n}} \right)^{7/9} 
	\one\left\lbrace\frac{\mu_1}{\sqrt{n}}\leq \eps^{-1} \right\rbrace \right)^{6/7} + e^{-\Omega(n)}\,.$$
	By Lemma \ref{lem:4/5} and the calculation at~\eqref{eq:mu-moment}, the expectation on the right is bounded by a constant.
\end{proof}

\vspace{2mm}

We now focus on the latter term on the right-hand-side of \eqref{eq:two-cases}.
By considering the dyadic partition $ 2^j \leq \|A^{-1}X\|_2 / \|A^{-1}\|_{*} \leq 2^{j+1}$ we see the second term on the RHS of \eqref{eq:two-cases} is 
\begin{align}
 \ls  \sum_{j=1}^{\log n}\Pr^{\cE_0}\left(Q_{\ast}(A, X) \leq 2^{j+1}C \eps\,, \frac{\|A^{-1}X\|_2}{ \|A^{-1}\|_{*}}\geq 2^{j}\right)\ + e^{-\Omega(n)}\, . \label{eq:unionbound} 
\end{align}
Here we have dealt with the terms for which $j \geq \log n$ by using the fact that
$$\P_X\big(\, \|A^{-1} X\|_2 \geq \sqrt{n} \|A^{-1}\|_{\ast} \big) \ls e^{-\Omega(n)}\,, $$
which follows from Hanson-Wright and the inequality $\|A^{-1}\|_{\ast}\geq \|A^{-1}\|_{\HS}$.

We now show that the event $\|A^{-1}X\|_2 \geq t \|A^{-1}\|_\ast $ implies that $X$ must correlate with one of the eigenvectors of $A$.   
\begin{lemma}\label{lem:split-evec}
	For $t>0$, we have
\begin{align*}
\Pr_{X}\left(Q_{\ast}(A,X)\leq 2Ct \eps, \frac{ \|A^{-1}X\|_2}{ \|A^{-1}\|_{*}}\geq t\right) \leq  2\sum_{k=1}^n\Pr_X\left(Q_{\ast}(A,X)\leq 2Ct \eps , \langle X,v_k\rangle\geq t \log (1 + k)\right)
\end{align*}
where $\{v_k\}$ is an orthonormal basis of eigenvectors of $A$.  
\end{lemma}
\begin{proof}
Assume that $\|A^{-1}X\|_2 \geq t \|A^{-1}\|_{*}$ and use the singular value decomposition associated with $\{v_k\}_k$ to write
 \[ t^2\sum_{k} \mu_i^2(\log(k+1))^2 =  t^2\|A\|^2_{\ast} \leq \| A^{-1} X\|_2^2 = \sum_{k} \mu_k^{2} \langle v_k,X\rangle^2. \] Thus
 $$\{\|A^{-1} X \|_2 \geq t \| A^{-1}\|_\ast \} \subset \bigcup_{k} \big\lbrace |\langle X, v_k \rangle| \geq t \log(k+1)   \big\rbrace \,.$$
 To finish the proof of Lemma~\ref{lem:split-evec}, we union bound and treat the case of $-X$ the same as $X$ (by possibly changing the sign of $v_k$) at the cost of a factor of $2$. \end{proof}

\begin{proof}[Proof of Theorem \ref{thm:main}]
	
	Recall that it suffices to establish \eqref{eq:sufficient}.  Combining  \eqref{eq:two-cases} with Lemma~\ref{lem:split-evec} and Lemma~\ref{lem:small-t} tells us
	that
	\begin{equation}\label{eq:pfmain1}
	\P^{\cE_0}\left(Q(A,X) \leq C\eps \right) 
	 \ls \eps + 2\sum_{j=1}^{\log n}\sum_{k = 1}^n \P^{\cE_0}\left(Q_{\ast}(A,X)\leq 2^{j+1}C \eps , \langle X,v_k\rangle\geq 2^j \log(1 + k)\right) +  e^{-\Omega(n)} \, .
	\end{equation}
	 We now apply Lemma~\ref{cor:inductive-step} for all $t>0$, with $\delta= 2Ct\eps$, $s=t \log(k+1)$ and $u=v_k$ to see that,
	\begin{equation} \label{eq:pfmain2} \P^{\cE_0}\big( Q_{\ast}(A,X)\leq 2Ct \eps, \langle X,v_k\rangle\geq t \log(1 + k) \big)
	 \ls \eps t (k+1)^{-t} \cdot I^{6/7}+ e^{-\Omega(n)}\, .\end{equation}
	where \[ I := \E_{A}  \left(\frac{\mu_1(A)}{\sqrt{n}}\right)^{7/9} \one\left\{\frac{\mu_1(A)}{\sqrt{n}} \leq \eps^{-1} \right\} . \]
	Using \eqref{eq:pfmain2} in \eqref{eq:pfmain1} yields
	\[ \P^{\cE_0}(Q(A,X) \leq C\eps ) \ls  \eps I^{6/7} \sum_{j=1}^{\log n}\sum_{k = 1}^{n} 2^j(k+1)^{-2^j} + e^{-\Omega(n)} \ls \eps \cdot I^{6/7} + e^{-\Omega(n)}, \]
	since $\sum_{j=1}^{\infty}\sum_{k = 1}^{\infty} 2^j(k+1)^{-2^j} = O(1)$. Now we write 
	\[ I = \E_{A}\, \left(\frac{\mu_1(A)}{\sqrt{n}}\right)^{7/9} \one\left\{\frac{\mu_1(A)}{\sqrt{n}} \leq \eps^{-1} \right\}   \leq \int_{0}^{\eps^{-7/9}}\P\left(\sigma_{\min}(A)\leq x^{-9/7} n^{-1/2} \right)\, dx\]
and apply Lemma \ref{lem:4/5} to see	
\[ \int_{0}^{\eps^{-7/9}}\P\left(\sigma_{\min}(A)\leq x^{-9/7} n^{-1/2} \right)\, dx \ls \int_{1}^{\infty} s^{-9/7} \, ds + 1 \ls 1. \]
	 Thus, Lemma~\ref{lem:distance-conditioned} completes the proof of Theorem~\ref{thm:main}.
\end{proof}

\section*{Acknowledgments}
The authors thank Rob Morris for comments on the presentation of this paper. The authors also thank the anonymous referees for many useful comments and a simplification to the proof of Corollary~\ref{cor:distortion}. Marcelo Campos is partially supported by CNPq. Matthew Jenssen is supported by a UKRI Future Leaders Fellowship MR/W007320/1. Marcus Michelen is supported in part by NSF grants DMS-2137623 and DMS-2246624.

\begin{appendices}
\renewcommand{\thesection}{\Roman{section}}

	\section{Introduction to the appendices}
In these appendices, we lay out the proof of Theorem~\ref{thm:qLCD},
	the ``master quasirandomness theorem,'' which we left unproved in the main body of the paper, and the proof of Theorem~\ref{thm:ILwO2}. The proofs of these results are technical adaptations of the authors' previous work on the singularity of random symmetric matrices \cite{RSM2}. The last three appendices also tie up some other loose 
	ends in the main body of the text. 
	
In particular, the proof of Theorem~\ref{thm:qLCD} is similar to the proof of the main theorem in \cite{RSM2}, with only a few tweaks and additions required to make the adaptation go through. 
	In several places, we need only update the constants and will be satisfied in pointing the interested reader to \cite{RSM2} for more detail. Elsewhere, more significant adaptations are required, and we outline these changes in full detail. As such, parts of these appendices will bore the restless expert, but we hope it will provide a useful source for those who are taking up the subject or want to avoid writing out the (sometimes extensive) details for oneself.	
	
\subsection{Definitions} We collect a few definitions from the main body of the text that are most relevant for us here. Throughout $\zeta$ will be a random variable with mean $0$ and variance $1$. Such a random variable is said to be \emph{subgaussian} if the \emph{subgaussian moment}
	\[\| \zeta\|_{\psi_2} := \sup_{p \geq 1} p^{-1/2} (\E |\zeta|^p)^{1/p} \]
	is finite. For $B >0$, we let $\G_B$ denote the set of mean $0$ variance $1$ random variables with subgaussian moment $\leq B$ and we let $\Gamma = \bigcup_{B >0} \G_B$.
	
	For $\z \in \G$, let $\Sym_{n}(\z)$ denote the probability space of $n \times n$ symmetric matrices with $(A_{i,j})_{i\leq j} $ i.i.d.\ distributed according to $\zeta$. Let $\col_n(\z)$ be the probability space on vectors of length $n$ with independent coordinates distributed according to $\zeta$.

	For $v\in \S^{n-1}$ and $\mu,\alpha,\gamma \in (0,1)$, Define the \emph{least common denominator} (LCD) of the vector $v$ via 
	\begin{align}\label{eq:D-def}
	D_{\alpha,\gamma}(v): = \inf \big\lbrace t>0: \|tv\|_{\T} < \min\{\gamma\|t v\|_2, \sqrt{\alpha n}\} \big\rbrace \, ,
	\end{align}
	where $\|w\|_{\T} := \dist(w,\Z^n)$. We also define
	\begin{equation}
		\label{eq:Dhat-def}
		\hat{D}_{\alpha,\gamma,\mu}(v) := \min_{\substack{I\subset [n]\\|I|\geq (1-2\mu)n}}D_{\alpha,\gamma}\left(v_I\right)\, .
	\end{equation}

	\begin{remark}
		We note that in the main body of the paper we work with a slightly different notion of $\hat{D}$, where we define $\hat{D}_{\alpha,\gamma,\mu}(v) = \min_I D_{\alpha,\gamma} (v_I/\|v_I\|_2)$.  This makes no difference for us, as Lemma \ref{lem:invertOnIc} below eliminates those $v$ for which $\|v_I\|_2$ is less than a constant.  Thus, we work with the slightly simpler definition \eqref{eq:Dhat-def} throughout.
	\end{remark}	
	
We define the set of ``structured direction on the sphere''
\begin{equation}\label{eq:Sigma-def}
		\Sigma = \Sigma_{\alpha,\g,\mu} := \big\lbrace v \in \S^{n-1} :  \hat{D}_{\alpha,\gamma,\mu}(v) \leq e^{c_{\Sigma} n} \big\rbrace \,.
	\end{equation}
	
	Now for $\z \in \G$, $A \sim \Sym_n(\zeta)$ and a given vector $w \in \R^n$, we define the quantity (as in Section~\ref{sec:quasi-randomness})
	\begin{equation}\label{eq:qw-def}
		q_n(w) = q_n(w;\alpha,\gamma,\mu) := \Pr_A\left(\, \exists v\in \Sigma  \text{ and } \exists s,t\in [-4\sqrt{n}, 4\sqrt{n}]:~Av=sv+tw \right).
	\end{equation}	We then define 
		\begin{equation}\label{eq:def-qn} q_n := \max_{w \in \S^{n-1}} q_n(w) .\end{equation}

\subsection{Main theorems of the appendix} Let us now restate the two main objectives of this appendix. Our first goal is to prove the following. 

	\begin{theorem}[Master quasi-randomness theorem] \label{thm:qLCDApp}
		For $B >0$ and $\zeta \in \G_B$, there exist constants $\alpha,\gamma,\mu,c_{\Sigma},c \in (0,1)$ depending only on $B$ so that  
		\[q_{n}(\alpha, \gamma ,\mu) \leq  2e^{-cn}\,. \]
	\end{theorem} 

	The second main goal of this appendix is to prove Theorem~\ref{thm:ILwO2}, which we will prove on our way to proving Theorem~\ref{thm:qLCDApp}.
	
	\begin{theorem}\label{th:negative-correlation}  
For $B>0$, let $\zeta \in \G_B$. For  $d \in \N$, $\alpha,\gamma \in (0,1)$ and $\nu \in (0,2^{-15})$, there are constants $c_0,R > 0$ depending only on $\alpha,\gamma,\nu,B$ so that the following holds.  Let $0\leq k \leq c_0 \alpha d$ and $t \geq \exp(-c_0\alpha d)$; let $v \in \S^{d-1}$ and let $w_1,\ldots,w_k \in \S^{d-1}$ be orthogonal.

Let $\zeta'$ be an independent copy of $\zeta$, let $Z_\nu$ be a Bernoulli random variable with parameter $\nu$  and let $\tau \in \R^d$ be a random vector whose coordinates are i.i.d.\ copies of the random variable with distribution $(\z - \z')Z_\nu$.

If $D_{\alpha,\gamma}(v) > 1/t$ then
\begin{equation*}
\PP\left( |\la \tau, v \ra| \leq t\, \text{ and }\, \sum_{j = 1}^k \langle w_j, \tau\rangle^2 \leq c_0 k \right) \leq R t \cdot e^{-c_0 k}\,.
\end{equation*}
\end{theorem}

	The proofs of Theorem~\ref{thm:qLCDApp} and Theorem~\ref{th:negative-correlation} follow the same path as \cite{RSM2} where the authors proved analogous statements for the case where the entries of $A$ are uniform in $\{-1,1\}$. We refer the reader to the following Section~\ref{sec:readersGuide} for a discussion of how this appendix is structured relative to \cite{RSM2}.

	\subsection{A Reader's Guide for the appendices}\label{sec:readersGuide}

	Here we describe the correspondence between sections in this appendix and sections in \cite{RSM2} and point out the key changes that come up.  
	
	In Section~\ref{sec:subgaussian} we set up many of the basic notions that we will need for the proof of Theorem~\ref{thm:qLCDApp}. The main novelty here
	is in the definitions of several auxiliary random variables, related to $\z$, that will be used to study $\z$ in the course of the paper. 
	
	In Section~\ref{sec:proof}, we turn to prove Theorem \ref{thm:qLCDApp}, while assuming several key results that we either import from \cite{RSM2} or prove in later sections. This section is the analogue of Section 9 in \cite{RSM2} and the main difference between these sections arises from the different definitions of $q_n$ in these two papers (see \eqref{eq:def-qn}). Here $q_n$ is defined in terms of the least common denominator $D_{\alpha,\gamma}$, rather than the threshold $\cT_L$ (see \eqref{eq:def-threshold}). In the course of the proof we \emph{also} need to break things up according to $\cT_L$, and define nets as 
	we did in \cite{RSM2}, but another net argument is required to exclude vectors with $\cT_L$ small but $D_{\alpha,\g}$ large.
	
In Section~\ref{sec:fourier}, we define many of the key Fourier-related notions that we will need to prove the remaining results, including Theorem~\ref{th:negative-correlation}. The main differences between the two papers in these sections comes from the different definition of the sublevel sets $S_W$ (see \eqref{eq:SW-def}).
This new definition requires us to reprove a few of our basic Lemmas from \cite{RSM2}, however the proofs go through easily.

In Section \ref{subsec:iLwo}, we state our main inverse Littlewood-Offord Theorem for conditioned random walks and deduce Theorem \ref{th:negative-correlation} from it.  Lemma \ref{lem:CondWalkLCMfinal} in this section is also one of the main ingredients that goes into Theorem \ref{thm:netThm}. This section corresponds to Section 3 of \cite{RSM2}.

Section \ref{sec:replacement} deals with Fourier replacement and is the analogue of	Appendix B in \cite{RSM2}. Here the only difference between the sections is that here we lack an explicit form for the Fourier transform. However, this difficulty is easily overcome.  
	
	In Section \ref{sec:ILwO-CondWalks} we prove Lemma~\ref{lem:CondWalkLCMfinal}.  This corresponds to Sections 4 and 5 of \cite{RSM2}, from which several key geometric facts are imported wholesale, making our task significantly lighter here.  The difference in the definitions from Section~\ref{sec:fourier} are salient here, but the majority of the proof is the same as in \cite[Section 5]{RSM2}, up to the constants involved.
	
	The next three sections, Sections \ref{sec:ILwO-Matrix}, \ref{sec:sizenet} and \ref{sec:approxnet}, correspond to Sections 6,7, and 8 respectively  of \cite{RSM2}. Here the adaptation to this paper requires little more than updating constants.  These three sections amount to converting Lemma \ref{lem:CondWalkLCMfinal} into the main net bound Theorem \ref{thm:netThm}.
	
	Finally, in Section \ref{app:HW} we deduce the Hanson-Wright inequality, Lemma \ref{lem:HW}, from Talagrand's inequality; this corresponds to Appendix E of \cite{RSM2} where the difference again is only up to constants.
	
	
\section{Preparations}\label{sec:subgaussian}

	\subsection{Symmetrizing and truncating the random variable}
	
	We will work with symmetrized, truncated and lazy versions of the variable $\zeta$. This is primarily because these altered versions will have better behaved Fourier properties. Here we introduce these random variables and also note some properties of their characteristic functions. These properties are not so important until 
Section~\ref{sec:fourier}, but we have them here to help motivate some of the definitions.

	Let $\zeta'$ be an independent copy of $\zeta$ and define 
	\[ \tz = \zeta - \zeta'. \] We will want to truncate $\tz$ to a bounded window, as this will be useful for our construction of a non-degenerate and not-too-large LCD in Section \ref{sec:ILwO-CondWalks}.  In this direction, define $I_B = (1,16B^2)$ and $p := \P(|\tz| \in I_B)$.  Our first step is to uniformly bound $p$ in terms of $B$.
	
	\begin{lemma}\label{lem:p-bound}
		
		$p \geq \frac{1}{2^{7} B^4}$.
	\end{lemma}
	\begin{proof}
		By the Paley-Zygmund inequality 
		 \[
		 \P(|\tz| > 1) = \P(|\tz|^2 >  \E |\tz|^2 / 2) \geq \frac{(1 - \frac{1}{2})^{2}(\E \tz^2)^{2} }{(\E \tz^4 )} \geq \frac{1}{2^6 B^4}
		 \]
		where we have used $\E \tz^4= 2 \E \zeta^4+6 \leq 2^5B^4 +6$ and $B\geq 1$.  By Chebyshev's inequality we have 
		\[
		\P(|\tz| \geq 16 B^2) \leq \frac{2}{2^{8} B^4}\,.
		\]
		Combining the bounds completes the proof.
	\end{proof}

\vspace{2mm}	
	
\noindent For a parameter $\nu \in (0,1)$, define $\xi_\nu$ by 
	\[ \xi_\nu := \one\{|\tz| \in I_B \} \tz Z_\nu,\] 
where $Z_\nu$ is an independent Bernoulli variable with mean $\nu$.   For $\nu \in (0,1)$ and $d \in \N$, we write $X \sim \Xi_\nu(d; \zeta)$ to indicate that $X$ is a random vector in $\R^d$ whose entries are i.i.d.\ copies of the variable $\xi_\nu$; similarly, we write  $X\sim \Phi_\nu(d; \zeta)$ to denote a random vector whose entries are i.i.d.\ copies of the random variable $\tz Z_\nu$.
	
We compute the characteristic function of $\xi_\nu$ to be 
	$$\phi_{\xi_\nu}(t) = \E e^{i 2\pi t \xi_\nu}  = 1 - \nu + \nu (1 - p) + \nu p \E_{\tz } [\cos( 2\pi t \tz) \,|\, |\tz| \in (1, 16 B^2)  ] \, .   $$

Define the variable $\zb$ as $\tz$ conditioned on  $|\tz| \in I_B$, where we note that this conditioning makes sense since Lemma \ref{lem:p-bound} shows $p > 0$. In other words, for every Borel set $S$, 
	\[ \PP( \zb \in S) = p^{-1} \PP(\tz  \in S\cap (I_B \cup -I_B) )\, . \]
Therefore we can write the characteristic function of $\xi_{\nu}$ as \begin{equation} \label{eq:phi-xi-def}
	\phi_{\xi_\nu}(t) = 1 -  \nu p + \nu p \E_{\zb} \cos(2\pi t \zb)\,.
	\end{equation}	
	For $x\in\R$, define $\|x \|_{\T} := \dist(x,\Z)$, and note the elementary inequalities 
	\begin{equation*}
		1 - 20 \|a\|_{\T}^2 \leq \cos(2\pi a) \leq 1 - \| a \|_{\T}^2\, ,
	\end{equation*}
	for $a\in \R$. These imply that
	\begin{equation}\label{eq:xi-bounds}
	\exp\left(- 32\nu p \cdot \E_{\zb} \| t \zb \|_\T^2  \right) \leq \phi_{\xi_\nu}(t) \leq \exp\left(- \nu p\cdot  \E_{\zb} \| t \zb \|_\T^2  \right)\, .
	\end{equation}
	Also note that since $\phi_{\tz Z_\nu}(t)=1-\nu+\nu \E_{\tz}[\cos(2\pi t\tz)]$ we have 
	\begin{align}
	\phi_{\tz Z_\nu}(t) \leq  1 - \nu + \nu (1 - p) + \nu p \E_{\tz } [\cos( 2\pi t \tz) \,|\, |\tz| \in I_B  ]= \phi_{\xi_\nu}(t)\,.\label{eq:xi-vs-zeta-bounds}
	\end{align}

	\subsection{Properties of subgaussian random variables and matrices}\label{ss:norm}
	We will use a basic fact about exponential moments of one-dimensional projections of subgaussian random variables (see, e.g. \cite[Prop.\ 2.6.1]{vershynin2018high})
	\begin{fact}\label{fact:exp-moment}
	For $B >0$, let $Y = (Y_1,\ldots,Y_d)$ be a random vector with $Y_1,\ldots,Y_d \in \G_{B}$.  Then for all $u \in \S^{d-1}$ we have $\E\, e^{\langle Y, u \rangle } = O_B(1)$.
	\end{fact}

We will also use a large deviation bound for the operator norm of $A$ 
	(see~\eqref{eq:op-norm-bound})

	\begin{fact}\label{fact:op-norm} For $B >0$, let $\zeta \in \G$ and $A \sim \Sym_n(\zeta)$. Then
		$$\P(\|A\|_{op} \geq 4 \sqrt{n}) \leq 2 e^{-\Omega(n)}\, .$$ 
	\end{fact}

\vspace{2mm}

\noindent We also define the event $\cK = \{\|A\|_{op} \geq 4\sqrt{n}\}$, and define the measure $\P^{\cK}$ by 
\begin{equation} \label{eq:def-PK} \P^{\cK}(\cE) = \P(\cK \cap \cE), \end{equation} 
for every event $\cE$.

	\subsection{Compressibility and eliminating non-flat vectors} 
\label{ss:compressibility}
As in \cite{RSM2}, we may limit our attention to vectors that are  ``flat'' on a constant proportion of their coordinates.  
This reduction is a consequence of the now-classical work of Rudelson and Vershynin on \emph{compressible} and \emph{incompressible} vectors \cite{RV}.

Following \cite{RV}, we say that a vector in $\S^{n-1}$ is $(\delta,\rho)$-compressible if it has distance at most $\rho$ from a vector with support of size at most $\delta n$.  For $\delta,\rho \in (0,1)$, let $\Comp(\delta,\rho)$ denote 
	the set of all such compressible vectors in $\S^{n-1}$. Proposition 4.2 from Vershynin's paper \cite{vershynin-invertibility} takes care of all compressible vectors.
	
	\begin{lemma}\label{lem:compressible}
	For $B >0$, let $\zeta \in \Gamma_B$, let $A_n \sim \Sym_{n}(\zeta)$ and let $K \geq 1$.  
Then there exist $\rho,\delta,c >0$ depending only on $K, B$ so that for every $\l \in \R$ and $w\in \R^n$ we have 
$$\P\big( \inf_{x \in \Comp(\delta,\rho)} \|(A_n + \lambda I)x-w \|_2 \leq c \sqrt{n} \text{ and } \|A_n + \lambda I\|_{op} \leq K \sqrt{n}\big) \leq 2 e^{-cn}\,.$$ 
\end{lemma}

For the remainder of the paper, we let $\delta,\rho$ be the constants given in Lemma~\ref{lem:compressible}.  Define 
\[ \Inc(\delta,\rho) := \S^{n-1} \setminus \Comp(\delta,\rho)\]
to be the set of $(\delta,\rho)$-\emph{incompressible} vectors. The key property of incompressible vectors is that they are ``flat'' for a constant proportion of coordinates.  This is made quantitative in the following lemma of Rudelson and Vershynin \cite{RV}.
	
	\begin{lemma}\label{lem:spread} Let $v\in \Inc(\delta,\rho)$. Then
		\begin{equation*} (\rho/2) n^{-1/2} \leq |v_i| \leq \delta^{-1/2} n^{-1/2} 
		\end{equation*}
		for at least $\rho^2\delta n/2$ values of $i\in[n]$.
	\end{lemma}
We now fix a few more constants to be held fixed throughout the paper. Let $\k_0 = \rho/3$ and $\k_1 = \delta^{-1/2}+\rho/6$, where $\delta,\rho$ are as in Lemma~\ref{lem:compressible}. For $D\subseteq [n]$ define the set of directions in $\S^{n-1}$ that are ``flat on $D$'':
	\begin{equation*}
		 \cI(D) = \left\{ v\in\S^{n-1}: (\k_0 + \k_0/2)n^{-1/2} \leq |v_i| \leq  (\k_1 -\k_0/2) n^{-1/2} \text{ for all } i\in D   \right\} 
	\end{equation*}
	and let 
	\[ \cI = \cI_d := \bigcup_{D \subseteq [n], |D| = d } \cI(D).\]  
	
	Applying Lemmas \ref{lem:compressible} and \ref{lem:spread} in tandem will allow us to eliminate vectors outside of $\cI$.

	\begin{lemma}\label{lem:invertOnIc} 
	Let $\delta,\rho, c>0$ be the constants defined in Lemma~\ref{lem:compressible} and let $d < \rho^2 \delta n/2$.
		Then 
		\begin{equation}\label{eq:invertOnIc} 
			\max_{w\in\S^{n-1}}\P_A\left( \exists v \in \S^{n-1} \setminus \cI \text{ and } \exists s,t\in [-4\sqrt{n},+4\sqrt{n}] : \|Av-sv-tw\|_2 \leq c \sqrt{n}/2  \right) \leq 2 e^{-\Omega(n)}\,.
		 \end{equation}
	\end{lemma}
	\begin{proof}  
		Lemma~\ref{lem:spread}, along with the definitions of $\k_0,\k_1$ and $\cI$, implies that 
	\[\S^{n-1}\setminus \cI \subseteq \Comp(\delta,\rho).\] 
		Now fix a $w \in \R^{n}$ and take a $c\sqrt{n}/8$-net $\cN$ for $[-4\sqrt{n},4\sqrt{n}]^2$ of size $O(c^{-2})$ to see that 
		$\|Av-sv-tw\|_2 \leq c \sqrt{n}/2$ implies that there exists $(s',t')\in \cN$ for which 
		\[ \|(A-s'I)v-t'w\|_2 \leq c \sqrt{n}.\]
		Thus the left-hand-side of \eqref{eq:invertOnIc} is \begin{equation*}
			\leq \sum_{(s',t') \in \cN}  \P_A\left( \exists v \in \Comp(\delta,\rho) :  \|(A-s'I)v-t'w\|_2 \leq c \sqrt{n} \right) \leq |\cN|\cdot 2e^{-\Omega(n)},
		\end{equation*} where the final inequality follows by first intersecting each term in the sum with the event $\cE := \{ \|A - s'I\|_{op} \leq 16n^{1/2} \}$
		 (noting that $\PP(\cE^c) \leq 2e^{-\Omega(n)}$, by Fact~\ref{fact:op-norm}) and applying Lemma~\ref{lem:compressible} to each term in the sum with $\lambda = -s'$ and $K = 16$.\end{proof}

	\subsection{Zeroed out matrices}\label{ss:M-def}
	To study our original matrix $A$, it will useful to work with random symmetric matrices that have large blocks that are ``zeroed out'' and entries that are distributed like $\tz Z_\nu$ elsewhere (see \cite{RSM2} for more discussion on this).  For this, we set $d :=c_0^2 n$ (where $c_0>0$ is a small constant to be determined later) and write $M \sim \cM_n(\nu)$ for the $n\times n$ random matrix
	\begin{equation}\label{eq:Mdef}
		M =  \begin{bmatrix}
			{\bf 0 }_{[d] \times [d]} & H_1^T \\
			H_1 & { \bf 0}_{[d+1,n] \times [d+1,n]}  
		\end{bmatrix}\,, 
	\end{equation}
 	where $H_1$ is a $(n-d) \times d$ random matrix with whose entries are i.i.d.\ copies of $\tz Z_\nu$.  
 	
 	In particular the matrix $M$ will be useful for analyzing events of the form $\|Av\|_2 \leq \eps n^{1/2} $, when $v \in \cI([d])$. 
 	
 	We now use the definition of $\cM_n(\nu)$ to define another notion of ``structure'' for vectors $v \in \S^{n-1}$. This is a very different measure of ``structure'' from that provided by the LCD, which we saw above. For $L > 0$ and $v \in \R^n$ define the \emph{threshold} of $v$ as  
 	\begin{equation}\label{eq:def-threshold}
 		\cT_L(v) := \sup\big\lbrace t \in [0,1]: \P(\|Mv\|_2 \leq t\sqrt{n}) \geq (4Lt)^n \big\rbrace\,.
 	\end{equation}
 One can think of this $\cT_L(v)$ as the ``scale'' at which the structure of $v$ (relative to $M$) starts to emerge. So ``large threshold'' means ``more structured''.

\section{Proof of Theorem~\ref{thm:qLCDApp}}\label{sec:proof}

	Here we recall some key notions from \cite{RSM2}, state analogous lemmas, and prove Theorem \ref{thm:qLCDApp} assuming these lemmas.

	\subsection{Efficient nets}	\label{ss:efficient-nets}
	Our goal is to obtain an exponential bound on the quantity
	 \[ q_n =  \max_{w \in \S} \Pr_A\left(\exists v\in \Sigma  \text{ and } \exists s,t\in [-4\sqrt{n}, 4\sqrt{n}]:~Av=sv+tw \right), \] defined at~\eqref{eq:def-qn},
	 where  
	\[ \Sigma = \Sigma_{\alpha,\g,\mu} := \big\lbrace v \in \S^{n-1} :  \hat{D}_{\alpha,\gamma,\mu}(v) \leq e^{c_{\Sigma} n}\, \big\rbrace. \]
	In the course of the proof we will choose $\alpha,\g,\mu$ to be sufficiently small.
	
	 We cover $\Sigma \subseteq \S^{n-1}$ with two regions which will be dealt with in very different ways. First we define
	\[ S :=\big\lbrace v \in \S^{n-1} : ~\cT_L(v)\geq  \exp(-2c_{\Sigma} n)\big\rbrace . \]
	This will be the trickier region and will depend on the net construction from \cite{RSM2}. We also need to take care of the region
	\[ S' := \{ v \in \S^{n-1} : \hat{D}_{\alpha,\gamma, \mu}(v) \leq \exp(c_{\Sigma} n),  ~\cT_L(v)\leq  \exp(-2c_{\Sigma} n)\}\, \] 
	which we take care of using the nets constructed by Rudelson and Vershynin in \cite{RV}. We recall that $\cT_L$ is defined at \eqref{eq:def-threshold}.
	
	We also note that since the event $\cK := \{ \|A\|_{\text{op}} \geq 4n^{1/2} \}$ fails with probability $2e^{-cn}$ (Fact~\ref{fact:op-norm}) and we only need to 
	deal with incompressible vectors $v \in \cI$ (by Lemma~\ref{lem:invertOnIc}), it is enough to show
	\begin{equation}\label{eq:pf-goal1} \sup_{w\in \S^{n-1}}\P_A^{\cK}\left( \exists v\in \cI \cap S,~s,t\in [-4\sqrt{n},+4\sqrt{n}] :~Av=sv+tw\right) \leq  e^{-\Omega(n)}, 
	\end{equation} 
	and the same with $S'$ replacing $S$. We recall that we define $\P^{\cK}(\cE) := \P(\cK \cap \cE)$ for every event $\cE$.
	
	To deal with the above probability, we will construct nets to approximate vectors in $\cI\cap S$ and $\cI\cap S'$. To define the nets used, we recall a few definitions from \cite{RSM2}. For a random variable $Y \in \R^d$ and $\eps>0$, we define the L\'evy concentration of $Y$ by \begin{equation} \label{eq:def-levy}
		\cL(Y,\eps) = \sup_{w \in \R^d} \P( \|Y - w\|_2 \leq \eps )\,.
	\end{equation} Now for $v\in \R^n$, $\eps>0$, define 
	 \begin{equation} \label{eq:def-Lop}
	 	\cL_{A,op}(v,\eps\sqrt{n}) := \sup_{w\in\R^n} \P^{\cK}(\|Av - w\|_{2} \leq \eps \sqrt{n} )\,.
 	\end{equation}
	Slightly relaxing the requirements of $\cI$, we define
		\[ \cI'([d])  := \left\lbrace v \in \R^{n} :  \k_0 n^{-1/2} \leq |v_i| \leq  \k_1 n^{-1/2} \text{ for all } i\in [d] \right\rbrace . \] 
		Define the (trivial) net 
	\begin{equation*} 
		\L_{\eps} := B_n(0,2) \cap \left(4 \eps n^{-1/2} \cdot \Z^n\right) \cap \cI'([d])\,. 
	\end{equation*}
\subsubsection{Definition of net for $v \in S$} To deal with vectors in $S$, for $\eps \geq \exp(-2c_{\Sigma} n)$ define
	 \begin{align}\label{eq:sigma-def} 
	  \Sigma_{\eps} := \big\lbrace v\in \cI([d]):~\cT_L(v)\in [\eps,2\eps]\big\rbrace \,. 
	  \end{align}
	If $v\in \Sigma_\eps$, for some $\eps\geq \exp(-2c_{\Sigma} n)$ then the proof will be basically the same as in \cite{RSM2}. As such, we  approximate $\Sigma_\eps$ by $\cN_\eps$, where we define
	\begin{equation*} \cN_{\eps}  := \left\{  v \in \L_{\eps} : (L\eps)^n \leq    \P(\|Mv\|_2\leq 4\eps\sqrt{n}) \text{ and }  \cL_{A,op}(v,\eps\sqrt{n}) \leq (2^{10} L\eps)^n \right\}\, ,
	\end{equation*}
	and show that $\cN_\eps$ is appropriately small.
	
	First the following lemma allows us to approximate $\Sigma_\eps$ by $\cN_\eps$.
	\begin{lemma}\label{thmnet} 
		Let $\eps\in (\exp(-2c_{\Sigma}n),\k_0/8)$. For each $v \in \Sigma_{\eps}$ then there is $u \in \cN_{\eps}$ such that $\|u-v\|_{\infty} \leq 4\eps n^{-1/2}$.
	\end{lemma}
	This lemma is analogous to Lemma 8.2 in \cite{RSM2} and we postpone its proof to Section~\ref{sec:approxnet}. The main difficulty faced in \cite{RSM2} is to prove an appropriate bound on $|\cN_{\eps}|$. In our case we have an analogous bound.
\begin{theorem}\label{thm:netThm} For $L\geq 2$ and $0 < c_0 \leq 2^{-50}B^{-4}$, let $n \geq L^{64/c_0^2}$,  $d \in [c_0^2n/4, c_0^2 n] $ and  $\eps >0$ be so that $\log \eps^{-1} \leq  n L^{-32/c_0^2} $. Then
\begin{equation*}
	|\cN_{\eps}|\leq \left(\frac{C}{c_0^6L^2\eps}\right)^{n}, 
\end{equation*}
where $C>0$ is an absolute constant.
\end{theorem}
The proof of Theorem \ref{thm:netThm} will follow mostly from Lemma \ref{lem:CondWalkLCMfinal}, with the rest of the deduction following exactly the same path as in \cite{RSM2}, which we present in Sections \ref{sec:ILwO-Matrix} and \ref{sec:sizenet}.

\subsubsection{Definition of net for $v \in S'$} We now need to tackle the vectors in $S'$; that is, those with 
\[ \cT_L(v)\leq \exp(-2c_{\Sigma} n) \text{ and } \hat{D}_{\alpha,\gamma,\mu}(v)\leq \exp(c_{\Sigma} n).\]  
Here we construct the nets using only the second condition using a construction of Rudelson and Vershynin \cite{RV}. Then the condition $\cT_L(v)\leq \exp(-2c_{\Sigma} n)$ will come in when we union bound over nets. With this in mind, let 
\[\Sigma'_\eps:=\big\lbrace v\in \cI([d])\cap S': \hat{D}_{\alpha,\gamma,\mu}(v)\in [(4\eps)^{-1},(2\eps)^{-1}] \big\rbrace.\] We will approximate $v\in \Sigma'_\eps$ by the net $G_\eps$, where we define
 \begin{align}\label{eq:Gepsdef}
G_{\eps}:=\bigcup_{|I|\geq (1-2\mu) n}\left\{\frac{p}{\|p\|_2}:~p\in \left(\Z^I\oplus \sqrt{\alpha} \Z^{I^c}\right)\cap B_n(0, \eps^{-1})\setminus\{0\}\right\}.
\end{align}

The following two lemmas tell us that $G_{\eps}$ is a good $\eps\sqrt{\alpha n}$-net for $\Sigma'_{\eps}$. Here, this $\sqrt{\alpha}$ is the 
``win'' over trivial nets. 

\begin{lemma} Let $\eps >0$ satisfy $\eps\leq \gamma (\alpha n)^{-1/2}/4$.
If $v \in \Sigma'_{\eps}$, then there exists $u\in G_{\eps}$ such that $\|u-v\|_{2} \leq 16 \eps \sqrt{\alpha n}$.
\end{lemma}
\begin{proof}
Set  $D=\min_{|I|\geq (1-2\mu) n}D_{\alpha,\gamma}(v_I)$, and let $I$ be a set attaining the minimum.  By definition of $D_{\alpha,\g}$, there is 
$p_I \in \Z^I \cap B_n(0, \eps^{-1}) $ so that 
\[\left\|D v_I-p_I\right\|_2< \min \{\gamma D\|v_I\|_2, \sqrt{\alpha n}\}\leq \sqrt{\alpha n},\] 
and thus $p_I \not= 0 $. We now may greedily choose $p_{I^c} \in \sqrt{\alpha} \Z^{I^c} \cap B_n(0, \eps^{-1})$ so that 
\[\left\|D v_{I^c}-p_{I^c}\right\|_2\leq  \sqrt{\alpha n}.\]
 Thus, if we set $p = p_I \oplus p_{I^c}$, by the triangle inequality we have
\[\left\|v-\frac{p}{\|p\|_2}\right\|_2\leq \frac{1}{D}(\|D v-p\|_2+|D-\|p\|_2|)\leq 4D^{-1}\sqrt{\alpha n}\leq 16\eps \sqrt{\alpha n}, \]
as desired. \end{proof}

\vspace{2mm}

\noindent We also note that this net is sufficiently small for our purposes (see \cite{RV}).

\begin{fact}\label{fact:RVnetsize}
For $\alpha, \mu \in (0,1)$, $K\geq 1$ and $\eps\leq Kn^{-1/2}$ we have
\[|G_\eps|\leq \left(\frac{32K}{\alpha^{2\mu}\eps\sqrt{n}}\right)^n\, ,\]

where $G_\eps$ is as defined at~\eqref{eq:Gepsdef}.
\end{fact}

The following simple corollary tells us that we can modify $G_{\eps}$ to build a net $G'_{\eps} \subseteq \Sigma_{\eps}$, at the cost of 
a factor of $2$ in the accuracy of the next. That is, it is a $32\eps\sqrt{\alpha n}$-net rather than a $16\eps\sqrt{\alpha n}$ net.

\begin{corollary}\label{lem:RVnetapprox}
For $\alpha, \mu \in (0,1)$, $K\geq 1$ and $\eps\leq Kn^{-1/2}$ there is a $32 \eps \sqrt{\alpha n}$-net $G'_\eps$ for $\Sigma'_\eps$ with $G'_\eps\subset \Sigma'_\eps$ and \[|G'_\eps|\leq \left(\frac{32K}{\alpha^{2\mu}\eps\sqrt{n}}\right)^n.\] 
\end{corollary}

This follows from a standard argument.
	
\subsection{Proof of Theorem~\ref{thm:qLCDApp}}\label{ss:proof}
	
	We need the following easy observation to make sure we can use Corollary \ref{lem:RVnetapprox}.
	
	\begin{fact}\label{lem:LwO-application} Let $v \in \cI$, $\mu<d/4n$ and $\gamma<\kappa_0 \sqrt{d/2n}$, then $ \hat{D}_{\alpha,\gamma,\mu}(v)\geq (2\kappa_1)^{-1} \sqrt{n} $.
	\end{fact}
	\begin{proof}
	Since $v\in \cI$ there is $D\subset [n]$ such that $|D|=d$ and $\kappa_0 n^{-1/2}\leq |v_i| \leq \kappa_1 n^{-1/2}$ for all $i\in D$. 
	Now write $\hat{D}(v) = \min_{|I|\geq (1-2\mu) n}D_{\alpha,\gamma}(v_I)$, and let $I$ be a set attaining the minimum. Since 
	$|I|\geq (1-2\mu)n\geq n-d/2$, we have $|I\cap D|\geq d/2$. So put $D' := I\cap D$ and note that for all $t\leq (2\kappa_1)^{-1}\sqrt{n}$, we have 
	\begin{equation*}
		\min_{I}d(t  v_I,\Z^n)\geq d(t  v_{D'},\Z^{D'})=t\|v_{D'}\|_2\geq t\kappa_0\sqrt{d/2n}>\gamma t.
	\end{equation*} Therefore, $D_{\alpha,\gamma}(v_I)\geq (2\kappa_1)^{-1}\sqrt{n}$, by definition.\end{proof}
	
	\vspace{2mm}
	
	\noindent When union-bounding over the elements of our net, we will also want to use the following lemma to make sure $\cL(Av,\eps)$ is small whenever $\cT_L(v)\leq \eps$.
	
	\begin{lemma}\label{lem:replacement}
		Let $\nu \leq 2^{-8}$.  For $v \in \R^n$ and $t \geq \cT_{L}(v)$ we have 
		\begin{equation*}
			\cL(Av,t\sqrt{n}) \leq (50 L t)^n\,.
		\end{equation*}
	\end{lemma}
	
\noindent	We prove this lemma in Section \ref{sec:replacement} using a fairly straight-forward argument on the Fourier-side. We now prove 
our main theorem, Theorem~\ref{thm:qLCDApp}.

	\begin{proof}[Proof of Theorem~\ref{thm:qLCDApp}]
We pick up from \eqref{eq:pf-goal1} and look to show that
\begin{equation}\label{eq:pf-goal2}
 q_{n,S} := \sup_{w\in \S^{n-1}}\P_A^{\cK}\left( \exists v\in \cI \cap S,~s,t\in [-4\sqrt{n},+4\sqrt{n}] :~Av=sv+tw\right) \leq  e^{-\Omega(n)},  \end{equation}
and the same with $S'$ in place of $S$. We do this in three steps.

We first pause to describe how we choose the constants. We let $c_0 >0$ to be sufficiently small so that
Theorem~\ref{thm:netThm} holds and we let $d := c_0^2n$. The parameters, $\mu, \g$ will be chosen small compared to $d/n$ and $\k_0$ so that Fact~\ref{lem:LwO-application} holds. $L$ will be chosen to be large enough so that $L>1/\kappa_0$ and so that it is larger than some absolute 
constants that appear in the proof. We will choose $\alpha>0$ to be small compared to $1/L$ and $1/\k_0$ and we will choose $c_{\Sigma}$ small compared to $1/L$.

\vspace{2mm}		
	
\noindent\textbf{Step 1: Reduction to $\Sigma_\eps$ and $\Sigma_\eps'$}. Using that $\cI = \bigcup_{D} \cI(D),$ we union bound over all choices of $D$. By symmetry of the coordinates we have
\begin{equation} \label{eq:I([d])bnd}  
		q_{n,S} \leq 2^n \sup_{w\in \S^{n-1}}\, \P_A^{\cK}\left( \exists v\in \cI([d]) \cap S,~s,t\in [-4\sqrt{n},+4\sqrt{n}] :~Av=sv+tw \right)\,. \end{equation} Thus it is enough to show that the supremum at \eqref{eq:I([d])bnd} is at most $4^{-n}$, and the same with $S$ replaced by $S'$.
		
Now, let $\cW=\left(2^{-n}\Z \right)\cap [-4\sqrt{n},+4\sqrt{n}] $  and notice that for all $s, t\in [-4\sqrt{n},+4\sqrt{n}]$, there is $s', t'\in \cW$ with $|s-s'|\leq 2^{-n}$ and $|t-t'|\leq 2^{-n}$. So, union-bounding over all $(s',t')$, the supremum term in \eqref{eq:I([d])bnd} is at most 
\begin{equation*}
		\leq 8^n\sup_{w\in \R^n,~|s|\leq 4\sqrt{n}}\, \P_A^{\cK}\left( \exists v\in \cI([d]) \cap (S\cup S') :~\|Av-sv-w\|_2\leq 2^{-n+1} \right)\,
		\end{equation*} and the same with $S$ replaced with $S'$.
		
		We now need to treat $S$ and $S'$ a little differently. Starting with $S$, we let $\eta:=\exp(-2c_{\Sigma} n)$ and note that for $v \in S$ we have, by definition, that 
		\begin{equation}\label{eq:choiceOfL}  \eta\leq \cT_L(v)\leq 1/L\leq \kappa_0/8, \end{equation}
		where we will guarantee the last inequality holds by our choice of $L$ later.
		
		Now, recalling the definition of $\Sigma_{\eps} := \Sigma_{\eps}([d])$ at~\eqref{eq:sigma-def}, we may write
		\[\cI([d]) \cap S \subseteq \bigcup_{j=0}^n \left\{v\in \cI : \cT_L(v)\in [2^{j}\eta,2^{j+1}\eta] \right\}\, = \bigcup_{j=0}^{ j_0} \Sigma_{2^j\eta}\,  ,\]
		where $j_0$ is the largest integer such that $2^{j_0}\eta\leq\kappa_0/2$. Thus, by the union bound, it is enough to show
		\begin{equation}\label{eq:bndSEps} 
		Q_\eps:= \max_{w\in \R^n,~|s|\leq 4\sqrt{n}}\P_A^{\cK}\left( \exists v\in \Sigma_{\eps}:~\|Av-sv-w\|_2\leq 2^{-n+1} \right) \leq 2^{-4n},
	\end{equation}
		for all $\eps \in [\eta,\k_0/4]$.
		
		We now organize $S'$ in a similar way, relative to the sets $\Sigma_\eps'$.  For this, notice that for $v \in \cI([d]) \cap S'$ we have \[(2\kappa_1)^{-1}\sqrt{n}\leq \hat{D}_{\alpha,\gamma,\mu}(v)\leq \exp(c_{\Sigma} n)=\eta^{-1/2},\] 
		by Fact \ref{lem:LwO-application}. So if we recall the definition
		\[\Sigma'_\eps:=\{v\in \cI([d])\cap S': \hat{D}_{\alpha,\gamma,\mu}(v)\in [(4\eps)^{-1},(2\eps)^{-1}]\}\] then 
		\[\cI([d]) \cap S' \subseteq \bigcup_{j=-1}^{j_1} \Sigma'_{2^j\sqrt{\eta}}\, ,\]
		where $j_1$ is the least integer such that $2^{j_1}\sqrt{\eta}\geq \kappa_1/(2\sqrt{n})$. Union-bounding over $j$ shows that it is sufficient to show\begin{equation}\label{eq:bndSpEps} 
			Q'_\eps:=\max_{w\in \R^n,~|s|\leq 4\sqrt{n}}\P_A^{\cK}\left( \exists v\in \Sigma_{\eps}':~\|Av-sv- w\|_2\leq 2^{-n+1} \right) \leq 2^{-6n},
			\end{equation} for all $\eps \in [\sqrt{\eta}, \k_1/\sqrt{n}]$.
		
		\medskip
	\noindent	\textbf{Step 2: A Bound on $Q_\eps$}: Take $w\in \R^n$ and $|s|\leq 4\sqrt{n}$; we will bound the probability uniformly over  $w$ and $s$. 
		Since $\exp(-2c_{\Sigma} n)<\eps < \k_0/8$, for $v \in \Sigma_{\eps}$ we apply Lemma~\ref{thmnet}, 
		to find a $u \in \cN_{\eps} = \cN_{\eps}([d])$ so that $\|v - u\|_2 \leq 4\eps$. So if $\| A \|_{op}\leq 4\sqrt{n}$, we see that 
		\begin{align*} \|Au-su -w\|_2 &\leq \|Av-sv -w\|_2 + \|A(v-u)\|_2+|s|\|v-u\|_2 \\
			&\leq \|Av -sv -w\|_2 + 8\sqrt{n}\|(v-u)\|_2 \\
			& \leq 33\eps\sqrt{n}
		\end{align*}
		and thus
		\[ \{ \exists v\in \Sigma_{\eps} :~\|Av-sv-w\|_2\leq 2^{-n+1} \} \cap \{ \|A\|\leq 4\sqrt{n} \} \subseteq \{ \exists u \in \cN_{\eps} : \| Au-su-w\|\leq  33\eps\sqrt{n} \}.   \] 
So, by union bounding over our net $\cN_{\eps}$, we see that 
\begin{align*}
 Q_{\eps} \leq \P_A^{\cK}\left(\exists v \in \cN_{\eps} : \|Av-sv-w\|\leq  33\eps\sqrt{n} \right) 
&\leq \sum_{u \in \cN_{\eps}} \PP_A^{\cK}( \|Au - s'u-w\|_2 \leq 33\eps\sqrt{n}) \\ 
 &\leq \sum_{u \in \cN_{\eps}} \cL_{A,op}\left(u, 33\eps \sqrt{n} \right), \end{align*}
where $\cL_{A,op}$ is defined at \eqref{eq:def-Lop}.

Note that for any $u$ we have that $\cL_{A,op}\left(u, 33\eps \sqrt{n} \right) \leq (67)^n \cL_{A,op}(u,\eps\sqrt{n})$ (see, e.g.,  Fact 6.2 in \cite{RSM2}); as such, for any $u \in \cN_\eps$ we have $\cL_{A,op}\left(u, 33\eps \sqrt{n} \right) \leq (2^{17}L\eps)^n$. Using this bound gives  

\[  Q_{\eps}  \leq |\cN_{\eps}|(2^{17} L\eps)^n \leq \left(\frac{C}{L^2\eps}\right)^n(2^{17} L\eps)^n \leq 2^{-4n}, \]
		where the penultimate inequality follows from our Theorem~\ref{thm:netThm} and the last inequality holds for the choice of $L$ large enough relative to the universal constant $C$ and so that \eqref{eq:choiceOfL} holds. To see that the application 
		of Theorem~\ref{thm:netThm} is valid, note that 
		\[ \log 1/\eps \leq \log 1/\eta = 2c_{\Sigma} n \leq nL^{-32/c_0^2}, \]
		where the last inequality holds for $c_{\Sigma}$ small compared to $L^{-1}$.

			\medskip
		\noindent	\textbf{Step 3: A Bound on $Q_\eps'$}. 
		To deal with $Q'_\eps$, we employ a similar strategy. Fix $w\in \R^n$ and $|s|\leq 4\sqrt{n}$. Since we chose $\mu,\g$ to be sufficiently small so that Fact~\ref{lem:LwO-application} holds, we have that 
		 \[ \eps\leq \k_1/\sqrt{n}.\] Thus we may apply Corollary~\ref{lem:RVnetapprox} with $K=\kappa_1$ for each $v\in \Sigma'_{\eps}$ to get $u\in G'_\eps\subset \Sigma'_\eps$ such that $\|v-u\|_2\leq 	32\eps\sqrt{\alpha n}$. Now since 
		\[ \{  \exists v\in \Sigma'_{\eps} :~\|Av-sv-w\|_2\leq 2^{-n+1} \} \cap \{ \|A\|\leq 4\sqrt{n} \} \subseteq \{  \exists u \in G'_{\eps} : \| Au-su-	w\|\leq  2^9\eps\sqrt{\alpha}n \}   \] and since $2^9\eps \sqrt{\alpha n} \geq \exp(-2c_{\Sigma} n)\geq \cT_L(u)$, by Lemma \ref{lem:replacement} we have \[Q'_\eps\leq \left(\frac{32\kappa_1}{\alpha^{\mu}				\eps\sqrt{n}}\right)^n\sup_{u\in G'_\eps}\cL(Au,2^9\eps \sqrt{\alpha} n)\leq (2^{20}L\kappa_1\alpha^{1/4})^n\leq 2^{-4n},\]
		assuming that $\alpha$ is chosen to be sufficiently small relative to $L\kappa_1$.
		This completes the proof of Theorem~\ref{thm:qLCDApp}.
	\end{proof}

	\section{Fourier preparations for Theorem~\ref{th:negative-correlation}}\label{sec:fourier}

	\subsection{Concentration, level sets, and Esseen-type inequalities}\label{ss:fourier-esseen}

One of the main differences between this work and \cite{RSM2} is the notion of a ``level set'' of the Fourier transform, an change that requires us to 
make a fair number of small adjustments throughout. Here we set up this definition along with a few related definitions. 

For a random variable $Y \in \R^d$ and $\eps>0$, we recall that  L\'evy concentration of $Y$ was defined at \eqref{eq:def-levy} by 
\[ \cL(Y,\eps) = \sup_{w \in \R^d} \P( \|Y - w\|_2 \leq \eps ). \]
		Our goal is to compare the concentration of certain random vectors to the gaussian measure of associated (sub-)level sets.
		Given a $2d \times \ell$ matrix $W$, define the \emph{$W$-level} set for $t \geq 0$ to be
	\begin{equation}\label{eq:SW-def}
		S_W(t) := \{ \theta \in \R^{\ell} : \E_{\zb}\, \| \zb W \theta \|_{\T}^2 \leq t  \}\,.
	\end{equation}
		Let $g = g_d$ denote the gaussian random variable in dimension $d$ with mean $0$ and covariance matrix $(2\pi)^{-1} I_{d \times d}$.  Define $\gamma_d$ to be the corresponding measure, i.e.\ $\gamma_d(S) = \P_g(g \in S)$ for every Borel set $S \subset \R^d$.  We first upper bound the concentration via an Esseen-like inequality.
	
	\begin{lemma} \label{lem:esseen}
		Let $\beta > 0, \nu \in (0,1/4)$, let $W$ be a $2d \times \ell$ matrix and $\tau \sim \Phi_\nu(2d;\zeta)$. Then there is an $m > 0$ so that
		\begin{equation*}
			\cL(W^T \tau, \beta \sqrt{\ell})  \leq 2 \exp\left(2 \beta^2 \ell - \nu p m/2 \right)\gamma_{\ell}(S_W(m))\,.
			\end{equation*}
	\end{lemma}
	\begin{proof}
	For $w\in \R^\ell$, apply Markov's inequality to obtain
		\begin{equation*}
			\P_\tau\big( \|W^T \tau - w \|_2 \leq \beta \sqrt{\ell} \big) \leq \exp\left(\frac{\pi}{2} \beta^2 \ell \right) \E_\tau \exp\left(- \frac{\pi \|W^T \tau - w\|_2^2 }{2}\right)\, .
		 \end{equation*}
		Using the Fourier transform of a Gaussian, we compute
		\begin{equation} \label{eq:essen1} 
			\E_{\tau} \exp\left(-\frac{ \pi \| W^T \tau - w\|_2^2}{2}\right) = \E_{g}\, e^{-2\pi i\langle w, g\rangle} \E_\tau e^{ 2\pi i g^T W^T \tau }.
		\end{equation} 
		Now denote the rows of $W$ as $w_1,\ldots,w_{2d}$ and write
	\[ \E_\tau e^{ 2\pi i g^T W^T \tau } = \prod_{i=1}^{2d} \E_{\tau_i} e^{2\pi i \sum \tau_i \la g, w_i\ra } = \prod_{i=1}^{2d} \phi_{\tau}(  \la g, w_i\ra ), \]
	where $\phi_{\tau}(\t)$ is the characteristic function of $\tau$. Now apply \eqref{eq:xi-vs-zeta-bounds} and then \eqref{eq:xi-bounds} to see the right-hand-side of \eqref{eq:essen1} is
		\[ \leq  \left| \E_{g}\, e^{-2\pi i\langle w, g\rangle} \E_\tau e^{ 2\pi i g^T W^T \tau } \right| \leq \E_{g}\, \exp(-\nu p \E_{\zb}\| \zb W g\|_{\T}^2). \]
We rewrite this as
		\begin{align*} 
			\int_{0}^{1} \P_{g}(\exp(-\nu p \E_{\zb}\|\zb W g\|_{\T}^2)\geq t)\, dt &= \nu p\int_{0}^{\infty} \P_{g}(\E_{\zb}\|\zb W g\|_{\T}^2\leq u) e^{-\nu p u}\, du \\
		&= \nu p\int_{0}^{\infty} \gamma_{\ell}(S_W(u)) e^{-\nu p u}\, du \,  ,
		\end{align*}
		where for the first equality we made the change of variable $t= e^{-\nu p u}$. Choosing $m$ to maximize  $\gamma_{\ell}(S_W(u)) e^{-\nu p u/2}$ as a function of $u$ yields
		\begin{equation*}
		\nu p\int_{0}^{\infty} \gamma_{\ell}(S_W(u)) e^{-\nu p u} du \leq \nu p \gamma_{\ell}(S_W(m))e^{-\nu p m/2} \int_{0}^{\infty}e^{-\nu p u/2}du 
		= 2\gamma_{\ell}(S_W(m))e^{-\nu p m/2}\, .
		\end{equation*}
		Putting everything together, we obtain
		\begin{equation*}
		\P_\tau(\|W^T\tau-w\|_2\leq 2\beta\sqrt{\ell}) \leq 2e^{ 2\beta^2 \ell } e^{-\nu p m/2} \gamma_{\ell}(S_W(m))\, .
		\end{equation*}
	\end{proof}
	
We also prove a comparable lower bound. 
	
	\begin{lemma} \label{lem:revEsseen}
		Let $\beta > 0$, $\nu \in (0,1/4)$, let $W$ be a $2d \times \ell$ matrix and let $\tau \sim \Xi_\nu(2d;\zeta)$.  Then for all $t \geq 0$ we have  
		\[ \g_{\ell}(S_W(t))e^{-32\nu p t} \leq \P_{\tau}\big( \|W^T \tau\|_2\leq \beta\sqrt{\ell} \big)+ \exp\left(-\beta^2\ell\right). \]
	\end{lemma}
	\begin{proof}
	Set $X = \|W^T\tau\|_2$ and write 
\begin{equation*} 
	\E_X e^{-\pi X^2/2}  = \E_X\, \1( X\leq \beta\sqrt{\ell} )e^{-\pi X^2/2} 
+ \E_X\,\1\big(  X \geq \beta\sqrt{\ell} \big) e^{-\pi X^2/2} \leq \P_X(X\leq \beta\sqrt{\ell} ) + e^{-\pi \beta^2\ell/2}\, .  
\end{equation*}
Bounding $\exp(-\pi \beta^2\ell/2)\leq \exp(-\beta^2\ell)$ implies
\begin{equation*}
  \E_\tau \exp\left(\frac{-\pi \|W^T \tau\|_2^2}{2}\right) \leq \P_\tau(\|W^T \tau\|_2\leq \beta\sqrt{\ell}) + e^{-\beta^2\ell}.  
\end{equation*}
		As in the proof of Lemma \ref{lem:esseen} above, use the Fourier transform of the Gaussian and \eqref{eq:xi-bounds} to lower bound
		\begin{equation*}
			\E_\tau \exp\left(-\frac{ \pi\|W^T \tau\|_2^2}{2}\right)  \geq \E_{g}[\exp(-32\nu p\E_{\zb}\|\zb W g\|_{\T}^2)]\,. 
		\end{equation*} 
		Similar to the proof of Lemma~\ref{lem:esseen},  write
		\begin{equation*} 
			\E_g[\exp(-32\nu p \E_{\zb} \| W g\|_{\T}^2)] = 32\nu p\int_{0}^{\infty} \g_{\ell}(S_W(u)) e^{-32\nu p u} du \geq 32\nu p\g_{\ell}(S_W(t))\int_t^{\infty} e^{-32 \nu p u}\, du,
		\end{equation*}
		where we have used that $\g_{\ell}(S_W(b)) \geq \g_{\ell}(S_W(a))$ for all $b \geq a$. This completes the proof of Lemma~\ref{lem:revEsseen}.
	\end{proof}
	
	\subsection{Inverse Littlewood-Offord for conditioned random walks}\label{subsec:iLwo}
	First we need a generalization of our important Lemma 3.1 from \cite{RSM2}. Given a $2d \times \ell$ matrix $W$ and a vector $Y\in \R^d$, we define the $Y$-augmented matrix $W_Y$ as 
	\begin{equation}\label{eq:WYdef} 
		W_Y = \begin{bmatrix} \, \, \,  W \, \, \,  , \begin{bmatrix} { \bf 0}_d \\ Y \end{bmatrix} , \begin{bmatrix} Y \\ { \bf 0}_d \end{bmatrix} \end{bmatrix} .  
	\end{equation}
	
	When possible, we are explicit with the many necessary constants and ``pin'' several to a constant $c_0$, which we treat as a parameter to be taken sufficiently small. We also recall the definition of `least common denominator' $D_{\alpha,\g}$ from~\eqref{eq:D-def}
	\[D_{\alpha,\gamma}(v): = \inf \big\lbrace t>0: \|tv\|_{\T} < \min\{\gamma\|t v\|_2, \sqrt{\alpha n} \}\big\rbrace.\]
	
The following is our generalization of Lemma 3.1 from \cite{RSM2}.

	\begin{lemma}\label{lem:CondWalkLCMfinal}
		For any $0<\nu\leq 2^{-15}$, $c_0\leq 2^{-35}B^{-4}\nu$, $d \in \N$, $\alpha \in (0,1)$ and $\gamma \in (0,1)$, let $k\leq 2^{-32}B^{-4}\nu \alpha d$ and $t \geq \exp\left(-2^{-32}B^{-4}\nu \alpha d\right)$. Let $Y \in \R^d$ satisfy $\| Y \|_2 \geq 2^{-10} c_0 \gamma^{-1}t^{-1}$,
		let $W$ be a $2d \times k$ matrix with $\|W\| \leq 2$, $\|W\|_{\HS}\geq \sqrt{k}/2$, and let $\tau \sim \Phi_\nu(2d;\zeta)$.
		
		If $D_{\alpha,\gamma}(Y)> 2^{10} B^2$ then 
		\begin{equation}\label{eq:LCM-hypo}  
		\cL \left( W^T_Y \tau, c_0^{1/2} \sqrt{k+1} \right) 
		\leq \left( R t \right)^2 \exp\left(-c_0 k\right)\,, \end{equation} 
		where $R = 2^{35} B^2 \nu^{-1/2} c_0^{-2}$.
	\end{lemma}
	
	\vspace{3mm}
	We present the proof of Lemma \ref{lem:CondWalkLCMfinal} in Section \ref{sec:ILwO-CondWalks}, and deduce our standalone ``inverse Littlewood-Offord theorem'' Theorem \ref{th:negative-correlation} here: 

	\begin{proof}[Proof of Theorem \ref{th:negative-correlation}]
	Let $c_0= 2^{-35}B^{-4}\gamma^2\nu$. First note that 
	\[\P\left( |\langle v, \tau\rangle|\leq t \text{ and }  \sum_{i=1}^k \langle w_i, \tau\rangle^2\leq c_0 k\right)^2 \leq \P\left( |\langle v, \tau\rangle|\leq t\, , |\langle v, \tau'\rangle|\leq t \text{ and }  \sum_{i=1}^k \langle w_i, \tau\rangle^2\leq c_0 k\right) \]
	where 
	$\tau,\tau' \sim \Phi_\nu(d;\zeta)$ are independent. We now look to bound the probability on the right-hand-side using Lemma \ref{lem:CondWalkLCMfinal}.
	
	Let $W$ be the $2d \times k$ matrix
	\[W=\begin{bmatrix}
	\, w_1 \, \ldots \, w_k\\ \, {\bf{0}_d }\, \ldots \, {\bf{0}_d}\,
	\end{bmatrix}\] and $Y= \sqrt{c_0/2} vt^{-1}$.
	Note that if $|\langle v, \tau\rangle|\leq t$, $|\langle v, \tau'\rangle|\leq t$ and $\sum_{i=1}^k \langle w_i,  \tau\rangle^2 \leq c_0 k$ then $\|W^T_Y (\tau,\tau')\|_2\leq c_0^{1/2} \sqrt{k+1}$. 
	 Therefore \[\P\left( |\langle v, \tau\rangle|\leq t\, , |\langle v, \tau'\rangle|\leq t \text{ and }  \sum_{i=1}^k \langle w_i, \tau\rangle^2\leq c_0 k\right)\leq \cL \left( W^T_Y (\tau,\tau'), c_0^{1/2} \sqrt{k+1} \right).\] 
	Now, $\|Y\|_2=\sqrt{c_0/2}t^{-1}>2^{-10}c_0\gamma^{-1}t^{-1}$, $\|W\|=1$, $\|W\|_{\HS}=\sqrt{k}$ and $$D_{\alpha,\gamma}(Y)\geq t c_0^{-1/2}D_{\alpha,\gamma}(v)>2^{10} B^2.$$ 
	 We may therefore apply Lemma \ref{lem:CondWalkLCMfinal} to bound
	  \[\cL \left( W^T_Y (\tau,\tau'), c_0^{1/2} \sqrt{k+1} \right) 
		\leq \left( R t \right)^2 \exp\left(-c_0 k\right).\]
		The result follows.
		\end{proof}

	\section{Fourier Replacement}\label{sec:replacement}
	The goal of this section is to prove Lemma~\ref{lem:replacement}, which relates the ``zeroed out and lazy'' matrix $M$, defined at \eqref{eq:Mdef}, to our original matrix $A$. We will need a few inequalities on the Fourier side first.

	\begin{lemma}\label{lem:fourier-zeta-xi-compare}
		For every $t \in \R$ and $\nu \leq 1/4$ we have $$|\phi_\zeta(t)| \leq  \phi_{\tilde{\zeta}Z_\nu}(t)\,.$$
	\end{lemma}
	\begin{proof}
		
		Note $|\phi_\zeta(t)|^2 = \E_{\tz} \cos(2\pi t\tz)$.  Use the elementary inequality 
		\[\cos(a) \leq 1-2\nu(1-\cos(a)) \qquad \text{ for } \nu\leq 1/4,\] and that $\sqrt{1-x}\leq 1-x/2$ to bound $$
		|\phi_\zeta(t)| = \sqrt{\E_{\tz} \cos(2\pi t\tz)} \leq \sqrt{1-2\nu \E_{\tz} (1-\cos(2\pi t\tz))} \leq \phi_{\tilde{\zeta}Z_\nu}(t)\,.
		$$\end{proof}

	We also need a bound on a Gaussian-type moment for $\|Mv\|_2$. On a somewhat technical point, we notice that $\cT_L(v) \geq 2^n$, since the definition of 
	$\cT_L$ \eqref{eq:def-threshold} depends on the definition of $M$ at \eqref{eq:Mdef}, which trivially satisfies 
	\[ \PP_M( Mv = 0 ) \geq \PP_M( M= 0) =  (1-\nu)^{\binom{n+1}{2}},\] for all $v$ and $\nu < 1/2$.
	\begin{fact}\label{fact:expForm} For $v \in \R^n$, and $t \geq \cT_L(v)$, we have
		\[  \E \exp(-\pi \|Mv\|_2^2 / 2t^2)  \leq  (9 Lt )^n .\]
	\end{fact}
\begin{proof}
	Bound
	\begin{equation} \label{eq:Mv-split}
	\E \exp(-\pi \|Mv\|_2^2 / 2t^2) \leq 
		\P(\|M v\|_2 \leq t \sqrt{n}) + \sqrt{n} \int_{t}^\infty e^{-s^2 n /t^2}\P(\|M v \|_2 \leq s \sqrt{n})\,ds\,.
	\end{equation}
	Since $t \geq \cT_L(v)$, we have $\P(\|Mv \|_2 \leq s\sqrt{n}) \leq (4Ls)^n$ for all $s\geq t$.  Thus we may bound 
	\begin{equation*} 
		\sqrt{n}\int_{t}^\infty \exp\left(- \frac{s^2 n }{t^2}\right)\P(\|M v \|_2 \leq s \sqrt{n})\,ds 
	\leq \sqrt{n}(8Lt)^n \int_t^\infty  \exp\left(- \frac{s^2 n }{t^2}\right)(s/t)^n \,ds\, . 
	\end{equation*}
	Changing variables $u=s/t$, we may bound the right-hand side by
	\begin{equation*}  t^{-1} \sqrt{n}(4Lt)^n \int_1^\infty \exp(-u^2n) u^n \,du \nonumber \leq t^{-1}\sqrt{n}(4Lt)^n \int_1^\infty \exp(-u^2/2)\,du \leq (9 Lt )^n, 
	\end{equation*}
	as desired. Note here that we used that $t \geq 2^{-n}$.
\end{proof}

\vspace{2mm}

For $v,x \in \R^n$ and $\nu \in (0,1/4)$ define the characteristic functions of $Av$ and $Mv$, respectively, $\psi_v$ and $\chi_{v,\nu}$, by
	$$\psi_v(x) := \E_A\, e^{2\pi i \langle Av,x\rangle} = \left( \prod_{k = 1}^n \phi_\zeta(v_k x_k ) \right)\left(\prod_{j < k} \phi_\zeta( x_j v_k + x_k v_j) \right)$$
	and 
	$$\chi_{v}(x) := \E_M\, e^{2\pi i \langle M v,x\rangle} = \prod_{j = 1}^d \prod_{k = d+1}^n \phi_{\tz Z_\nu}( x_j v_k + x_k v_j)\,.$$
	
Our ``replacement'' now goes through.
	
	\begin{proof}[Proof of Lemma \ref{lem:replacement}]
By Markov, we have 
\begin{equation}\label{eq:Av-markov}
\P(\|A v - w\|_2 \leq t \sqrt{n}) \leq  \exp(\pi n/2) \E\, \exp\left(- \pi\| A v - w\|_2^2 / (2t^2)\right)\,.
\end{equation}
Then use Fourier inversion to write
\begin{equation}\label{eq:Av-FI}
\E_A\, \exp\left(- \pi \| A v - w\|_2^2 / (2t^2)\right) = \int_{\R^n} e^{-\pi \| \xi \|_2^2} \cdot e^{-2\pi it^{-1}\langle w, \xi\rangle}  \psi_v(t^{-1}\xi)\,d\xi\,.
 \end{equation}
Now apply the triangle inequality, Lemma~\ref{lem:fourier-zeta-xi-compare} and the non-negativity of $\chi_{v}$ yields that the right-hand side of \eqref{eq:Av-FI} is \begin{equation*} 
\leq	\int_{\R^n} e^{-\pi \| \xi \|_2^2 } \chi_v(t^{-1}\xi)\,d\xi = \E_M \exp(-\pi \|Mv\|_2^2 / 2t^2)\,.  
\end{equation*}
Now use Fact~\ref{fact:expForm} along with the assumption $t \geq \cT_L(v)$ to bound
\begin{equation*}  
	\E_M \exp(-\pi \|Mv\|_2^2 / 2t^2)\leq  (9 Lt )^n, 
\end{equation*}
as desired.\end{proof}

\section{Proof of Lemma \ref{lem:CondWalkLCMfinal}}\label{sec:ILwO-CondWalks}
	
	In this section we prove the crucial Lemma~\ref{lem:CondWalkLCMfinal}. Fortunately much of the geometry needed to prove this theorem can be pulled from the proof of the $\{-1,0, 1\}$-case in \cite{RSM2}, and so the deduction of the theorem becomes relatively straightforward.
	
	\subsection{Properties of Gaussian space and level sets}\label{subsec:gauss-space}
	
	For $r, s > 0$ and $k \in \N$ define the \emph{cylinder} $\G_{r,s}$ by
	\begin{equation} \label{eq:defCy} 
	\Gamma_{r,s} := \left\{\theta \in \R^{k+2} : \left\|\theta_{[k]} \right\|_2\leq r, |\theta_{k+1}|\leq s \text{ and } |\theta_{k+2}|\leq s \right\}.
	\end{equation}
	For a measurable set $S \subset \R^{k+2}$ and $y \in \R^{k+2}$ define the set $$F_y(S; a,b) := \{\theta_{[k]} = (\theta_1,\ldots, \theta_k) \in \R^{k} : (\theta_1,\ldots,\theta_k,a,b) \in S - y \}\,.$$

	Recall that $\gamma_k$ is the $k$-dimensional gaussian measure defined by $\g_k(S) = \P(g \in S)$, where $g \sim \mathcal{N}(0, (2\pi)^{-1} I_{k})$, and where $I_k$ denotes the $k \times k$ identity matrix.
	The following is a key geometric lemma from~\cite{RSM2}.
	
	\begin{lemma}\label{lem:geo-comparison}
		Let $S \subset \R^{k+2}$ and $s > 0$ satisfy 
		\begin{equation}\label{eq:geo-comp-hypothesis}
		8 s^2 e^{-k/8} + 32 s^2 \max_{a, b, y} \left(\gamma_k(F_y(S;a,b) - F_y(S;a,b) ) \right)^{1/4} \leq \gamma_{k+2}(S)\,.
		\end{equation}
		Then there is an $x \in S$ so that $$(\Gamma_{2\sqrt{k},16} \setminus \Gamma_{2\sqrt{k},s} + x) \cap S \neq \emptyset\,.$$
	\end{lemma}
	
	This geometric lemma will be of crucial importance for identifying the LCD.  Indeed we will take $S$ to be a representative level set, on the Fourier side, for the probability implicit on the left-hand-side of Lemma \ref{lem:CondWalkLCMfinal}.  The following basic fact will help explain the use of the difference appearing in Lemma \ref{lem:geo-comparison}.

	\begin{fact}\label{fact:LevelSetTriangleInq} For any $2d \times \ell$ matrix $W$ and $m > 0$ we have
		\begin{equation*} 
			S_W(m) - S_W(m) \subseteq S_W(4m)\,.
			\end{equation*}
	\end{fact}
	\begin{proof}
		For any $x,y\in S_W(m)$ we have $\E_{\zb}\|\zb W x\|_{\T}^2, \E_{\zb}\|\zb W y\|_{\T}^2\leq m\, $. The triangle inequality implies
		$$\| \zb W (x-y) \|_{\T}^2 \leq 2 \|\zb W x \|_{\T}^2 + 2 \|\zb W y\|_{\T}^2\,.$$  
		Taking $\E_{\zb}$ on both sides completes the fact.
	\end{proof}

	\subsection{Proof of Lemma~\ref{lem:CondWalkLCM}}\label{subsec:ILwOcond-walks-main}

	The following is our main step towards Lemma~\ref{lem:CondWalkLCMfinal}.
	
	\begin{lemma}\label{lem:CondWalkLCM}
		For $d \in \N$, $\gamma,\alpha \in (0,1)$ and $0<\nu\leq 2^{-15}$, let $k\leq 2^{-17}B^{-4}\nu\alpha d $ and $t \geq \exp(-2^{-17}B^{-4}\nu\alpha d)$.  
		For $c_0 \in (0,2^{-50}B^{-4})$, let $Y \in \R^d$ satisfy $\|Y \| \geq  2^{-10} c_0 \gamma^{-1} / t$ and let $W$ be a $2d \times k$ matrix with $\|W\| \leq 2$.  
		
		Let  $\tau \sim \Xi_\nu(2d;\zeta)$ and $\tau' \sim \Xi_\nu(2d;\zeta)$ with $\nu = 2^{-7}\nu$ and let $\beta \in [c_0/2^{10},\sqrt{c_0}]$ and $\beta' \in (0,1/2) $. If 
		\begin{equation}\label{eq:LCM-hypo-lazy}  
		\cL(W^T_Y\tau, \beta\sqrt{k+1}) 
		\geq \left( R t\right)^2 \exp(4\beta^2 k)\left(\P(\|W^T \tau'\|_2\leq \beta'\sqrt{k}) + \exp(-\beta'^2 k) \right)^{1/4} 
		\end{equation}
		then $D_{\alpha,\gamma}(Y)\leq 2^{10}B^2$.  Here we have set $R = 2^{35}\nu^{-1/2} B^2 /c_0^2$.
	\end{lemma}
	
	\begin{proof}
		By Lemma~\ref{lem:esseen} we may find an $m$ for which the level set $S = S_{W_Y}(m)$ satisfies 
		\begin{equation} \label{eq:EssenApp} 
		\cL(W^T_Y\tau, \beta\sqrt{k+1})  \leq 4 e^{-\nu p m/2 + 2\beta^2k}\g_{k+2}(S).   
		\end{equation}
		Combining \eqref{eq:EssenApp} with the assumption \eqref{eq:LCM-hypo-lazy} provides a lower bound of 
		\begin{equation} \label{eq:gS-LB} 
		\g_{k+2}(S) \geq  \frac{1}{4} e^{\nu p m /2+ 2 \beta^2 k} \left( R t\right)^2 \left(\P(\|W^T \tau'\|_2 \leq \beta'\sqrt{k}) + \exp(-\beta'^2 k) \right)^{1/4}. 
		\end{equation}
		Now, preparing for an application of Lemma \ref{lem:geo-comparison}, define
		\begin{equation} \label{eq:defs} 
		r_0 :=  \sqrt{k} \qquad \text{ and } \qquad s_0 := 2^{16} c_0^{-1}(\sqrt{m}+8 B^2\sqrt{k})t \,.
		\end{equation}
		Recalling the definition of our \emph{cylinders} from \eqref{eq:defCy}, we state the following Claim:
		
		\begin{claim}\label{claim:condWalkLCM1}
			There exists $x \in S \subseteq \R^{k+2}$ so that 
			\begin{equation}\label{eq:non-empty}  
			\left( \Gamma_{2r_0,16} \setminus \Gamma_{2r_0,s_0} + x \right) \cap S \neq \emptyset\, .
			\end{equation}
		\end{claim}
		\begin{proof}[Proof of Claim \ref{claim:condWalkLCM1}]
			We will use Lemma \ref{lem:geo-comparison} with $s = s_0$, and so we check the hypotheses.   We first observe that for any $y, a, b$, if $\t_{[k]},\t'_{[k]} \in F_y(S;a,b)$ then we have 
			\[\t'' := (\t_1-\t_1',\ldots,\t_{k}-\t'_{k},0,0) \in S_{W_Y}(4m)\] by Fact \ref{fact:LevelSetTriangleInq}.  This shows that for any $y, a, b$ we have
			\begin{equation}\label{eq:claim-diff2}
			F_y(S;a,b) - F_y(S;a,b) \subset S_{W_Y}(4m) \cap \{ \t \in \R^{k+2} : \t_{k+1} = \t_{k+2} = 0 \} = S_W(4m) \, ,
			\end{equation} where the equality holds by definition of $W_Y$ and the level set $S_{W_Y}$.
			Thus we may apply Lemma~\ref{lem:revEsseen} to obtain 
			
			\begin{equation} \label{eq:S(m)upBound2} 
			\g_{k}(S_W(4m))\leq e^{128 \nu p m}\left(\P(\|W^T\tau'\|_2\leq \beta'\sqrt{k})+\exp(- \beta'^2 k)\right)\,. 
			\end{equation} 
		
		Combining lines \eqref{eq:gS-LB}, \eqref{eq:claim-diff2} and \eqref{eq:S(m)upBound2}, we note that in order to apply Lemma \ref{lem:geo-comparison}, it is sufficient to  check 
			\begin{align}
			8s_0^2 e^{-k/8} &+ 32 s_0^2  e^{32 \nu p m}\left(\P(\|W^T \tau'\|_2\leq \beta'\sqrt{k})+\exp(- \beta'^2 k)\right)^{1/4} \nonumber \\
			&\qquad < \frac{1}{4} e^{\nu p m/2 + 2 \beta^2 k}\left( R t\right)^2 \left(\P(\|W^T \tau'\|_2\leq \beta'\sqrt{k}) + \exp(-\beta'^2 k) \right)^{1/4}\,. \label{eq:numerics-need}
			\end{align}
			
			We will show that each term on the left-hand-side of \eqref{eq:numerics-need} is at most half of the right-hand side. Bound \begin{equation}\label{eq:s0-bound}
			s_0^2 = 2^{32}c_0^{-2}(\sqrt{m} + 8 B^2\sqrt{k})^2t^2 < 2^{33}(m+64 B^4 k)(t/c_0)^2\leq 2^{-20}\nu(c_0^2 k+(2B)^{-6}m)(Rt)^2\end{equation}
			since $R= 2^{35}B^2\nu^{-1/2}c_0^{-2}$. By Lemma \ref{lem:p-bound} we have that $p\geq 2^{-7}B^{-4}$ and so we may bound 
			
			\[8 s_0^2 e^{-k/8} \leq e^{-k/8} 2^{-17}\nu(c_0^2k+(2B)^{-4}m)(Rt)^2\leq \frac{1}{8}e^{\nu p m/2}(Rt)^2 e^{-\beta'^2 k/4}\, .\]
			
			 Similarly, use \eqref{eq:s0-bound}, $c_0\leq \beta$ and $\nu= 2^{-7}\nu $ to bound $$32 s_0^2 e^{32 \nu p m} \leq 2^{-15}(c_0^2k+(2B)^{-4}m)(Rt)^2 \exp(\nu p m/4)   \leq \frac{1}{8}(Rt)^2 e^{\nu p m/2 + \beta^2 k}$$
			thus showing \eqref{eq:numerics-need}.  Applying Lemma \ref{lem:geo-comparison} completes the claim.
		\end{proof}

		The following basic consequence of Claim~\ref{claim:condWalkLCM1} will bring us closer to the construction of our LCD:
		
		\begin{claim} \label{claim:non-empty} 
			We have that $S_{W_Y}(4m) \cap (\Gamma_{2r_0,16} \setminus \Gamma_{2r_0,s_0}) \neq \emptyset\,$.
		\end{claim}
		\begin{proof}[Proof of Claim \ref{claim:non-empty}]
			Claim \ref{claim:condWalkLCM1} shows that there exists $x,y \in S = S_{W_Y}(m)$ so that $y \in ( \G_{2r_0,16} \setminus \G_{2r_0,s_0} + x \big) $.  Now define define $\phi := y-x$ and note that $\phi \in S_{W_Y}(4m) \cap (\G_{2r_0,16} \setminus \G_{2r_0,s_0})$ due to Fact~\ref{fact:LevelSetTriangleInq}.
		\end{proof}

		We now complete the proof of Lemma~\ref{lem:CondWalkLCM} by showing that an element of the non-empty intersection above provides an LCD.
		
		\begin{claim}\label{claim:CondRW3}
			If $\phi \in S_{W_Y}(4m) \cap (\Gamma_{2r_0,16} \setminus \Gamma_{2r_0,s_0})$ then there is a $\zb_0 \in (1,16 B^2)$ and  $i \in \{k+1,k+2\}$ so that
			
			$$ 
			\|\zb_0 \phi_i Y\|_{\T} < \min\{\gamma\zb_0 \phi_i \| Y\|_{2}, \sqrt{\alpha d}\}\,. 
			$$
		\end{claim}
		\begin{proof}[Proof of Claim \ref{claim:CondRW3}]
			Note that since $\phi \in S_{W_Y}(4m)$ we have $$\E_{\zb} \| \zb W_Y \phi\|_{\T}^2 \leq 4m\,.$$
			Thus there is some instance $\zb_0 \in (1,16 B^2)$ of $\zb$ so that \begin{equation}\label{eq:prob-method-zb}
			\| \zb_0 W_Y \phi\|_{\T}^2 \leq 4m\,.
			\end{equation} 
			For simplicity, define $\psi =: \zb_0 \phi$. 
			
			By \eqref{eq:prob-method-zb}, there is a $z \in \Z^{2d}$ so that $W_Y \psi \in B_{2d}(z,2\sqrt{m})$. Expand 
			$$
			W_Y\psi = W\psi_{[k]} + \psi_{k+1}  \begin{bmatrix} Y \\ \mathbf{0}_d \end{bmatrix}  + \psi_{k+2} \begin{bmatrix} \mathbf{0}_d \\ Y \end{bmatrix}\,
			$$
			and note that
			\begin{equation} \label{eq:almostLCM}  
			\psi_{k+1} \begin{bmatrix} Y \\ {\bf{0}}_d \end{bmatrix}  + \psi_{k+2} \begin{bmatrix} {\bf{0}}_d \\ Y \end{bmatrix} \in   
			B_{2d}(z,2\sqrt{m}) -  W \psi_{[k]} \subseteq B_{2d}(z, 2\sqrt{m} + 2^6 B^2\sqrt{k})\,,
			\end{equation}
			where the last inclusion holds because
			\[ \|W\psi_{[k]}\|_2 \leq \|W\|_{op} \|\psi_{[k]}\|_2 \leq 2 |\zb_0| \|\phi_{[k]}\|_2 \leq 32\sqrt{k}B^2, \]
			since $\phi \in \G_{2r_0,16}$, $|\zb_0| \leq 16 B^2$ and $\|W\|_{op} \leq 2$.
			
			Since $\phi \not\in \G_{2r_0,s_0}$ and $\zb_0 > 1$, we have $\max\{|\psi_{k+1}|,|\psi_{k+2}|\} > s_0$ and so we assume, without loss, that $|\psi_{k+1}|>s_0$.
			Projecting \eqref{eq:almostLCM} onto the first $d$ coordinates yields
			\begin{equation}\label{eq:phiY-close} 
			\psi_{k+1} Y \in B_{d}(  z_{[d]} , 2\sqrt{m} + 2^6 B^2\sqrt{k}) . 
			\end{equation}
			
			Now we show that $\|\psi_{k+1} Y\|_{\T} < \gamma \psi_{k+1}\| Y\|_2$. Indeed,
			\begin{equation} \label{eq:non-triv} \psi_{k+1}\| Y\|_2\gamma \geq s_0 \|Y\|_2\gamma > \bigg(\frac{2^{15}(\sqrt{m} + 8B^2 \sqrt{k})t}{c_0}\bigg)\bigg(2^{-10}\frac{c_0}{t}\bigg) 
			\geq (2\sqrt{m} +  2^6 B^2 \sqrt{k}), \end{equation}
			where we used the definition of $s_0$ and that $\|Y\|_2 > 2^{-10}c_0 \gamma^{-1}/t$.

			We now need to show \begin{equation} \label{eq:LCD-show}
			2\sqrt{m} + 2^6 B^2 \sqrt{k} \leq \sqrt{\alpha d}
			\end{equation}
			Note that since $k \leq 2^{-32} \alpha d / B^4$ we have $2^8 B^2 \sqrt{k}  \leq \sqrt{\alpha d}/2$.  We claim that $m \leq 2^{-4}\alpha d$.  To show this, apply the lower bound \eqref{eq:gS-LB} and $\g_{k+2}(S) \leq 1$ to see
			\begin{equation*}
				e^{-2^{-11}\nu m / B^4} \geq  e^{-\nu p m/2} \geq   \g_{k+2}(S) e^{-\nu p m/2} \geq (Rt)^2e^{-2\beta'^2 k} \geq t^2 e^{- k} \geq e^{-2^{-15}\nu\alpha d/ B^4 }, 
			\end{equation*}
			where we have used $k \leq 2^{-17} \nu\alpha d / B^4$ and $t \geq e^{-2^{-17} \nu\alpha d / B^4}$.  Therefore $m \leq 2^{-4}\alpha d$, i.e. $2\sqrt{m} \leq \sqrt{\alpha d}/2$.   Combining this with \eqref{eq:phiY-close} and \eqref{eq:non-triv} we see
			\[ \|\psi_{k+1} Y \|_{\T} \leq  \sqrt{\alpha d }, \]
			as desired. This completes the proof of the Claim~\ref{claim:CondRW3}. \end{proof}
		
		\vspace{3mm}
		
		Let $\phi$, $\zb_0$ and $i \in \{k+1,k+2\}$ be as guaranteed by Claim \ref{claim:CondRW3}. Then $\zb_0\phi_i \leq 2^{10} B^2 $, and 
		\[ \|\zb_0 \phi_i  Y\|_{\T} < \min\{\|\zb_0 \phi_i Y\|_{2}\gamma, \sqrt{\alpha d}\},\] and so $D_{\alpha,\gamma}(Y)\leq 2^{10}B^2$ thus completing the proof of Lemma~\ref{lem:CondWalkLCM}.
	\end{proof}
	
	\subsection{Proof of Lemma \ref{lem:CondWalkLCMfinal}}\label{subsec:proofOfCondWalkLCMfinal}

	In order to bridge the gap between Lemmas \ref{lem:CondWalkLCM} and \ref{lem:CondWalkLCMfinal}, we need an anticoncentration lemma for $\| W \sigma \|_2$ when $\sigma$ is random and $W$ is fixed.  We will use the following bound, which is a version of the Hanson-Wright inequality \cite{Hanson-Wright,RV-HW}.

	\begin{lemma}\label{lem:HW}
		
		Let $\nu\in(0,1)$ and $\beta'\in(0,2^{-7}B^{-2}\sqrt{\nu })$. Let $W$ be a $2d \times k$ matrix satisfying $\|W \|_{\HS} \geq \sqrt{k}/2$ and $\| W \| \leq 2$ and $\tau'\sim \Xi_\nu(2d; \zeta)$.  Then $$\P( \| W^T \tau' \|_2 \leq \beta' \sqrt{k}) \leq 4 \exp\left(-2^{-20} B^{-4}\nu k  \right)\,.$$
	\end{lemma}

	We derive Lemma \ref{lem:HW} from Talagrand's inequality in Section \ref{app:HW}, (see \cite{RV-HW} or \cite{Hanson-Wright} for more context).  From here, we are ready to prove Lemma \ref{lem:CondWalkLCMfinal}.

	\begin{proof}[Proof of Lemma \ref{lem:CondWalkLCMfinal}]
		Recalling that $c_0\leq 2^{-35} B^{-4}\nu$, and that our given $W$ satisfies $\|W\|_{\HS}\geq \sqrt{k}/2$ and $\|W\|\leq 2$, we apply Lemma~\ref{lem:HW}, with $\beta'=2^{6}\sqrt{c_0}$ and the $\nu$-lazy random vector $\tau'\sim \Xi_\nu(2d;\zeta)$, where $\nu = 2^{-7}\nu$, to see
		\begin{equation*}\P(\|W^T\tau'\|_2\leq \beta' \sqrt{k})\leq 4\exp\left(-2^{-27} B^{-4}\nu k  \right)\leq 4\exp(- 32c_0 k).  
		\end{equation*}
		We now consider the right-hand-side of \eqref{eq:LCM-hypo-lazy} in Lemma~\ref{lem:CondWalkLCM}: if $\beta\leq\sqrt{c_0}$ we have 
		\begin{align*}
		e^{4\beta^2 k}\left(\P(\|W^T\tau'\|_2 
		\leq \beta' \sqrt{k})+\exp(-\beta'^2 k)\right)^{1/4}
		&\leq \exp\left(4 c_0 k-8 c_0 k\right)+\exp\left(4 c_0 k-16 c_0 k\right)\\ 
		&\leq 2\exp(-c_0 k)\,.
		\end{align*}
		We now note that the hypotheses in Lemma~\ref{lem:CondWalkLCMfinal}  align with the hypotheses in Lemma~\ref{lem:CondWalkLCM} with respect to the selection of $\beta, \alpha, t, R, Y, W$; if we additionally assume $D_{\alpha,\gamma}(Y) > 2^{10}B^2$, we may apply the contrapositive of Lemma~\ref{lem:CondWalkLCM} to obtain 
		\begin{align*}
		\cL\left(W_Y^T \tau ,  \beta \sqrt{k+1} \right)  &\leq(2^{35} B^{2}\nu^{-1/2} c_0^{-2} t/2)^{2} e^{4 \beta^2 k} \left(\P(\| W^T \tau' \|_2 \leq 2 \beta' \sqrt{k}) + e^{- \beta'^2 k} \right)^{1/4}\\
			 &\leq  (Rt)^2 \exp(-c_0k) \,, 
		\end{align*}
		as desired.
	\end{proof}
	
	\section{Inverse Littlewood-Offord for conditioned matrix-walks}\label{sec:ILwO-Matrix}
	
	In this section we prove an inverse Littlewood-Offord theorem for matrices conditioned on their robust rank. Everything in this section will be analogous to section 6 of~\cite{RSM2}.
	
	\begin{theorem}\label{lem:rankH} 
		For $n \in \N$ and $0 < c_0 \leq 2^{-50}B^{-4}$, let $d \leq c_0^2 n$, and for $\alpha,\gamma \in (0,1)$, let $0\leq k\leq 2^{-32}B^{-4}\alpha d$ and $N\leq \exp(2^{-32}B^{-4}\alpha d)$. 
		Let $X \in \R^d$ satisfy $\|X\|_2 \geq c_02^{-10} \gamma^{-1}n^{1/2} N$, and let $H$ be a random $(n-d)\times 2d$ matrix with i.i.d.\ rows sampled from $\Phi_\nu(2d;\zeta)$ with $\nu = 2^{-15}$.
		If $D_{\alpha,\gamma}(r_n \cdot X)> 2^{10}B^2$ then
		\begin{equation} \label{eq:RankofH}
			\P_H\left(\sigma_{2d-k+1}(H)\leq c_02^{-4}\sqrt{n} \text{ and } \|H_1X\|_2,\|H_2 X\|_2\leq n\right)\leq e^{-c_0nk/3}\left(\frac{R}{N}\right)^{2n-2d}\, ,
		\end{equation}
		where we have set $H_1 := H_{[n-d]\times [d]}$, $H_2 := H_{[n-d] \times [d+1,2d]}$, $r_n := \frac{c_0}{32\sqrt{n}}$ and $R := 2^{43}B^2 c_0^{-3}$. 
	\end{theorem}

	\subsection{Tensorization and random rounding step}
	We import the following tensorization Lemma from~\cite{RSM2}.
	
	\begin{lemma}\label{lem:tensor} For $d < n$ and $k \geq 0$, let $W$ be a $2d \times (k+2)$ matrix and let $H$ be a $(n-d)\times 2d$ random matrix with i.i.d.\ rows. Let $\tau \in \R^{2d}$ be a random vector with the same distribution as the rows of $H$.
		If $\beta \in (0,1/8)$ then
		\begin{equation*}  
			\P_H\big( \|HW\|_{\HS} \leq \beta^2 \sqrt{(k+1)(n-d)} \big)  \leq \left(2^{5}e^{2\beta^2 k}\cL\big( W^T \tau, \beta \sqrt{k+1} \big)\right)^{n-d}. 
		\end{equation*} 
	\end{lemma}
	
	Similarly we use net for the singular vectors of $H$, constructed in~\cite{RSM2}. Let $\cU_{2d,k} \subset \R^{[2d] \times [k]}$ be the set of $2d \times k$ matrices with orthonormal columns.  
	
	\begin{lemma}\label{lem:basis-net}
		For $k \leq d$ and $\delta \in (0,1/2)$, there exists $\cW = \cW_{2d,k} \subset \R^{[2d]\times [k]}$ with $|\cW| \leq (2^6/\delta)^{2dk}$ 
		so that for any $U\in \cU_{2d,k}$, any $r \in \N$ and $r \times 2d$ matrix $A$ there exists $W\in \cW$ so that 
		\begin{enumerate}
			\item \label{it:rr-mat-1} $\|A(W-U)\|_{\HS}\leq \delta(k/2d)^{1/2}  \|A\|_{\HS} $, 
			\item \label{it:rr-mat-2} $\|W-U\|_{\HS}\leq \delta \sqrt{k}$ and
			\item \label{it:rr-mat-3} $\|W-U\|_{op} \leq 8\delta .$
	\end{enumerate}\end{lemma}

	\subsection{Proof of Theorem~\ref{lem:rankH}}
	
	We also use the following standard fact from linear algebra.
	\begin{fact}\label{fact:singvalues}
		For $3d < n$, let $H$ be a $(n-d) \times 2d$ matrix. If $\s_{2d-k+1}(H) \leq x$
		then there exist $k$ orthogonal unit vectors $w_1,\ldots,w_k \in \R^{2d}$ so that $\|Hw_i\|_2 \leq x$. In particular, there exists $W \in \cU_{2d,k}$ so that 
		$\|HW\|_{\HS} \leq x\sqrt{k}$.
	\end{fact}

	We will also need a bound on $\|H\|_{\HS}$:
	\begin{fact}\label{fact:H-HS} Let $H$ be the random $(n - d) \times (2d)$ matrix whose rows are i.i.d.\ samples of $\Phi_\nu(2d; \zeta)$.  Then
		\[\P(\|H\|_{\HS}\geq 2\sqrt{ d (n-d)})\leq 2\exp\left(-2^{-21}B^{-4}nd\right)\]
	\end{fact}
	
	We are now ready to prove Theorem~\ref{lem:rankH}.
	
	\begin{proof}[Proof of Theorem~\ref{lem:rankH}]
		Let $Y := \frac{c_0}{32\sqrt{n}}\cdot X$.  We may upper bound the left-hand side of \eqref{eq:RankofH} by Fact~\ref{fact:singvalues}
		\begin{align*}
			\P(&\sigma_{2d-k+1}(H)\leq c_02^{-4}\sqrt{n} \text{ and } \|H_1 X\|_2,\|H_2 X\|_2\leq n) \\
			&\qquad \leq \P(\exists U\in \cU_{2d,k}: \|H U_Y\|_{\HS} \leq c_0\sqrt{n (k+1)}/8). 
		\end{align*}
		Set $\delta := c_0/16$, and let $\cW$ be as in 
		Lemma \ref{lem:basis-net}.
		
		For each fixed $H$, if we have $\|H\|_{\HS}\leq 2\sqrt{ d (n-d)}$ and there is some $U \in \cU_{2d,k}$ so that $\|HU_Y\|_{\HS} \leq c_0\sqrt{n (k+1)}/8$, we may apply Lemma~\ref{lem:basis-net} to find $W \in \cW$ so that 
		\begin{equation*} \|HW_Y\|_{\HS} \leq \|H(W_Y-U_Y)\|_{\HS} + \|HU_Y\|_{\HS} \leq \delta(k/2d)^{1/2} \|H\|_{\HS}+ c_0\sqrt{n(k+1)}/8 
		\end{equation*}
		which is at most $c_0\sqrt{n(k+1)}/4$.  This shows the bound
		\begin{equation*} \P_H\left( \exists U\in \cU_{2d,k}:~\|H U_Y\|_{\HS} \leq c_0\sqrt{n(k+1)}/8 \right)
		\leq \P_H\left( \exists W \in \cW : \|H W_Y\|_{\HS} \leq c_0\sqrt{n(k+1)}/4\right) .  
		\end{equation*}
		Conditioning on the event that $\| H \|_{\HS} \leq 2\sqrt{d(n-d)}$, applying Fact \ref{fact:H-HS}, and union bounding over $\cW$ shows that the right-hand-side of the above is at most
		\begin{equation*}\sum_{W\in \cW}\P_H\left( \|H W_Y\|_2\leq c_0\sqrt{n(k+1)}/4 \right)+2\exp\left(-2^{-21}B^{-4}nd\right)\, . 
		\end{equation*}
		Bound 
		\begin{equation*} 
			|\cW| \leq (2^6/\delta)^{2dk} \leq  \exp( 32 dk\log c_0^{-1} )  \leq  \exp( c_0 k(n-d)/6), 
		\end{equation*} 
		where the last inequality holds since $d\leq c_0^2 n$.  Thus 
		\begin{equation}\label{eq:postTensor}  
			\sum_{W\in \cW}\P_H(\|H W_Y\|_2\leq c_0\sqrt{n(k+1)}/4) \leq \exp(c_0 k(n-d)/6)\max_{W\in \cW}\P_H(\|H W\|_2\leq c_0\sqrt{n(k+1)}/4).
		\end{equation}
		For each $W \in \cW$ apply Lemma~\ref{lem:tensor} with $\beta :=\sqrt{c_0/3}$ (noting that $\sqrt{n-d}/3\geq \sqrt{n}/4$) to obtain
		\begin{equation}\label{eq:tensorapp}
			\P_H(\|H W_Y\|_2\leq c_0\sqrt{n(k+1)}/4)\leq \left(2^{5}e^{2c_0 k/3}\cL\big( W_Y^T \tau, c_0^{1/2} \sqrt{k+1} \big)\right)^{n-d} \,.
		\end{equation}
		
		Preparing to apply Lemma~\ref{lem:CondWalkLCMfinal}, define $t := (c_0 N/32)^{-1} \geq \exp(- 2^{-32}B^{-4}\alpha d)$ and \\ 
		$R_0 := 2^{-8}c_0 R = 2^{-8}c_0(2^{43}B^2c_0^{-3}) =  2^{35}B^2c_0^{-2}$ so that we have 
		\begin{equation*}
			\|Y\|_2=c_0\|X\|_2/(32n^{1/2}) \geq 2^{-15}c_0^2 N \gamma^{-1} = 2^{-10}c_0\gamma^{-1}/t \,.
		\end{equation*}
	 Since $W \in \cW$, we have $\|W\|_{op}\leq 2$ and $\|W\|_{\HS} \geq \sqrt{k}/2$.  We also note the bounds $k \leq 2^{-32}B^{-4}\alpha d$, $ D_{\alpha,\gamma}(\frac{c_0}{32\sqrt{n}} X) = D_{\alpha,\gamma}(Y) > 2^{10}B^2$.  Thus, we may apply Lemma~\ref{lem:CondWalkLCMfinal} to see that
		\begin{equation*}  
			\cL\big( W_Y^T \tau, c_0^{1/2} \sqrt{k+1} \big)\leq (R_0t)^2e^{-c_0k}\leq \left(\frac{R}{8N}\right)^2e^{-c_0k}\,. 
		\end{equation*}
		Substituting this bound into \eqref{eq:tensorapp} gives
		\begin{equation*} \max_{W \in \cW }\, \P_H(\|H W_Y\|_2\leq c_0 \sqrt{n (k+1)}/4 )\leq \frac{1}{2}\left(\frac{R}{N}\right)^{2n-2d}e^{-c_0 k(n-d)/3}\,.
		\end{equation*} 
		Combining with the previous bounds and noting 
		$$2\exp\left(-2^{-21}B^{-4}nd\right)\leq\frac{1}{2}\left(\frac{R}{N}\right)^{2n-2d}e^{-c_0 k(n-d)/3}$$ shows
		\begin{equation*}
			\P(\sigma_{2d-k+1}(H)\leq c_0\sqrt{n}/16 \text{ and } \|H_1 X\|_2,\|H_2 X\|_2\leq n)\leq \left(\frac{R}{ N}\right)^{2n-2d}e^{-c_0 k(n-d)/3}\,.
		\end{equation*}
		This completes the proof of Theorem~\ref{lem:rankH}.
	\end{proof}

	\section{Nets for structured vectors: Size of the Net }\label{sec:sizenet}
	
	The goal of this subsection is to prove Theorem \ref{thm:netThm}.  We follow the same path as Section 7 of \cite{RSM2}.  As such, we work with the intersection of 
	$\cN_{\eps}$ with a selection of ``boxes'' which cover  a rescaling of the trivial net $\L_{\eps}$.  We recall the definition of the relevant boxes from \cite{RSM2}.
	\begin{definition}
		Define a $(N,\kappa,d)$-\emph{box} to be a set of the form $\mathcal{B}=B_1 \times \ldots \times B_n\subset \mathbb Z^n$ where
		$|B_i|\geq N$ for all $i\geq 1$;  $B_i = [-\kappa N,-N]\cup[N, \kappa N]$, for $i \in [d]$; and  $|\cB|\leq (\kappa N)^n$.
	\end{definition}

	We now interpret these boxes probabilistically and seek to understand the probability that we have 
	\[ \P_M(\|MX\|_2\leq n)\geq \left(\frac{L}{N}\right)^n, \]
	where $X$ is chosen uniformly at random from $\cB$.  Theorem \ref{thm:netThm} will follow quickly from the following ``box'' version:
	
	\begin{lemma}\label{thm:invertrandom} For $L \geq 2$ and $0 < c_0 \leq 2^{-50}B^{-4}$, let $n > L^{64/c_0^2}$ and let $\frac{1}{4}c_0^2n\leq d\leq c_0^2 n$. 
		For $N \geq 2$, satisfying $N \leq \exp(c_0 L^{-8n/d} d)$, and $\k \geq 2$, let $\cB$ be a $(N,\kappa ,d)$-box.  If $X$ is chosen uniformly at random from $\cB$ then 
		\begin{equation*}\P_X\left(\P_M(\|MX\|_2\leq n)\geq \left(\frac{L}{N}\right)^n\right)\leq \left(\frac{R}{L}\right)^{2n},
		\end{equation*}
		where $R := C c_0^{-3}$ and $C>0$ is an absolute constant.
	\end{lemma}
	
	\subsection{Counting with the LCD and anti-concentration for linear projections of random vectors}
	We first show that if we choose $X \in \cB$ uniformly at random, then it typically has a large LCD.

	\begin{lemma}\label{lem:lcd-rare}
		For $\alpha \in (0,1), K \geq 1$ and $\k \geq 2$, let $n \geq d\geq K^2/\alpha$ and let $N \geq 2$ be so that $ K N < 2^d $. 
		Let $\cB=\left([-\k N,-N]\cup [N,\k N]\right)^d$ and let $X$ be chosen uniformly at random from $\cB$. Then
		\begin{equation} \label{eq:lcd-rare} 
			\P_X\left( D_{\alpha,\gamma}\big( r_n  X \big) \leq K \right) \leq (2^{20} \alpha)^{d/4}\, ,\end{equation}
		where we have set $r_n := c_02^{-5} n^{-1/2}$.
	\end{lemma}
	
	\begin{proof}
		Writing $\phi = \psi r_n$, note that 
		\begin{equation*}
		\P_X\big( D_{\alpha,\gamma}(r_nX) \leq K \big) 
		= \P\big(\, \exists~\phi \in (0,Kr_n] : \|\phi X \|_{\T} < \min \{\gamma \phi \|X\|_2, \sqrt{\alpha d}  \} \big)\,.
	\end{equation*}
		We note that any such $\phi$ must have $|\phi| \geq (2 \kappa N)^{-1}$, since if we had $\phi < (2 \kappa N)^{-1}$ then each coordinate of $\phi X$ would lie in $(-1/2,1/2)$, implying $\|\phi X\|_{\T} = \phi\| X\|_2$, i.e.\ $\|\phi X \|_{\T} > \gamma \phi \|X \|_2$.  The proof of Lemma 7.4 in \cite{RSM2} shows that \begin{equation*}
			\P_X\big(\, \exists~\phi \in [(2\kappa N)^{-1},r_n K] :  \|\phi X \|_{\T} < \sqrt{\alpha d} \big)	\leq (2^{20} \alpha)^{d/4}
		\end{equation*} completing the Lemma.
	\end{proof}
	
\vspace{2mm}	
	
		We also import from \cite[Lemma 7.5]{RSM2} a result showing anti-concentration for random vectors $AX$, where $A$ is a fixed matrix and $X$ is a random vector with 
		independent entries.  As noted in \cite{RSM2}, this is essentially a rephrasing of Corollary 1.4 and Remark 2.3 in Rudelson and Vershynin's paper \cite{rudelson2015small}: 

		\begin{lemma}\label{lem:LwO-for-AX}
			Let $N \in  \N$, $n,d,k \in \N$ be such that $n-d \geq 2d > 2k$, $H$ be a $2d \times (n-d)$ matrix with $\s_{2d-k}(H)\geq c_0\sqrt{n}/16$ and $B_1,\ldots, B_{n-d}\subset \Z$ with $|B_i|\geq N$. 
			If $X$ is taken uniformly at random from $\cB:=B_1\times \ldots \times B_{n-d}$, then
			\[ \P_X(\|HX\|_2\leq n)\leq \left(\frac{Cn}{dc_0 N}\right)^{2d-k},\]
			where $C>0$ is an absolute constant. 
		\end{lemma}

		\subsection{Proof of Theorem~\ref{thm:invertrandom} }
			
	Recall that the matrix $M$ is defined as
		\begin{equation*} M =  		\begin{bmatrix}
			{\bf 0 }_{[d]\times [d]} & H^T_1 \\
			H_1 & { \bf 0}_{[n-d] \times [n-d]},  
		\end{bmatrix}
	 	\end{equation*}
		where $H_1$ is a $(n-d) \times d$ random matrix with whose entries are i.i.d.\ copies of $\tz Z_\nu$. Let $H_2$ be an independent copy of $H_1$ and define $H$ to be the $ (n-d) \times 2d $ matrix 
		\begin{equation*}H := \begin{bmatrix}
			H_1 & H_2 \end{bmatrix} .
		\end{equation*}
		For a vector $X \in \R^n$, we define the events $\cA_1 = \cA_1(X)$ and $\cA_2 = \cA_2(X)$  by
		\begin{align*}
			 \cA_1 &:= \left\{ H :  \|H_1 X_{[d]}\|_2\leq n \text{ and } \|H_{2} X_{[d]}\|_2\leq n \right\} \\
			 \cA_2 &:= \left\{ H : \|H^T X_{[d+1,n]}\|_2\leq 2n \right\}\,. 
			\end{align*}

		We now note a simple bound on $\P_M(\|MX\|_2 \leq n)$ in terms of $\cA_1$ and $\cA_2$.
		
		\begin{fact}\label{fact:2ndMoment} For $X \in \R^n$, let $\cA_1 =\cA_1(X)$, $\cA_2 = \cA_2(X)$ be as above. We have
			\begin{equation*} 
				\left( \P_M(\|M X \|_2 \leq n) \right)^2 \leq \P_{H}(\cA_1 \cap \cA_2) .
			 \end{equation*}
		\end{fact}
		
		This fact is a straightforward consequence of Fubini's theorem, the details of which are in \cite[Fact 7.7]{RSM2}. 
		We shall also need the robust'' notion of the rank of the matrix $H$ used in \cite{RSM2}:  for $k = 0,\ldots,2k$ define $\cE_k$ to be the event 
		\begin{equation*} \cE_k := \left\{ H : \sigma_{2d-k}(H)\geq c_0\sqrt{n}/16 \text{ and } \sigma_{2d-k+1}(H)\leq c_0\sqrt{n}/16 \right\}
		 \end{equation*}
		and note that always at least one of the events $\cE_0,\ldots,\cE_{2d}$ holds.
		
		We now define 
		\begin{equation} \label{eq:defalpha}
			 \alpha:= 2^{13}L^{-8n/d} 
		 \end{equation}
		and for a given box $\cB$ we define the set of typical vectors $T(\cB) \subseteq \cB$ by
		\begin{equation*} 
			T = T(\cB) := \left\{ X \in \cB :  D_{\alpha}(c_0 X_{[d]}/(32\sqrt{n})) > 2^{10}B^2 \right\}. 
		\end{equation*}
		Now set $K:=2^{10}B^2$ and note the following implication of Lemma~\ref{lem:lcd-rare}: if $X$ is chosen uniformly from $\cB$ and $n  \geq L^{64/c_0^2}\geq 2^{10}B^2/\alpha$ then we have that
		\begin{equation}\label{eq:Tbd}
			\P_X(X\not \in T)=\P_X(D_{\alpha}(c_0 X_{[d]}/(32\sqrt{n})) \leq 2^{10}B^2)\leq \left(2^{33}L^{-8n/d}\right)^{d/4}\leq \left(\frac{2}{L}\right)^{2n}.
		\end{equation}
		
		\begin{proof}[Proof of Lemma~\ref{thm:invertrandom}]
			Let $M$, $H_1,H_2$, $H$, $\cA_1,\cA_2$, $\cE_k$, $\alpha$ and $T := T(\cB)$ be as above. Define
			\begin{equation*}
				\cE :=  \left\{X \in \cB : \P_M(\|MX\|_2\leq n)  \geq (L/N)^n\right\} 
			\end{equation*} 
			and bound
			\begin{equation*} \P_X( \cE )  \leq \P_X( \cE  \cap \{ X \in T \} ) + \P_X( X \not\in T)\,.  
			\end{equation*}
			For each $X$ define 
			\begin{equation*}
				f(X) := \P_M(\| MX\|_2 \leq n)\1( X \in T ) 
			\end{equation*}
			and apply~\eqref{eq:Tbd} to bound
			\begin{equation}\label{eq:ProbE-bound}  
				\P_X( \cE ) \leq \P_X\left( f(X) \geq (L/N)^n\right)  + (2/L)^{2n} \leq (N/L)^{2n}\E_X\, f(X)^2 + (2/L)^{2n}, 
			\end{equation}
			where the last inequality follows from Markov's inequality.  Thus, in order to prove Lemma \ref{thm:invertrandom} it is enough to 
			prove $\E_X\, f(X)^2 \leq 2(R/N)^{2n}$.
			
			Apply Fact~\ref{fact:2ndMoment} to write
			\begin{equation} \label{eq:PMexpress}
				\P_M(\|M X \|_2 \leq n)^2  \leq \P_H(\cA_1 \cap \cA_2) = \sum_{k=0}^d \P_H( \cA_2 | \cA_1 \cap \cE_k)\P_H(\cA_1 \cap \cE_k) \end{equation}
			and so 
			\begin{equation}\label{eq:fsquare} 
				f(X)^2 \leq \sum_{k=0}^d \P_H( \cA_2 | \cA_1 \cap \cE_k)\P_H(\cA_1 \cap \cE_k)\1( X \in  T) . \end{equation}
			
			We will now apply Theorem~\ref{lem:rankH} to upper bound $\P_H(\cA_1 \cap \cE_k)$ for $X \in T$.  For this, note that $d\leq c_0^2 n$, $N\leq \exp(c_0L^{-8n/d}d)\leq \exp(2^{-32}B^{-4}\alpha n)$ and set $R_0 := 2^{43}B^2c_0^{-3}$.  Also note that
			by the definition of a $(N,\kappa,d)$-box and the fact that $d\geq \frac{1}{4}c_0^2 n$, we have that $\|X_{[d]}\|_2 \geq d^{1/2}N \geq c_02^{-10}\sqrt{n}N$. Now set 
			$\alpha':=2^{-32}B^{-4}\alpha$ and apply Theorem \ref{lem:rankH} to see that for $X \in T$ and $0\leq k \leq \alpha' d$, we have
			\begin{equation*}
			\P_H(\cA_1 \cap \cE_k ) \leq e^{-c_0 n k/3}\left(\frac{R_0}{N} \right)^{2n-2d}\, . 
		\end{equation*}
			Additionally by Theorem~\ref{lem:rankH} we may bound the tail sum:
			\begin{equation*} \sum_{k \geq \alpha' d} \P_H(\cA_1 \cap \cE_k) \leq \P_H\big( \{ \sigma_{2d-\alpha' d}(H) \leq c_0\sqrt{n}/16 \} \cap \cA_1  \big) \leq e^{-c_0 \alpha' dn/4}.
			\end{equation*}
			Thus, for all $X \in \cB$, the previous two equations bound
			\begin{equation}\label{eq:sing-sq-uncond}
				f(X)^2 \leq  \sum_{k = 0}^{\alpha' d} \P_H(\cA_2 \,|\, \cA_1 \cap \cE_k)e^{-c_0 n k/3}\left(\frac{R_0}{N}\right)^{2n-2d} + e^{-c_0 \alpha' dn/3}\,.
			\end{equation}
			Seeking to bound the right-hand side of \eqref{eq:sing-sq-uncond}, define $g_k(X) := \P_H(\cA_2 \,|\,\cA_1 \cap \cE_k)$.  Write
			\begin{equation*}\E_X[ g_k(X) ] = \E_X \E_H\big[ \cA_2 \,|\,\cA_1 \cap \cE_k \big] = \E_{X_{[d]}}\, \E_H\left[ \E_{X_{[d+1,n]}} \one[\cA_2] \,\big\vert\, \cA_1 \cap \cE_k \right]\,.
			\end{equation*}
			Let $k \leq \alpha'd$.  Note that each $H \in \cA_1 \cap \cE_k$ has $\sigma_{2d-k}(H) \geq c_0 \sqrt{n}/16$ and thus we may apply Lemma~\ref{lem:LwO-for-AX} to bound 
			\begin{equation*}
			\E_{X_{[d+1,n]}}\, \one[\cA_2] = \P_{X_{[d+1,n]}}( \|H^T X_{[d+1,n]} \|_2 \leq n ) \leq \left(\frac{C'n}{c_0 d N}\right)^{2d - k} \leq \left(\frac{4C'}{c_0^3 N}\right)^{2d - k}  \end{equation*}
			for an absolute constant $C'>0$, where we used that $d\geq \frac{1}{4}c_0^2 n$. Thus, for each $0\leq k \leq \alpha' d$, if we define $R := \max\{ 8C' c_0^{-3}, 2R_0\} $ then we have
			\begin{equation} \label{eq:gk-bnd} 
				\E_X[ g_k(X) ] \leq \left(\frac{R}{2N}\right)^{2d - k}\,. 
			\end{equation}
			Applying $\E_X$ to \eqref{eq:sing-sq-uncond} using \eqref{eq:gk-bnd}  shows 
			\begin{equation*} 
				\E_X f(X)^2 \leq \left(\frac{R}{2N}\right)^{2n} \sum_{k=0}^{\alpha' d} \left(\frac{2N}{R}\right)^k e^{-c_0nk/3}  + e^{-c_0 \alpha' dn/3}\,. 
			\end{equation*}
			Using that  $N\leq e^{c_0L^{-8n/d} d}= e^{c_0\alpha' d/8}$ and $N \leq e^{c_0 n /3}$ bounds 
			\begin{equation}\label{Ef2-bnd} 
				\E_X\, f(X)^2 \leq   2 \left(\frac{R}{2N}\right)^{2n}. 
			\end{equation}
			Combining \eqref{Ef2-bnd} with \eqref{eq:ProbE-bound} completes the proof of Lemma~\ref{thm:invertrandom}.
		\end{proof}

	\subsection{Proof of Theorem~\ref{thm:netThm}}
	
	The main work of proving Theorem \ref{thm:netThm} is now complete with the proof of Lemma~\ref{thm:invertrandom}.  In order to complete it, we need to cover the sphere with a suitable set of boxes.  Recall the definitions from Section \ref{ss:efficient-nets}:
	\begin{equation*} \cI'([d])  := \left\{ v \in \R^{n} :  \k_0 n^{-1/2} \leq |v_i| \leq  \k_1 n^{-1/2} \text{ for all } i\in [d]   \right\}, \end{equation*}
	and 
	\begin{equation*} \Lambda_{\eps} := B_n(0,2) \cap \big(4 \eps n^{-1/2} \cdot \Z^n\big) \cap \cI'([d])\,, 
	\end{equation*}
	and that the constants $\k_0,\k_1$ satisfy $0 < \k_0 < 1 < \k_1$ and are defined in Section~\ref{ss:compressibility}.
	
	We import the following simple covering lemma from \cite[Lemma 7.8]{RSM2}
	\begin{lemma}\label{lem:covZBall} For all $\eps\in[0,1]$, $\k \geq \max\{\k_1/\k_0,2^8 \kappa_0^{-4} \}$, there exists a family $\cF $ of $(N,\k,d)$-boxes with $|\cF| \leq \k^n$ so that 
		\begin{equation}\label{eq:covZBall} \L_{\eps} \subseteq  \bigcup_{\cB \in \cF} (4\eps n^{-1/2}) \cdot \cB\, , 
		\end{equation}
		where $N =  \k_{0}/(4\eps)$.
	\end{lemma}
	
	Combining Lemma~\ref{lem:covZBall} with Lemma~\ref{thm:invertrandom} will imply Theorem~\ref{thm:netThm}.
	
	\begin{proof}[Proof of Theorem~\ref{thm:netThm}]
		Apply Lemma~\ref{lem:covZBall} with $\kappa = \max\{\k_1/\k_0,2^8 \kappa_0^{-4} \}$ and use the fact that $\cN_{\eps} \subseteq \L_{\eps}$ to write 
		\begin{equation*} 
			\cN_{\eps} \subseteq \bigcup_{\cB \in \cF} \left(  (4\eps n^{-1/2}) \cdot  \cB \right) \cap \cN_{\eps}   
		\end{equation*}
		and so
		\begin{equation*} |\cN_{\eps}| \leq \sum_{\cB \in \cF} | (4\eps n^{-1/2} \cdot \mathcal B ) \cap \mathcal N_{\eps}| 
		\leq |\cF| \cdot \max_{\cB \in \cF}\, | (4\eps n^{-1/2} \cdot \mathcal B ) \cap \mathcal N_{\eps}|\,. 
		\end{equation*} 
		Rescaling by $\sqrt{n}/(4\eps)$ and applying Lemma~\ref{thm:invertrandom} bounds
		\begin{equation*} | (4\eps n^{-1/2} \cdot \mathcal B ) \cap \mathcal N_{\eps}| 
		\leq \left| \left\{ X \in \cB : \P_M(\|MX\|_2\leq n) \geq (L\eps)^n \right\} \right| \leq  \left(\frac{R}{L} \right)^{2n} |\cB|. \end{equation*}
	To see that the application of Lemma~\ref{thm:invertrandom} is justified, note that $0 < c_0 \leq 2^{-50}B^{-4}$, $c_0^2 n/2 \leq d \leq c_0^2 n$, $\k \geq 2$, and $\log 1/\eps \leq n/L^{64/c_0^2}$ and so
		\begin{equation*} \log N = \log \k_0/(4\eps) \leq n/L^{64/c_0^2} \leq c_0L^{-8n/d}d\,,
		\end{equation*} as required by Lemma~\ref{thm:invertrandom}, since $\k_0<1$, $d\geq L^{-1/c_0^2}n$, $c_0\geq L^{-1/c_0^2}$ and $8n/d\leq 16/c_0^2$. 
		Using that $|\cF| \leq \k^{n}$ and $|\cB| \leq (\k N)^n$ for each $\cB \in \cF$ bounds
		\begin{equation*}
			 |\cN_{\eps}| \leq \k^{n} \left(\frac{R}{L} \right)^{2n} |\mathcal B| 
		\leq \k^{n}\left(\frac{R}{L} \right)^{2n} (\k N)^n \leq \left(\frac{C}{c_0^6L^2\eps}\right)^{n},
	\end{equation*} where we set $C:=\kappa^2 R^2c_0^{6}$.  This completes the proof of Theorem~\ref{thm:netThm}.
	\end{proof}

	\section{Nets for structured vectors: approximating with the net}
	\label{sec:approxnet}
	In this section we prove Lemma~\ref{thmnet} which tells us that $\cN_{\eps}$ is a net for $\Sigma_{\eps}$. The proof uses the random rounding technique developed by Livshyts \cite{livshyts2018smallest} in the same way as in \cite{RSM2}.

\begin{proof}[Proof of Lemma \ref{thmnet}]
Given $v \in \Sigma_{\eps}$, we define a random variable $r = (r_1,\ldots,r_n)$ where the $r_i$ are independent and satisfy $\E\,r_i = 0 $ as well as the deterministic properties $|r_i| \leq  4\eps n^{-1/2}$
and $v - r\in 4 \eps n^{-1/2} \Z^n$.   We then define the random variable $u := v - r$. We will show that with positive probability that $u\in \cN_{\eps}$.

By definition, $\|r\|_{\infty} = \|u - v\|_{\infty} \leq 4\eps n^{-1/2}$ for all $u$.  Also, $u \in \cI'([d])$ for all $u$, since $v \in \cI([d])$ and 
$\|u-v\|_{\infty} \leq 4\eps/\sqrt{n} \leq \k_0/(2\sqrt{n})$. 
Thus, from the definition of $\cN_{\eps}$, we need only show that with positive probability $u$ satisfies
\begin{equation}\label{eq:lem-net-goal} 
	\P(\|Mu\|_2\leq 4\eps\sqrt{n}) \geq (L\eps)^n   \text{ and }   \cL_{A,op}(u,\eps \sqrt{n}) \leq (2^{10} L\eps)^n. \end{equation}
We first show that \emph{all} $u$ satisfy the upper bound at \eqref{eq:lem-net-goal}. To see this,  recall 
 $\cK = \{\|A\|_{\text{op}}\leq 4\sqrt{n} \}$ and let $w(u) \in \R^n$, be such that  
\begin{align*}
 \cL_{A,op}(u,\eps \sqrt{n}) &= \P^{\cK}\left( \|Av - Ar - w(u)\| \leq \eps \sqrt{n} \right) \\
  &\leq \P^{\cK}\left( \|Av - w(u)\| \leq 17\eps \sqrt{n}  \right) \\
  &\leq \cL_{A,op}(v,17\eps\sqrt{n} )  \leq \cL(Av, 17\eps\sqrt{n}). \end{align*}
  Since $v \in \Sigma_{\eps}$, Lemma~\ref{lem:replacement} bounds 
  \begin{align} 
  \cL(Av, 17\eps\sqrt{n})\leq ( 2^{10} L \eps)^n\, .
  \end{align}
   
We now show that 
\begin{equation} \label{eq:lemnetEgoal}
	 \E_u\,  \P_M(\|Mu\|_2\leq 4\eps\sqrt{n}) \geq (1/2)\P_M(\|Mv\|_2\leq 2\eps\sqrt{n}) \geq (1/4)(2\eps L)^n  
	 \, , \end{equation}
where the last inequality holds by the fact $v \in \Sigma_{\eps}$. From \eqref{eq:lemnetEgoal}, it then follows that there is some $u \in \L_{\eps}$ satisfying \eqref{eq:lem-net-goal}.

To prove the first inequality in \eqref{eq:lem-net-goal}, define the event \[\cE := \{ M : \|Mv\|_2 \leq 2\eps \sqrt{n} \text{ and } \|M\|_{\HS}\leq n/4\}\] and note that for all $u$, we have  
\begin{equation*} \P_M(\|Mu\|_2\leq 4\eps\sqrt{n})  = \P_M( \|Mv - Mr\|_2 \leq  4\eps\sqrt{n}) \geq \P_M( \|Mr\|_2 \leq 2\eps \sqrt{n} \text{ and } \cE )\,. 
\end{equation*}
Since by Bernstein inequality $\P(\|M\|_{\HS}^2\geq n^2/16)\leq 2\exp(-cn^2)$ and the fact that 
\[\eps\geq \exp(-2c_{\Sigma}n)\geq \exp(-cn)\] we have 
\[\P(\cE)\geq (2L\eps)^n-2\exp(-cn^2)\geq (1/2)(2L\eps)^n,\]
assuming that $c_{\Sigma}$ is chosen appropriately small compared to this absolute constant.
Thus
\begin{align*} \P_M(\|Mu\|_2\leq 4\eps\sqrt{n}) &\geq \P_M( \|Mr\|_2 \leq 2\eps \sqrt{n} \, \big| \cE ) \P( \cE ) \\
	&\geq  \left(1 - \P_M( \|Mr\|_2 > 2\eps \sqrt{n}\, \big| \cE )\right)(1/2)(2L\eps)^n \,. 
\end{align*}
Taking expectations gives
\begin{equation}\label{eq:netLemEE} 
	\E_{u}\P_M(\|Mu\|_2\leq 4\eps\sqrt{n}) \geq \left(1 - \E_u \P_M( \|Mr\|_2 > 2\eps \sqrt{n}\, \big\vert \cE ) \right)(1/2)(2L\eps)^n\,. \end{equation}
Exchanging the expectations and rearranging, we see that it is enough to show 
\begin{equation*} \E_M\left[ \P_r( \|Mr\|_2 > 2\eps \sqrt{n})\, \big\vert\, \cE \right] \leq 1/2\,. 
\end{equation*}
We will show that $\P_r( \|Mr\|_2 > 2\eps \sqrt{n}) \leq 1/4$ for all $M \in \cE$, by Markov's inequality. Note that 
\begin{equation*} 
	\E_r\, \|Mr\|_2^2 = \sum_{i,j} \E \left( M_{i,j}r_i \right)^2 = \sum_{i} \E\, r_i^2 \sum_{j} M_{i,j}^2 \leq 16\eps^2\|M\|_{\HS}^2/n\leq \eps^2 n,
 \end{equation*}
where for the second equality we have used that the $r_i$ are mutually independent and $\E\, r_i = 0$; for the third inequality, we used $\|r\|_\infty\leq 4\eps/\sqrt{n}$; and for the final inequality we used $\|M\|_{\HS}\leq n/4$. Thus by Markov's inequality gives
\begin{equation} \label{eq:NetlemMarkov} \P_{r}(  \|Mr\|_2 \geq 2\eps\sqrt{n}) \leq  (2\eps \sqrt{n})^{-2} \E_r\, \|Mr\|_2^2 \leq 1/4 \,. 
\end{equation}
Putting \eqref{eq:NetlemMarkov} together with \eqref{eq:netLemEE} proves \eqref{eq:lemnetEgoal}, completing the proof of \eqref{eq:lem-net-goal}.
\end{proof}

	\section{Proof of Lemma \ref{lem:HW}}
	\label{app:HW}
	
	We will derive Lemma \ref{lem:HW} from Talagrand's inequality:
	\begin{theorem}[Talagrand's Inequality]\label{thm:talagrand}
		Let $F:\R^n \to \R$ be a convex $1$-Lipschitz function and $\sigma = (\sigma_1,\ldots,\sigma_n)$ where the $\sigma_i$ are i.i.d.\ random variables such that $|\sigma_i|\leq 1$.  Then for any $t \geq 0$ we have 
		\begin{equation*}
			\P\left( \left| F(\sigma) - m_F \right| \geq t  \right) \leq 4 \exp\left(-t^2/16 \right)\, ,
		\end{equation*}
		 where $m_F$ is the median of $F(\sigma)$. 
	\end{theorem}

	\begin{proof}[Proof of Lemma \ref{lem:HW}]
		Note the theorem is trivial if $k \leq 2^{20} B^{4}/\nu$, so assume that $k > 2^{20} B^{4}/\nu$.
		Set $\sigma=2^{-4}B^{-2}\tau'$, define 
		\[ F(x) :=\|W\|^{-1}\|W^T x\|_2 \] and note that $F$ is convex and $1$-Lipschitz. Since $|\sigma_i|\leq 2^{-4}B^{-2}|\tau_i|\leq 1$ and the $\sigma_i$ are i.i.d., Theorem~\ref{thm:talagrand} tells
		us that $F(\sigma)$ is concentrated about the median $m_F$ and so we only need to estimate $m_F$. For this, write
		\begin{equation*} 
			m:= \E\, \|W^T \sigma\|_2^2 =\sum_{i,j}W_{ij}^2 \E\, \sigma_i^2 = \E \sigma_i^2 \|W\|_{\HS}^2,
		\end{equation*}
		and 
		\begin{equation*}
			m_2:= \E\, \|W^T \sigma\|_2^4-(\E\, \|W^T \sigma\|_2^2)^2 = \sum_{i,j}W_{ij}^2\big( \E\, \sigma_i^4 -(\E\, \sigma_i^2)^2\big)
		\leq \E\, \sigma_i^2 \|W\|_{\HS}^2,
	\end{equation*}
		where for the final inequality we used that $\E\, \sigma_i^4\leq \E\, \sigma_i^2$, since $|\sigma_i|\leq 1$.
		For $t>0$, Markov's inequality bounds
		\begin{equation*}
			\P(\|W^T \sigma\|_2^2\leq m-t)\leq t^{-2}\E\, \left( \|W^T \sigma\|_2^2-m \right)^2  = t^{-2}m_2 \leq t^{-2}\E\, \sigma^2_i \|W\|_{\HS}^2 .
		\end{equation*}
		Setting $t = \E\, \sigma_i^2\|W\|_{\HS}^2/2$ gives
		\begin{equation*}\P(\|W^T \sigma\|_2^2\leq \E\, \sigma_i^2\|W\|_{\HS}^2/2)\leq 4 (\E \sigma_i^2 \|W\|_{\HS}^2)^{-1}<1/2
		\end{equation*} since
		 $\E \sigma_i^2= 2^{-8}B^{-4}\E \tau_i'^2\geq 2^{-8}B^{-4} \nu$ and $\|W\|_{\HS}^2\geq k/4>2^{11}\nu^{-1}B^{4}$ (by assumption). It follows that 
		 \begin{equation*}
		 m_F\geq \sqrt{\E\, \sigma_i^2/2}\|W\|^{-1}\|W\|_{\HS}\geq 2^{-6}\|W\|^{-1}B^{-2}\sqrt{\nu k}\, ,
		 \end{equation*} since $\|W\|_{\HS}\geq \sqrt{k}/2$.  Now we may apply Talagrand's Inequality (Theorem  \ref{thm:talagrand}) with $t=m_F-\beta' \sqrt{k}\|W\|^{-1}$ to obtain 
		$$\P\left(\|W^T \sigma\|_2 \leq \beta'\sqrt{k} \right) \leq 4 \exp\left(-2^{-20}B^{-4}\nu k\right)$$ as desired.
	\end{proof}

\section{ Proof of Theorem \ref{thm:neg-dependence-2}}\label{app:neg-dep}
Here we deduce Theorem \ref{thm:neg-dependence-2}, which shows negative correlation between a small ball and large deviation event.  The proof is similar in theme to those in Section \ref{sec:decouplingQForms} but is, in fact, quite a bit simpler due to the fact we are working with a linear form rather than a quadratic form.

\begin{proof}[Proof of Theorem \ref{thm:neg-dependence-2}]
  We first write \begin{equation}\label{eq:app-esseen}
		\P(|\langle X, v \rangle| \leq \eps \text{ and } \langle X,u \rangle > t) \leq \E\left[\one\{|\langle X, v \rangle| \leq \eps\} e^{\lambda \langle X, u \rangle - \lambda t}    \right],\end{equation}
		where $\l \geq 0$ will be optimized later. Now apply Esseen's inequality in a similar way to Lemma \ref{lem:tilted-Esseen} to bound \begin{equation}\label{eq:app-compute-tilt}
		\E\left[\one\{|\langle X, v \rangle| \leq \eps\} e^{\lambda \langle X, u \rangle - \lambda t}    \right] \ls \eps e^{-\lambda t} \int_{-1/\eps}^{1/\eps} \left|\E e^{2\pi i \theta \langle X,v\rangle + \lambda \langle X, u\rangle} \right|\,d\theta\,.
	\end{equation}
	Applying Lemma \ref{lem:char-fcn-bound} bounds \begin{equation}\label{eq:app-LCD}
		\left|\E e^{2\pi i \theta \langle X,v\rangle + \lambda \langle X, u\rangle} \right| \ls \exp\left(-c \min_{r \in [1,c^{-1}]} \| \theta r v\|_{\T}^2 + c^{-1}\lambda^2 \right) + e^{-c\alpha n}\,.
	\end{equation}
	Combining the lines \eqref{eq:app-esseen},\eqref{eq:app-compute-tilt} and \eqref{eq:app-LCD} and choosing $C$ large enough gives the bound \begin{align*}\P(|\langle X, v \rangle| \leq \eps \text{ and } \langle X,u \rangle > t) &\ls \eps e^{-\lambda t + c^{-1}\lambda^2} \int_{-1/\eps}^{1/\eps} \left(e^{- c\gamma^2  \theta^2} + e^{-c\alpha n}\right)\,d\theta \\
		&\ls \eps e^{-\lambda t + c^{-1}\lambda^2} \gamma^{-1} + e^{-c\alpha n - \lambda t + c^{-1}\lambda^2}\,.
	\end{align*}
	Choosing $\lambda = ct/2$ completes the proof.
\end{proof}

\section{Proof of Lemma \ref{lem:V-compressible}}\label{app:compressible}
We deduce the second part of Lemma~\ref{lem:V-compressible} from the following special case of a Proposition of Vershynin \cite[Prop. 4.2]{vershynin-invertibility}.

\begin{prop}\label{pr:V-import}
	For $B >0$, let $\zeta \in \Gamma_B$, let $A_n \sim \Sym_{n}(\zeta)$ and let $K \geq 1$.  
Then there exist $\rho,\delta,c >0$ depending only on $K, B$ so that for every $\l \in \R$ and $w\in \R^n$ we have 
$$\P\big( \inf_{x \in \Comp(\delta,\rho)} \|(A_n + \lambda I)x-w \|_2 \leq c \sqrt{n} \text{ and } \|A_n + \lambda I\|_{op} \leq K \sqrt{n}\big) \leq 2 e^{-cn}\,.$$ 
\end{prop}

\begin{proof}[Proof of Lemma \ref{lem:V-compressible}]
To get the first conclusion of Lemma~\ref{lem:V-compressible} we may assume without loss of generality that $u\in \S^{n-1}$. So first let $\cN$ be a $c\sqrt{n}$-net for $[-4\sqrt{n},4\sqrt{n}]$, with $|\cN|\leq 8/c$. Note that $\Pr(\|A_n\|_{op} > 4 \sqrt{n})\ls e^{-\Omega(n)}$ so if $A_nx=tu$ then we may assume $t\in [-4\sqrt{n},4\sqrt{n}]$. So \begin{align*}
\P\big(\exists~x \in \Comp(\delta,\rho), \exists t \in [-4\sqrt{n},4\sqrt{n}]& : A_n x= tu \big)\\ \leq \sum_{t_0 \in \cN}\P&\left(\exists~x \in \Comp(\delta,\rho) : \|A_n x-t_0 u\|_2\leq c\sqrt{n} \right),
\end{align*} since for each $t\in[-4\sqrt{n},4\sqrt{n}]$ there's $t_0\in \cN$, such that if $A_n x=tu$ then $\|A_n x-t_0 u\|_2\leq c\sqrt{n}$. Now to bound each term in the sum take $\lambda = 0$, $K=4$, $w=t_0u$ in Proposition~\ref{pr:V-import} and notice we may assume $\|A_n\|_{op}\leq 4\sqrt{n}$ again. For the second conclusion, it is sufficient to show \begin{align}\label{eq:boundcompexp}
\begin{split}
		\P\big(\exists~x \in \Comp(\delta,\rho), \exists t \in [-4\sqrt{n},4\sqrt{n}] : \|(A_n - tI) x\|_2 = 0  \text{ and }\|A_n-tI\|_{op} &\leq 8\sqrt{n}\big)\\ \lesssim e^{-\Omega(n)}& \, ,
\end{split}
\end{align}
since we have $\P(\|A_n\|_{op} \geq 4\sqrt{n}) \ls e^{-\Omega(n)}$, by \eqref{eq:op-norm-bound}, so we may assume that all eigenvalues of $A_n$ lie in $[-4\sqrt{n},4\sqrt{n}]$ and $\|A_n-tI\|_{op}\leq |t|+\|A_n\|_{op}\leq 8\sqrt{n}$, for all $t\in [-4\sqrt{n},4\sqrt{n}]$.

For this, we apply Proposition \ref{pr:V-import} with $K = 8$ to obtain $\rho,\delta,c$. Again let $\cN$ be a $c\sqrt{n}$-net for the interval $[-4\sqrt{n},4\sqrt{n}]$ with $|\cN| \leq 8/c$. So, if $t \in [-4\sqrt{n},4\sqrt{n}]$ satisfies $A_nx = tx$ for some $x \in \S^{n-1}$, then there is a $t_0 \in \cN$ with $|t - t_0| \leq c\sqrt{n}$ and  $$\|(A_n - t_0I )x\|_2 \leq |t - t_0|\|x\|_2 \leq c\sqrt{n}\,.$$
Thus the left hand side of \eqref{eq:boundcompexp} s at most
\[\sum_{t_0 \in \cN} \P\left(\exists~x \in \Comp(\delta,\rho)  : \|(A_n - t_0I) x\|_2 \leq c\sqrt{n} \text{ and } \|A_n-t_0 I\|_{op}\leq 8\sqrt{n}\right) \ls e^{-cn}, \]
where the last line follows from Proposition \ref{pr:V-import}. 
\end{proof}

\end{appendices}
		\bibliographystyle{abbrv}
\bibliography{Bib-subG}

\end{document}